\documentclass[11pt]{article}
\usepackage{geometry}
\usepackage{lineno}
\usepackage{algorithm}
\usepackage{amsfonts}
\usepackage{amsmath}
\usepackage{amsthm}
\usepackage{amssymb}
\usepackage{graphicx} 
\usepackage{algorithmic}
\usepackage{enumitem}
\usepackage{textcomp}
\usepackage{tikz}

\usepackage{graphicx}
\usepackage{xcolor}
\usepackage{caption} 
\usepackage[colorlinks]{hyperref}

\newtheorem{theorem}{Theorem}[section]
   
\newtheorem{remark}{Remark}
\newtheorem{proposition}{Proposition}
\newtheorem{problem}{Problem}

\newtheorem{assumption}[theorem]{Assumption}
\newtheorem{lemma}[theorem]{Lemma}

\usepackage{bm}
\newcommand{\vect}[1]{\boldsymbol{\mathbf{#1}}}  
\renewcommand{\div}[1]{\operatorname{div}(#1)}
\DeclareMathOperator*{\argmin}{arg\,min} 
\newcommand{\dotu}{\dot{u}}
\newcommand{\vertiii}[1]{{\left\vert\kern-0.25ex\left\vert\kern-0.25ex\left\vert #1 
    \right\vert\kern-0.25ex\right\vert\kern-0.25ex\right\vert}}
\newcommand{\cprod}[2]{{(\kern-0.25ex( #1 , #2 )\kern-0.25ex)}}
\newcommand{\smallOad}{\mathcal{O}}
\newcommand{\Oad}{\mathcal{O}_{ad}}
\newcommand{\ADMMset}{\mathcal{A}}
\newcommand{\dcirc}{d_{\circ}}
\newcommand{\realu}{u_{1}}
\newcommand{\imagu}{u_{2}}
\newcommand{\realp}{p_{1}}
\newcommand{\imagp}{p_{2}}
\newcommand{\realq}{q_{1}}
\newcommand{\imagq}{q_{2}}
\newcommand{\realLambda}{\Lambda_{1}}
\newcommand{\imagLambda}{\Lambda_{2}}
\newcommand{\extu}{\tilde{u}} 
\newcommand{\supb}{b_{\infty}}

\newcommand{\sfTheta}{\vect{\Theta}}
\newcommand{\HH}{\vect{\mathsf{H}}}
\newcommand{\QQ}{\vect{\mathsf{Q}}}
\newcommand{\LL}{\vect{\mathsf{L}}}

\newcommand{\VV}{\vect{\theta}}
\newcommand{\divbb}{\operatorname{div}{\vect{b}}}
\newcommand{\divVV}{\operatorname{div}{\VV}}
\newcommand{\divWW}{\operatorname{div}{\WW}}
\newcommand{\spaceV}{{\vect{\mathsf{V}}}}
\newcommand{\Vn}{\theta_{n}} 
\newcommand{\WW}{\vect{\tilde{\theta}}}
\newcommand{\Wn}{\tilde{\theta}_{n}}

\newcommand{\nn}{\vect{n}}

\newcommand{\dn}[1]{\partial_{\nn}{#1}}

\newcommand{\intO}[1]{\int_{\Omega}{#1}{\, {d} x}}

\newcommand{\intOstar}[1]{\int_{\Omega^{\star}}{#1}{\, {d} x}}
\newcommand{\intS}[1]{\int_{\Sigma}{#1}{\, {d} s}}  
\newcommand{\intG}[1]{\int_{\Gamma}{#1}{\, {d} s}} 
\newcommand{\intD}[1]{\int_{D}{#1}{\, {d} x}} 
\newcommand{\chio}{\chi_{\omega^{\textsf{c}}}}
\newcommand{\chion}{\chi_{\omega_{n}^{\textsf{c}}}}
\newcommand{\intDO}[1]{\int_{\omega^{\textsf{c}}}{#1}{\, {d} x}}
\newcommand{\intDOn}[1]{\int_{\omega_{n}^{\textsf{c}}}{#1}{\, {d} x}}
\newcommand{\intOn}[1]{\int_{\Omega_{n}}{#1}{\, {d} x}}

\newcommand{\intGstar}[1]{\int_{\Gamma^{\star}}{#1}{\, {d} s}}

\newcommand{\norm}[1]{\left\|{#1}\right\|} 
\newcommand{\abs}[1]{\left|{#1}\right|}

\title{Enhanced shape recovery in advection--diffusion problems via a novel ADMM-based CCBM optimization}
\author{Elmehdi Cherrat$^{\ast, 1}$ $\cdot$ Lekbir Afraites$^{\ast, 2}$ $\cdot$ Julius Fergy Tiongson Rabago$^{\dagger}$}

\date{%
	{\footnotesize
        $^\ast$Mathematics and Interactions Teams (EMI)\\%
        Faculty of Sciences and Techniques\\%
        Sultan Moulay Slimane University\\%
        Beni Mellal, Morocco\\\vspace{-2pt}
       	\texttt{${}^{1}$cherrat.elmehdi@gmail.com, \, ${}^{2}$l.afraites@usms.ma}\\[2ex]
	$^{\dagger}$Faculty of Mathematics and Physics\\%
	 Institute of Science and Engineering\\%
         Kanazawa University, Kanazawa 920-1192, Japan\\\vspace{-2pt}
        \texttt{jftrabago@gmail.com}}\\[2ex]           
    \today 
}

\begin{document} 

\maketitle

\begin{abstract}
This work proposes a novel shape optimization framework for geometric inverse problems governed by the advection--diffusion equation, based on the coupled complex boundary method (CCBM). 
Building on recent developments \cite{Afraites2022, Rabago2023b, Rabago2025, RabagoAfraitesNotsu2025, RabagoNotsu2024}, we aim to recover the shape of an unknown inclusion via shape optimization driven by a cost functional constructed from the imaginary part of the complex-valued state variable over the entire domain.  
We rigorously derive the associated shape derivative in variational form and provide explicit expressions for the gradient and second-order information.  
Optimization is carried out using a Sobolev gradient method within a finite element framework.  
To address difficulties in reconstructing obstacles with concave boundaries, particularly under measurement noise and the combined effects of advection and diffusion, we introduce a \textit{state-of-the-art} numerical scheme inspired by the Alternating Direction Method of Multipliers (ADMM).  
In addition to implementing this non-conventional approach, we demonstrate how the adjoint method can be efficiently applied and utilize \textit{partial} gradients to develop a more efficient CCBM--ADMM scheme.  
The accuracy and robustness of the proposed computational approach are validated through various numerical experiments.

\medskip

\textit{Keywords:}{ Alternating direction method of multipliers, advection-diffusion model, coupled complex boundary method, geometric inverse problem, shape optimization, shape derivatives, adjoint method}

\medskip
\textit{MSC:}{ 49Q10, 49K20, 65K10}
\end{abstract}

\section{Introduction}
\label{sec:Introduction}
The advection--diffusion equation is a fundamental model in applied mathematics, describing the combined effects of transport and dispersion of substances within a medium \cite{Okubo1980}.  
It arises in key environmental applications, such as the spread of pollutants in air and water and the transport of solutes in rivers and groundwater.  
Accurate modeling of these processes is essential for pollution control, water resource management, and related environmental systems \cite{ZoppouKnight1999}.  
Because of its central role in predicting and managing transport phenomena, the advection--diffusion framework has been widely studied in both theory and applications.  

In recent years, significant attention has been devoted to inverse problems for advection--diffusion systems, particularly for identifying unknown shapes or sources from indirect measurements.  
Many works have focused on point source identification \cite{AndrleBadia2002,BadiaDuong2002,BadiaDuongHamdi2005,Hamdi2012,HamdiMahfoudhi2013,Neelz2006}, while others have investigated shape reconstruction using adjoint-based methods \cite{YanHouGao2017}, domain derivative techniques \cite{YanSuJing2014}, and topological derivatives \cite{Fernandezetal2021}.  
More recently, we have studied geometric inverse problems for simplified advection--diffusion systems using conventional numerical methods \cite{CherratAfraitesRabago2025b,CherratAfraitesRabago2025}.  
Building on this foundation, the present work considers a more general setting with spatially varying coefficients, and introduces a reconstruction strategy based on a non-conventional numerical shape optimization framework combined with an augmented Lagrangian method.  
\begin{figure}[htbp]
\centering
\begin{tikzpicture}[scale=0.8]
    \draw[very thick, draw=black, fill=gray!1] (0,0) ellipse (3.5cm and 2.2cm);
    \node at (0.0,-1.6) {\large $\varOmega = D \setminus \overline{\omega}$};

    \draw[dashed, very thick, draw=blue, fill=yellow!5] plot [smooth cycle] coordinates {(-1.2,0.7) (0.3,0.9) (1.1,-0.4) (-0.2,-1.1) (-1.3,-0.5)};
    \node at (-0.1,0) {\large  $\omega$};
    
   \node[color=blue] at (1.7,0.7) {\large  $\varGamma = \partial \omega$};
    
    \node[color=black]  at (4.5,0.25) {\large  $\varSigma=\partial D$};

    \node at (-1.5,1.25) {\large $D$};
\end{tikzpicture}
\caption{{Conceptual model}}
\label{fig:conceptual_model}
\end{figure}

The objective of this paper is to recover both the solution $u$ and an open, connected, bounded region $\varOmega \subset \mathbb{R}^d$ in an advection--diffusion problem.
The domain has an annular structure, $\varOmega = D \setminus \overline{\omega}$, {where the outer boundary $\varSigma := \partial D$} is known and accessible for measurements, while the inner boundary $\varGamma := \partial \omega$ is unknown and must be reconstructed.
Here, $D \subset \mathbb{R}^d$ represents the flow domain, and $\omega \Subset D$ is an inclusion modeling {an inert obstacle or a perfectly absorbing sink. Figure~\ref{fig:conceptual_model} shows a conceptual model.}
The reconstruction is based on a single pair of Cauchy data $(f,g)$ prescribed on $\varSigma$, with a homogeneous Dirichlet condition imposed on the unknown boundary $\varGamma$.
The corresponding mathematical model is the following overdetermined boundary value problem:
\begin{equation}\label{eq:surdeterminee}
-\div{\sigma \nabla u} + \vect{b} \cdot \nabla u = 0 \ \text{in } \varOmega, \quad
u = 0 \ \text{on } \varGamma, \quad
u = f, \;\; \sigma \dn{u} = g \ \text{on } \varSigma,
\end{equation}
where, in the context of pollutant transport in water, $\sigma := \sigma(x)$ denotes the diffusion coefficient, $u := u(x)$ the contaminant concentration, $\vect{b} := \vect{b}(x)$ the water flow velocity, and $\dn{u}$ denotes the outward normal derivative of $u$ on $\varSigma$.

Consequently, we consider the inverse geometry problem:
\begin{equation}\label{eq:inverse_problem}
	\text{Find } \varGamma \text{ and } u \text{ that satisfy the overdetermined system } \eqref{eq:surdeterminee}.
\end{equation}
We solve \eqref{eq:inverse_problem} using shape optimization \cite{DelfourZolesio2011,HenrotPierre2018,SokolowskiZolesio1992}, a well-established framework for free boundary problems \cite{EpplerHarbrecht2012b} and shape identifications \cite{CaubetDambrineKateb2013,CaubetDambrineKatebTimimoun2013,EpplerHarbrecht2005,HarbrechtTausch2011,HarbrechtTausch2013}.  
For overdetermined systems such as \eqref{eq:surdeterminee}, one typically prescribes a boundary condition on the free boundary to ensure a well-posed state equation, while the remaining data are identified via least-squares minimization \cite{CherratAfraitesRabago2025b,CherratAfraitesRabago2025}.

Here, we propose a non-conventional shape optimization method based on the \textit{coupled complex boundary method} (CCBM) \cite{Afraites2022,Chengetal2014,Rabago2023b,RabagoAfraitesNotsu2025,RabagoNotsu2024}.  
The CCBM encodes Neumann data as the real part and Dirichlet data as the imaginary part of a Robin-type complex boundary condition.  
The inverse problem is then solved by minimizing the $L^2$ norm of the imaginary part of the solution over the domain, which distinguishes the CCBM from conventional methods that impose misfit penalties directly on the boundary \cite{AfraitesDambrineEpplerKateb2007,AfraitesDambrineKateb2007,AfraitesDambrineKateb2008,AfraitesRabago2024,CherratAfraitesRabago2025b,CherratAfraitesRabago2025}.  

The main novelty of this study lies in extending the CCBM framework to advection--diffusion systems with spatially varying coefficients $\sigma$ and $\vect{b}$.  
This setting provides a richer and more realistic description of transport phenomena compared to constant-coefficient models.  
In addition, we consider complex geometries, including nonconvex shapes, beyond the smooth convex cases commonly addressed in earlier works \cite{YanHouGao2017,YanSuJing2014}.  
Our formulation therefore broadens the applicability of these non-conventional shape reconstruction techniques to more general scenarios relevant to practical modeling.

To improve the recovery of complex-shaped obstacles under noisy data, we embed the CCBM within an augmented Lagrangian framework based on the Alternating Direction Method of Multipliers (ADMM) \cite{CherratAfraitesRabago2025b,RabagoHadriAfraitesHendyZaky2024}.  
This provides an alternative to conventional regularization techniques such as perimeter or surface penalization.  
Note that existing ADMM-based shape optimization methods use real-valued auxiliary variables. 
In our case, the state variable is complex-valued, so the auxiliary variable must be redefined consistently and adapted to the CCBM structure.  
This results in a state-of-the-art ADMM algorithm specifically designed for CCBM.
In addition, we propose a partial gradient scheme that further improves reconstruction accuracy while reducing computational cost.  
Altogether, the resulting CCBM--ADMM framework provides a unified and efficient optimization strategy, whose effectiveness is demonstrated through numerical experiments, with particular emphasis on three-dimensional configurations.

The paper is organized as follows. Section~\ref{sec:problem_setting} formulates the shape optimization problem and the CCBM framework. Section~\ref{sec:Shape_sensitivity_analysis} carries out the shape sensitivity analysis. Section~\ref{sec:numerical_algorithms_and_examples} presents a gradient-based algorithm and numerical experiments: two-dimensional examples illustrate the conventional CCBM, followed by three-dimensional demonstrations of the ADMM approach for spatially varying coefficients. Section~\ref{sec:conclusion} concludes the paper.
\section{The problem setting}
\label{sec:problem_setting}
Throughout the article, we use the standard Lebesgue and Sobolev spaces for real-valued functions.
For complex-valued functions, we use $\HH^m(\varOmega)$, with the inner product 
$\cprod{u}{v}_{m,\varOmega} = \sum_{|\alpha| \leqslant m} \intO{D^{\alpha} u\, \overline{D^{\alpha} v}}$ 
and the associated norm $\vertiii{u}_{\HH^m(\varOmega)}^{2} = \cprod{u}{u}_{m,\varOmega}$.
We also introduce the subspaces 
$V(\varOmega) := \left\lbrace \varphi \in H^{1}(\varOmega) \ \big|\ \varphi|_{\varGamma} = 0 \right\rbrace$ 
and 
$\spaceV(\varOmega) := \left\lbrace \varphi \in \HH^{1}(\varOmega) \ \big|\ \varphi = 0 \ \text{on}\ \varGamma \right\rbrace$.
Moreover, we define $Q = L^{2}(\varOmega)$ and $\QQ = \LL^{2}(\varOmega)$.
We equip the space $\spaceV(\varOmega)$ with the norm $\vertiii{\varphi}_{\spaceV(\varOmega)} = \sqrt{\vertiii{\nabla{\varphi}}_{\QQ^{d}} }$.
Throughout the paper, $c>0$ denotes a generic positive constant, whose value may change from line to line.
\subsection{Key assumptions and the inverse geometry problem}
\begin{assumption}\label{assumption:key_conditions}
We consider Equation~\eqref{eq:surdeterminee} under the following assumptions: (i) $\sigma \in W^{1,\infty}(D)^{d\times d}$ and there is $\sigma_{0} \in \mathbb{R}_{+}$ such that $\sum_{i,j=1}^d \sigma_{ij} \xi_i \xi_j \geqslant \sigma_{0} \|\xi\|_d^{2}, \text{ for any } \xi \in \mathbb{R}^d$; (ii) $\vect{b} \in W^{1,\infty}(D)^d$; (iii) $f \in H^{3/2}(\varSigma)$ with $f \not\equiv 0$, and $g \in H^{1/2}(\varSigma)$ is the corresponding compatible boundary measurement; and (iv) there exists a constant $c> 0$, independent of $\Omega$, such that $\abs{\vect{b}}_{L^{\infty}(D)^{d}} < c \, \sigma_0$.
\end{assumption} 

To simplify the analysis and notation, we assume that $\sigma \in W^{1,\infty}(D)^{1 \times 1}$, i.e., $\sigma$ is scalar. We denote $\supb := \lvert \mathbf{b} \rvert_{L^{\infty}(D)^d}$, and we work throughout the paper under Assumption~\ref{assumption:key_conditions}, unless stated otherwise.

We fix some notations used throughout this paper.  
Let $D \subset \mathbb{R}^{d}$, $d \in \{2,3\}$, be a bounded open set and $\dcirc>0$.  
We define an open set $D_{\dcirc} \Subset D$ with ${C}^{\infty}$ boundary satisfying
$\{x \in D : d(x,\partial D) > \dcirc/2\} \subset D_{\dcirc} \subset \{x \in D : d(x,\partial D) > \dcirc/3\}$.  
In fact, we assume without further notice that the inclusion $\emptyset \neq \omega \subset D$ satisfies $\underline{d} \leqslant \operatorname{diam}(\omega) \leqslant \overline{d} < \operatorname{diam}(D)$, where $\underline{d}, \overline{d}>0$ are fixed positive constants.
These conditions are required for the existence proof of an optimal shape in the proposed optimization problem; see Subsection~\ref{subsec:existence_of_optimal_solution}. This amounts to a technical assumption ensuring that a uniform Poincar\'e--Friedrichs inequality holds (cf. \cite{BoulkhemairChakib2007}).

We then define the set of admissible obstacles as
\begin{equation}\label{eq:set_of_admissible_obstacles}
\smallOad := \left\{ \omega \Subset D_{\dcirc} \;\left|\;
\begin{aligned}
& \emptyset \neq \omega \text{ is ${C}^{1,1}$, $D \setminus \overline{\omega}$ is connected, }\\
&\text{and} \ \underline{d} \leqslant \operatorname{diam}(\omega) \leqslant \overline{d} < \operatorname{diam}(D)
\end{aligned} \right. \right\}.
\end{equation}

A domain $\varOmega \subset \mathbb{R}^{d}$ is said to be \emph{admissible} if  
$\varOmega = D \setminus \overline{\omega}$ for some $\omega \in \smallOad$, in which case we write $\varOmega \in \Oad$.  
We assume the existence of $\omega^\star \in \smallOad$ (equivalently $\varOmega^\star \in \Oad$) such that the overdetermined system~\eqref{eq:surdeterminee} admits a solution.  
Hence, the inverse geometry equation~\eqref{eq:inverse_problem} becomes
\begin{equation}\label{eq:shape_problem}
    \text{Find $\omega \in \smallOad$ and $u$ such that \eqref{eq:surdeterminee} holds.}
\end{equation}
{ All annular domains $\varOmega$ are assumed admissible unless stated otherwise.}

\subsection{The coupled complex boundary value problem}\label{subsec:formulations}  
To address the inverse geometry equation~\eqref{eq:shape_problem}, we adopt a non-conventional shape optimization approach based on the \textit{coupled complex boundary method} (CCBM) \cite{Chengetal2014}.  
Unlike traditional least-squares methods that fit boundary data directly and may suffer from numerical instabilities, CCBM relocates the data fitting to the domain interior.  
This naturally leads to a cost functional defined as a domain integral, which improves stability and robustness in the reconstruction process, in a manner comparable to the Kohn--Vogelius method \cite{KohnVogelius1987}.

The main idea of CCBM is to encode the measured Dirichlet and Neumann data into a complex-valued Robin boundary condition, with the real and imaginary parts representing these measurements.  
This formulation regularizes the inverse problem, providing stable and accurate reconstructions even in the presence of noise, while increasing the size of the resulting system.  
CCBM has been successfully applied to various inverse problems, including conductivity reconstruction, parameter identification in elliptic equations, free boundary problems, and tumor localization~\cite{Afraites2022,Chengetal2014,Gongetal2017,Ouiassaetal2022,Rabago2023b,Rabago2025,RabagoAfraitesNotsu2025,RabagoNotsu2024,ZhengChengGong2020}.

To implement CCBM, system~\eqref{eq:surdeterminee} is reformulated as the following complex-valued boundary value problem:

{
\begin{equation}\label{eq:complex_PDE}
\left\{
\begin{aligned}
    -\div{\sigma \nabla u} + \mathbf{b} \cdot \nabla u &= 0 \quad &&\text{in } \varOmega, \\
    u &= 0 \quad &&\text{on } \varGamma, \\
    \sigma \dn{u} + i u &= g + i f \quad &&\text{on } \varSigma.
\end{aligned}
\right.
\end{equation}
}where $u : \varOmega \to \mathbb{C}$ and $i = \sqrt{-1}$. 
Hereinafter, unless otherwise specified, $u = u(\varOmega)$ will always denote the solution of \eqref{eq:complex_PDE}.

Let the sesquilinear form $a: \spaceV(\varOmega) \times \spaceV(\varOmega) \to \mathbb{C}$ and the linear form $l: \spaceV(\varOmega) \to \mathbb{C}$ be
\begin{align}
    a(\varphi,\psi) & = \intO{ \sigma \nabla \varphi \cdot \nabla \overline{\psi} + (\vect{b} \cdot \nabla \varphi) \overline{\psi} } 
    + i \intS{\varphi \overline{\psi}}, \label{eq:sesquilinear_form_a}\\
    l(\psi) & = \intS{(g + i f) \overline{\psi}}, \quad \forall \varphi, \psi \in \spaceV(\varOmega).\label{eq:linear_form_l}
\end{align}
The weak formulation of \eqref{eq:complex_PDE} reads:
\begin{equation}\label{eq:weak_form_of_state}
    \text{Find $u \in \spaceV(\varOmega)$ such that $a(u, \psi) = l(\psi)$ for all $\psi \in \spaceV(\varOmega)$.}
\end{equation}

Existence and uniqueness of the solution $u$ follow from the complex Lax--Milgram lemma
\cite[p.~368]{DautrayLionsv31999}, since the bilinear form $a(\cdot,\cdot)$ is coercive in the sense that
$\Re\bigl(a(\varphi,\varphi)\bigr)\ge(\sigma_{0}-c\,\supb)\,
\vertiii{\varphi}_{\spaceV(\varOmega)}^{2}$,
with a constant $c>0$ that may depend on fixed parameters but is independent of $\varOmega$.
In addition, for data $(f,g)\in H^{3/2}(\varSigma)\times H^{1/2}(\varSigma)$,
the solution satisfies $u\in \HH^{2}(\varOmega)\cap \spaceV(\varOmega)$.

Let us write $u = \realu + i \imagu := \Re{\{u\}} + i \Im{\{u\}} $, where $\realu, \imagu : \varOmega \to \mathbb{R}$.
Then, we can split the complex PDE \eqref{eq:complex_PDE} into the coupled system of real PDEs

\begin{equation}\label{eq:real_part}
\left\{\arraycolsep=3pt\def\arraystretch{1}
    \begin{array}{rcll}
    -\div{\sigma\nabla \realu}+\vect{b} \cdot \nabla \realu & = & 0 & \text{in $\varOmega$},\\
    \realu & = & 0 & \text{on $\varGamma$},\\
    \sigma\dn{\realu} -\imagu & = & g& \text{on $\varSigma$},
    \end{array}
\right.\tag{Re}
\end{equation}
\begin{equation}\label{eq:imaginary_part}
\left\{\arraycolsep=3pt\def\arraystretch{1}
    \begin{array}{rcll}
    -\div{\sigma\nabla \imagu}+\vect{b} \cdot \nabla \imagu & = & 0 & \text{in $\varOmega$},\\
    \imagu & = & 0 & \text{on $\varGamma$},\\
    \sigma\dn{\imagu} +\realu & = & f & \text{on $\varSigma$}.
    \end{array}
\right.\tag{Im}
\end{equation}

Observe from \eqref{eq:real_part} and \eqref{eq:imaginary_part} that if $\imagu = 0$ in $\varOmega$, then $\imagu = \dn{\imagu} = 0$ on $\varSigma$, $\realu = u = 0$ on $\varGamma$, and $\realu = u = f$, $\dn{\realu} = \dn{u} = g$ on $\varSigma$. Hence, $(\varOmega, \realu)$ solves the overdetermined system~\eqref{eq:surdeterminee}. Conversely, any $(\varOmega, u)$ satisfying \eqref{eq:surdeterminee} ensures that $\realu$ and $\imagu$ satisfy \eqref{eq:real_part} and \eqref{eq:imaginary_part}.  

Consequently, problem \eqref{eq:inverse_problem} can be reformulated as follows:

\begin{problem}\label{prob:inverse}
Find $(\varOmega, u(\varOmega)) \in \Oad \times \spaceV(\varOmega)$ such that $\imagu = \Im\{u\} = 0$ in $\varOmega$.
\end{problem}

To solve Problem~\ref{prob:inverse}, we introduce the cost functional $J(\varOmega)$, defined by
\begin{equation}\label{eq:cost_functional}
    J(\varOmega) := \frac{1}{2} \intO{\abs{\imagu}^{2}}.
\end{equation}

For each fixed $\varOmega \in \Oad$, there exists a unique weak solution $u(\varOmega)$ to \eqref{eq:complex_PDE}, so the mapping $\varOmega \mapsto u(\varOmega) \in \spaceV(\varOmega)$ is well-defined. 
Hereinafter, \eqref{eq:complex_PDE} will be referred to as the state equation, and its solution simply as the state.

We can now state the shape identification problem:
\begin{problem}\label{prob:shape_optimization}
	Find $\varOmega^{\star} \in \Oad$ such that $J(\varOmega^{\star}) = \min_{\varOmega \in \Oad} J(\varOmega)$.
\end{problem}

\subsection{Existence of shape solution}\label{subsec:existence_of_optimal_solution}
It can be verified that if Problem~\ref{prob:inverse} has a solution $(u, \tilde{\varOmega})$, such that $\imagu = 0$ in $\tilde{\varOmega} \in \Oad$, then Problem~\ref{prob:inverse} is equivalent to Problem~\ref{prob:shape_optimization}.
Accordingly, we assert that Problem~\ref{prob:shape_optimization} admits a solution:
\begin{proposition}\label{prop:existence_of_optimal_solution}
	Problem~\ref{prob:shape_optimization} admits a solution in $\Oad$.
\end{proposition}
To prove Proposition~\ref{prop:existence_of_optimal_solution}, we require certain properties of the admissible domains. 
A sufficient condition is uniform regularity (see \cite{Chenais1975} or \cite[Chap.~5.6.4]{DelfourZolesio2011}), known as the $\epsilon$-property \cite{AfraitesRabago2024}, which holds since $\varOmega \in \Oad \subset C^{1,1}$ \cite[Thm.~2.4.7, p.~56]{HenrotPierre2018}.
Equipping $\Oad$ with the Hausdorff distance \cite[Def.~2.2.7, p.~30]{HenrotPierre2018}, existence of an optimal solution follows if: (i) $\Oad$ is compact; (ii) $\varOmega_n \to \varOmega$ in the Hausdorff sense implies $u(\varOmega_n) \to u(\varOmega)$ \cite[Def.~2.2.8, p.~30]{HenrotPierre2018}; and (iii) under (i) and (ii), $J(\varOmega_n, u(\varOmega_n)) \to J(\varOmega, u(\varOmega))$. 
The proof then proceeds based on these arguments.

On $\varGamma$, recall that $u = 0$. 
We define $\extu \in \HH^{1}_{\omega}(D) := \{\psi \in \HH^{1}(D) \mid \psi|_\omega = 0\}$ as the $\HH^1$-smooth extension of $u \in \spaceV(\varOmega)$ by zero in $D$ \cite{AdamsFournier2003,HenrotPierre2018}:
\begin{equation}\label{eq:zero_extension}
   \extu := \mathcal{E}u : \HH^{1}(\varOmega) \to \HH^{1}_{\omega}(D), \qquad
    \mathcal{E}u(x) =
        \begin{cases}
            u(x), & x \in \varOmega = D \setminus \overline{\omega}, \\
            0, & x \in \overline{\omega}.
        \end{cases}
\end{equation} 
The extension operator $\mathcal{E}$ is linear and bounded.  
Endowing $\HH^{1}_{\omega}(D)$ with the norm $\vertiii{\cdot}_{\HH^{1}_{\omega}(D)} := \vertiii{\nabla \cdot}_{\LL^2(D)}$, it can be shown that this norm is equivalent to the usual $\HH^1$-norm: $\vertiii{\varphi}_{\HH^{1}_{\omega}(D)} \sim \vertiii{\varphi}_{\HH^1(D)}$ for all $\varphi \in \HH^{1}_{\omega}(D)$.

Using the extension \eqref{eq:zero_extension}, the weak form \eqref{eq:weak_form_of_state} becomes
\begin{equation}\label{eq:extended_weak_form_of_state}
\tilde{a}(\extu, \psi) = \intD{\sigma \nabla \extu \cdot \nabla \overline{\psi}}
+ \intD{(\vect{b} \cdot \nabla \extu) \overline{\psi}}
+ i \intS{\extu \overline{\psi}}
= \intS{(g+if) \overline{\psi}},
\end{equation}
valid for all $\psi \in \HH^{1}_{\omega}(D)$, with $\extu \in \HH^{1}_{\omega}(D)$. Under Assumption~\ref{assumption:key_conditions}(iv), well-posedness follows from the complex-valued Lax-Milgram Theorem.

Lemma~\ref{lem:uniform_boundedness_of_state} below establishes that $\extu$ is bounded in $\HH^1(D)$. This is used in Proposition~\ref{prop:strong_convergence_of_extensions} to show that $u(\varOmega_n) \to u(\varOmega)$ as $\varOmega_n \xrightarrow{H} \varOmega$ in the Hausdorff sense. 
Proposition~\ref{prop:continuity_the_cost_function} then ensures the continuity of the cost functional, i.e., $\lim_{n \to \infty} J(\varOmega_n, u(\varOmega_n)) = J(\varOmega, u(\varOmega))$.
For notational convenience, we set $u_n := u(\varOmega_n)$ and let $\extu_n \in \HH^1(D)$ denote its zero extension to $D$, for all $n \in \mathbb{N} \cup \{0\}$.
\begin{lemma}\label{lem:uniform_boundedness_of_state}
Under Assumption~\ref{assumption:key_conditions}(iv), for any $(\varOmega, u) \in \Oad \times \spaceV(\varOmega)$, there exists a constant $c^{\star} > 0$ such that $\vertiii{\extu}_{\HH^1(D)} \leqslant c^{\star}$, where $\extu$ is the zero extension of $u$ satisfying \eqref{eq:extended_weak_form_of_state}.
\end{lemma}

\begin{proof}
We denote $\norm{(f,g)}:=\norm{g}_{H^{-1/2}(\varSigma)} + \norm{f}_{H^{1/2}(\varSigma)}$.
Then, by taking $\psi = \extu \in \HH^{1}_{\omega}(D)$ in \eqref{eq:extended_weak_form_of_state}, and using the continuous dependence of the state on the data, we obtain
$\tilde{a}(\extu,\extu) \leqslant \norm{(f,g)} \vertiii{\extu}_{\HH^{1/2}(\varSigma)}$.
Since $\extu \in \HH^{1}_{\omega}(D)$ (i.e., $\extu = 0$ in $\overline{\omega}$), the generalized Poincar\'{e}-Friedrichs inequality \cite{BoulkhemairChakib2007} gives 
$|\tilde{a}(\extu,\extu)| \geqslant \Re\{\tilde{a}(\extu,\extu)\} 
\geqslant (\sigma_0 - c \, \supb) \, \vertiii{\extu}_{\HH^1_{\omega}(D)}^2$, where $c > 0$ may depend on $D$, $\dcirc$, $\underline{d}$ and $\overline{d}$, but is independent of $\Omega$.
Then, choosing $c = c_P^{-1}$ in Assumption~\ref{assumption:key_conditions}(iv), applying the Sobolev embedding $\vertiii{\cdot}_{\HH^{s-1/2}(\varSigma)} \leqslant c_s \, \vertiii{\cdot}_{\HH^s(D)}$ for $s > 1/2$, and using the equivalence $\vertiii{\extu}_{\HH^1_{\omega}(D)} \sim \vertiii{\extu}_{\HH^1(D)}$, we obtain $\vertiii{\extu}_{\HH^1(D)} \leqslant c^\star$, where
$c^\star := c_s (\sigma_0 - c_P \supb)^{-1}\norm{(f,g)}$.
This concludes the proof.
\end{proof}
%
%


\begin{remark}
To avoid excessively wild or oscillating boundaries---which could cause constants such as the Poincar\'{e} constant to depend on $n$ (the index of a sequence) and compromise convergence in the space of admissible domains---it is natural to assume, as we do implicitly here, that the free boundaries are parametrized by uniformly Lipschitz functions. This assumption guarantees that the domains satisfy the $\varepsilon$-cone property, which is fundamental for ensuring convergence of domains; see \cite{Chenais1975,HenrotPierre2018} for a more rigorous discussion.
\end{remark}

\begin{proposition}\label{prop:strong_convergence_of_extensions}
Let $\{\varOmega_{n}\}_{n\in\mathbb{N}} \subset \Oad$ and $\varOmega \in \Oad$ be such that $\varOmega_{n} \xrightarrow{H} \varOmega$ (i.e., \textit{$\Oad$ is compact in the Hausdorff metric}).  
Then there exists a subsequence $\{\extu_{n_{k}}\}$ of $\{\extu_{n}\}$, where each $\extu_{n}$ solves \eqref{eq:extended_weak_form_of_state} in ${\HH_{\omega_{n}}^{1}(D)}$, for all $\psi \in {\HH_{\omega_{n}}^{1}(D)}$, and a function $u^{\star} \in \HH^{1}_{\omega}(D)$ such that:
    \begin{itemize}
    \renewcommand{\labelitemi}{$\bullet$}
        \item $\extu_{n_{k}} \rightharpoonup u^{\star}$ in $\HH^{1}(D)$,
        \item $u^{\star} \in \HH^{1}_{\omega}(D)$ satisfies \eqref{eq:extended_weak_form_of_state}, for all $\psi \in \HH^{1}_{\omega}(D)$,
        \item $u^{\star}|_{\varOmega} \in {\spaceV}(\varOmega)$ uniquely solves \eqref{eq:weak_form_of_state}, and
        \item $\extu_{n_{k}} \rightarrow u^{\star}$ in $\HH^{1}(D)$.
    \end{itemize}
\end{proposition}
\begin{proof}
Let $\{\varOmega_{n}\}_{n\in\mathbb{N}} \subset \Oad$ and $\varOmega \in \Oad$ be such that $\varOmega_{n} \xrightarrow{H} \varOmega$ (i.e., \textit{$\Oad$ is compact in the Hausdorff metric}).
Lemma~\ref{lem:uniform_boundedness_of_state} implies that $\vertiii{\extu_{n}}_{\HH^{1}(D)} \leqslant c^{\star}$, hence the existence of a sequence $\{\extu_{n}\}_{n\in\mathbb{N}}$ uniformly bounded in $\HH^{1}(D)$.
Consequently, we can extract a subsequence denoted by $\{\extu_{k}\}_{k\in\mathbb{N}}=\{\extu_{n_{k}}\}_{k\in\mathbb{N}} \subset \{\extu_{n}\}_{n\in\mathbb{N}}$ such that $\extu_{n_{k}} \rightharpoonup u^{\star}$ in $\HH^{1}(D)$.
We let $\psi \in \spaceV(\varOmega)$ and, by \eqref{eq:zero_extension}, denotes its zero extension in $D$ by $\tilde{\psi}$.
We construct, based on \eqref{eq:zero_extension}, a sequence $\{\psi_{k}\}_{k \in \mathbb{N}} \subset {\HH_{\omega_{k}}^{1}(D)}$, $\{\omega_{k}\}_{k \in \mathbb{N}} = \{\omega_{n_{k}}\}_{k \in \mathbb{N}} \in \smallOad$, such that $\psi_{k}  \xrightarrow{k \to \infty} \tilde{\psi}$ in $\HH^{1}_{\omega}(D)$. 

Because $\varOmega_{n} \xrightarrow{H} \varOmega$, then there exists an index $p\in \mathbb{N}$ such that, and for all $ n\geqslant p$, we have
\[
    \intOn{\sigma\nabla {u}_{n} \cdot \nabla \overline{\psi}_{k}}
    	+ \intOn{(\vect{b} \cdot \nabla {u}_{n}) \overline{\psi}_{k}}
    	+ i\intS{{u}_{n} \overline{\psi}_{k}}
    =\intS{{(g+if) \overline{\psi}_{k}}},
\]
for all $\psi_{k}|_{\varOmega_{n}} \in \HH^{1}({\varOmega_{n}})$.
For each $n \in \mathbb{N}$, we can extend ${u}_{n} \in \HH^{1}(\varOmega_{n})$ to $D$ using \eqref{eq:zero_extension} to obtain $\extu_{n} \in \HH^{1}_{\omega_{n}}(D)$ satisfying the variational equation $\tilde{a}(\extu_{n}, \psi_{k}) = \intS{{(g+if) \overline{\psi}_{k}}}$, for all ${\psi}_{k} \in \HH^{1}_{\omega_{n}}(D)$.
Taking another subsequence if necessary, we know that $u_{n} \rightharpoonup u^{\star}$ in $\HH^{1}_{\omega}(D)$, and so $\nabla{u}_{n} \rightarrow \nabla{u}^{\star}$ in $\HH^{1}_{\omega}(D)$.
This allows us to get, because $\varOmega_{n} \xrightarrow{H} \varOmega$, another variational equation $\tilde{a}(u^{\star}, \psi_{k}) 
	= \intS{{(g+if) \overline{\psi}_{k}}} 
	=: \tilde{l}(\psi_{k})$, for all ${\psi}_{k} \in \HH^{1}_{\omega}(D)$.
Therefore, using the compactness of the trace operator from $\HH^1(D)$ to $\LL^2(\varSigma)$, the Sobolev embedding $\vertiii{\cdot}_{\HH^{s-1/2}(\varSigma)} \leqslant c_s \, \vertiii{\cdot}_{\HH^s(D)}$ for $s > 1/2$, and Lemma~\ref{lem:uniform_boundedness_of_state}, we obtain the following convergences:
\begin{equation}\label{eq:convergences_of_forms}
\left\{\quad
\begin{aligned}
    \abs{\tilde{a}(u^{\star},\psi_{k}) - \tilde{a}(u^{\star},\tilde{\psi})}
    &\leqslant 
    c^{\star} \left( \sup_{D} \abs{\sigma} + \supb + c \right) \vertiii{\psi_{k}-\tilde{\psi}}_{\HH^{1}_{\omega}(D)}
    \xrightarrow{k \to \infty} 0,\\
    \abs{\tilde{l}(\psi_{k}) - \tilde{l}(\tilde{\psi})} 
    &\leqslant 
    cc_{s} \norm{(f,g)}\vertiii{\psi_{k}-\tilde{\psi}}_{\HH^{1}_{\omega}(D)}
    \xrightarrow{k \to \infty} 0,
\end{aligned}
\right.
\end{equation}
for some constant $c > 0$.
This means that $\tilde{a}(u^{\star},\psi_{k}) \xrightarrow{k \to \infty} \tilde{a}(u^{\star},\tilde{\psi})$ and $\tilde{l}(\psi_{k}) \xrightarrow{k \to \infty} \tilde{l}(\tilde{\psi})$.
Thus, $\tilde{a}(u^{\star}, \tilde{\psi})= \tilde{l}(\tilde{\psi})$, for all $\tilde{\psi} \in \HH^{1}_{\omega}(D)$.
The convergences in \eqref{eq:convergences_of_forms} and the equivalence $\vertiii{\varphi}_{\HH^{1}_{\omega}(D)}  \sim \vertiii{\varphi}_{\HH^{1}(D)} $ for functions $\varphi \in {\HH^{1}_{\omega}(D)}$ also tell us that $u^{\star} \in \HH^{1}(D)$ satisfies
\begin{equation}\label{eq:extended_weak_form_limit}
    \tilde{a}(u^{\star}, \tilde{\psi})= \tilde{l}(\tilde{\psi}),
    \qquad \forall \tilde{\psi} \in \HH^{1}(D).
\end{equation}

Observe that, by restricting the functions in ${\varOmega}$, we can verify that $u^{\star}|_{\varOmega} \in \spaceV(\varOmega)$.
Indeed, let us write $\omega^{\textsf{c}} = D\setminus \varOmega$ and $\omega_{n}^{\textsf{c}} = D\setminus \varOmega_{n}$.
We observe that $\intDOn{\abs{{\extu_{n}}}^{2}}=0$ and $\intDOn{\abs{{\extu_{n}}}^{2}} \longrightarrow \intDO{\abs{{u^{\star}_{n}}}^{2}}$.
Hence,
\begin{align*}
    \abs{
    \intDOn{\abs{{\extu_{n}}}^{2}}
    -\intDO{\abs{{u^{\star}}}^{2}}
    } 
    &\leqslant 
    \abs{\intD{(\chion-\chio)
    \abs{{u^{\star}}}^{2}}
    }\\
    &\qquad +\abs{
    \intD{\chion
    (\abs{{\extu_{n}}}^{2}-\abs{{u^{\star}}}^{2})}
    }
    =: \Psi_{1}^{n} + \Psi_{2}^{n}.
\end{align*} 
Since $\abs{{u^{\star}}}^{2} \in \LL^{1}(D)$ and $\chion\longrightarrow \chio$ in $L^{\infty}(D)$-weak-$\ast$ (see \cite[Prop.~2.2.28, p.~45]{HenrotPierre2018} or \cite[Eq.~(29)]{AfraitesRabago2024}), then $\Psi_{1}^{n} \xrightarrow{n \to \infty} 0$.
Because $\extu_{n} \rightharpoonup {u^{\star}} $ in $\HH^{1}(D)$ and $\HH^{1}(D) \hookrightarrow \hookrightarrow \LL^{2}(D)$, we infer that $\extu_{n} \rightarrow {u^{\star}}$ in $\LL^{2}(D)$. 
\sloppy By Lemma~\ref{lem:uniform_boundedness_of_state}, we also see that $\Psi_{2}^{n} \leqslant  \intD{ \abs{ \abs{{\extu_{n}}}^{2} - \abs{{u^{\star}}}^{2} } } \xrightarrow{n \to \infty} 0$.
Therefore $\intDO{\abs{{u^{\star}}}^{2}}=0$, and we conclude that $u^{\star} \in \spaceV({\varOmega})$.
Now, going back to \eqref{eq:extended_weak_form_limit}, we see that
\begin{equation}\label{eq:restricted_weak_form_limit}
    a(u^{\star}|_{\varOmega}, \psi)= l(\psi),
    \qquad \forall \psi = \tilde{\psi}|_{\varOmega} \in \spaceV(\varOmega).
\end{equation}

Finally, we will show the strong convergence $\extu_{n} \rightarrow {u^{\star}}$ in $\HH^{1}(D)$.
To do this, let us denote $w_{n} = \extu_{n} - u^{\star}$.
On the one hand, from \eqref{eq:extended_weak_form_limit}, with $\psi = w_{n} = \extu_{n} - u^{\star} \in \HH^{1}(D)$, we have
\begin{equation}\label{eq:first_equation}
		\tilde{a}(u^{\star}, w_{n})= \tilde{l}(w_{n}).
\end{equation}
On the other hand, we also have that
\begin{align*}
&\intOn{\sigma \nabla{u}_n \cdot \nabla \overline{(w_n|_{\varOmega_n})}} 
+ \intOn{(\vect{b} \cdot \nabla{u}_n) \overline{(w_n|_{\varOmega_n})}} 
+ i \intS{u_n \overline{(w_n|_{\varOmega_n})}} \\
&\qquad = \intS{(g+if) \overline{(w_n|_{\varOmega_n})}}.
\end{align*}
Extending all functions to $D$ via \eqref{eq:zero_extension}, we get
%
%
\begin{equation}\label{eq:second_equation}
	\tilde{a}({\extu}_{n}, w_{n})= \tilde{l}(w_{n}).
\end{equation}
Subtracting \eqref{eq:first_equation} from \eqref{eq:second_equation}, we get
\[
	\Psi({w}_{n}) := \intD{\sigma\nabla {w}_{n} \cdot \nabla \overline{w}_{n}}
    + \intD{\vect{b} \cdot \nabla {w}_{n} \overline{w}_{n}}
    = - i\intS{{w}_{n} \overline{w}_{n}}.
\]
Clearly, because ${w}_{n} = \extu_{n} - u^{\star}\in \HH^{1}(D)$ vanishes on a subdomain of $D$, we can apply a uniform Poincar\'{e}-Friedrich's inequality to obtain
\[
	0 \leqslant c(\sigma_{0} - c_{P} \supb) \vertiii{ \extu_{n} - u^{\star} }_{\HH^{1}(D)}^{2}
	\leqslant \abs{\Psi({w}_{n})}
	\leqslant \vertiii{ \extu_{n} - u^{\star} }_{\LL^{2}(\varSigma)}^{2},
\]
for some constant $c > 0$.

Using the compactness of the trace operator from $\HH^{1}(D)$ into $\LL^{2}({\varSigma})$, we can extract a subsequence, still denoted by
$\{\extu_{n}\}_{n \in \mathbb{N}}$, such that $\extu_{n}|_\varSigma \rightharpoonup  u^{\star}|_\varSigma $ in $\LL^{2}({\varSigma})$. 
Consequently, we get
\[
	0 \leqslant \vertiii{ \extu_{n} - u^{\star} }_{\HH^{1}(D)}^{2}
	\leqslant \dfrac{1}{c(\sigma_{0} - c_{P} \supb)} \vertiii{ \extu_{n} - u^{\star} }_{\LL^{2}(\varSigma)}^{2} 
	 \xrightarrow{n \to \infty} 0.
\]
In conclusion, $\extu_{n}$ converges strongly to $u^{\star}$ in $\HH^{1}(D)$.
This completes the proof of the proposition.
\end{proof}
\begin{proposition}\label{prop:continuity_the_cost_function}
Under the assumptions of Proposition~\ref{prop:strong_convergence_of_extensions}, we have \[ \lim_{n \to +\infty} J(\varOmega_n) = J(\varOmega). \]
\end{proposition}
The proof follows similar arguments to those employed in Proposition~\ref{prop:strong_convergence_of_extensions} and is therefore omitted.

\section{Shape Sensitivity Analysis}
\label{sec:Shape_sensitivity_analysis}
\subsection{Shape perturbations and material derivative} 
To solve Problem~\ref{prob:shape_optimization} numerically, we use a shape-gradient method with FEM, which requires the shape derivative of $J$. Here, we derive it via \textit{shape calculus} \cite{DelfourZolesio2011,HenrotPierre2018,MuratSimon1976,Simon1980,SokolowskiZolesio1992}, using a chain-rule approach based on the \textit{material} derivative of the solution to \eqref{eq:complex_PDE}. Let $\sfTheta$ denote admissible deformation fields $\VV$:
\[
\sfTheta := \{ \VV \in C^{1,1}(\overline{D}, \mathbb{R}^d) \mid \operatorname{supp}(\VV) \Subset \overline{D}_{\delta} \}.
\]
We let $\Vn=\langle \VV,\nn\rangle$, where $\nn$ is the outward unit normal vector to $\varOmega$.
From this point forward, we consider all deformation fields to be admissible unless indicated otherwise.

For $t \geqslant 0$ and $\VV \in \sfTheta$, we consider the operator $T_t: \overline{D} \to \overline{D}$ \cite[p.~147]{DelfourZolesio2011}:
\[
	T_t = \mathrm{id} + t \VV, \qquad T_0(\varOmega) = \varOmega \in \Oad.
\]
We assume that $t_0 > 0$ is sufficiently small so that, for all $t \in \textsf{I} := [0, t_0)$, $T_t$ is a $C^{1,1}$ diffeomorphism of $D$ onto itself with strictly positive Jacobian $I_t := \det(D T_t) > 0$.  
Thus, $T_t$ preserves the topology and regularity of $\varOmega$ under the perturbation.

Let us denote $A_{t} := I_{t}({D}T_{t}^{-1})({D}T_{t})^{-\top}$, $B_{t} := I_{t} |({D}T_{t})^{-\top} \nn|$, and $C_{t} := I_{t} ({D}T_{t})^{-\top}$. 
We assume that, for all $t \in \textsf{I}$, $A_t$ and $I_t$ are uniformly bounded and that $A_t$ is coercive \cite[p.~526]{DelfourZolesio2011}.  
It is straightforward to verify that the mappings $t \mapsto I_t$, $A_t$, $B_t$, and $C_t$ are continuously differentiable, with derivatives given by \cite[pp.~75--76, 79--85]{SokolowskiZolesio1992} (here $\dot{Z}_{0} = (d/dt){Z}_{t}|_{t=0}$): $\dot{I}_{0} = \divVV$, $\dot{B}_{0} = {\operatorname{div}}_{\tau} \VV = \divVV \big|_{\varSigma} - (D\VV\nn)\cdot\nn =:B$, $\dot{A}_{0} = (\divVV)\vect{I} -  D\VV - (D\VV)^\top =: A$, and $\dot{C}_{0} = (\divVV)\vect{I} - (D\VV)^\top =: C$, where ${\operatorname{div}}_{\tau}$ denotes the tangential divergence operator. 
Note that $B$ vanishes on $\varSigma$ for all $\VV \in \sfTheta$.


For $\varOmega \in \Oad$, $\varphi, \psi \in \spaceV(\varOmega)$, we introduce the bilinear form 
\[
	\mathcal{M}(\varphi,\psi) :=  -\left\{ \intO{ \left[ \left( \sigma A \nabla{\varphi} +  (\nabla \sigma \cdot \VV) \nabla{\varphi}  \right) \cdot \nabla \overline{\psi} + \left( C^{\top} \vect{b} + D \vect{b} \VV  \right) \cdot \nabla{\varphi} \overline{\psi} \right] } \right\}.
\]
If $\varOmega = \varOmega^{\star}$, we write $\mathcal{M}^{\star}(\varphi,\psi)$.

For \eqref{eq:complex_PDE}, the material derivative is stated in the following lemma.
\begin{lemma}\label{lem:material_derivative_of_u}
The material derivative $\dot{u} \in \spaceV(\varOmega)$ uniquely satisfies
\begin{equation}\label{eq:material_derivative_of_u}
	a(\dot{u}, \psi) = \mathcal{M}(u,\psi),
	\quad \forall \psi \in \spaceV(\varOmega).
\end{equation} 
\end{lemma}

The proof follows standard arguments and is analogous to Proposition~3.1 in \cite{Rabago2023b}, so it is omitted. Part of the proof requires showing that, for $t \in \textsf{I}$, the map $(t,\varphi,\psi) \mapsto \Phi(t,\varphi,\psi)$ is differentiable with respect to $t$, where the sesquilinear form $\Phi: \textsf{I} \times \spaceV(\varOmega) \times \spaceV(\varOmega) \to \mathbb{R}$ is
\[
	\Phi(t,\varphi, \psi) = \intO{\sigma^{t} A_{t} \nabla \varphi \cdot \nabla \overline{\psi}}
		 + \intO{ \vect{b}^{t} \cdot C_{t} \nabla \varphi \overline{\psi}}
		+ i \intS{B_{t} u \overline{\psi}},
\]
which is continuous and coercive on $\spaceV(\varOmega) \times \spaceV(\varOmega)$. Coercivity is ensured by choosing $c = \eta_{1} c_P^{-1} \abs{C_t}_\infty^{-1} > 0$ in Assumption~\ref{assumption:key_conditions}(iv).  
%
%
\subsection{Shape derivatives of the shape functional}
\label{subsec:Shape_derivatives_of_the_shape_functional}


It is straightforward to verify that $J$ is shape differentiable at $\varOmega \in \Oad$ along $\VV \in \sfTheta$ (see \cite[Eq.~(3.6), p.~172]{DelfourZolesio2011} for the defintion of \textit{shape differentiability}), since $T_t$ is a $C^{1,1}$-diffeomorphism and the maps $t \mapsto I_t$ and $t \mapsto u^t$ are differentiable.  
The following result characterizes the shape derivative of $J$ via the chain rule and Lemma~\ref{lem:material_derivative_of_u}.

%
%
\begin{proposition}[Shape gradient of $J$]
	\label{prop:shape_gradients}
	The first-order shape derivative of $J$ is given as follows:
	\begin{equation} \label{eq:shape_gradient} 	
		dJ(\varOmega)[\VV] 
		= \intG{G\nn \cdot \VV} 
		:= \intG{\left[\sigma\left(\dn{\realu} \dn{\imagp} - \dn{\imagu} \dn{\realp} \right)\right] \nn \cdot \VV},
	\end{equation} 
	where $G=G(u,p)$, $p = \realp + i \imagp \in \HH^{2}(\varOmega) \cap \spaceV(\varOmega)$ is the adjoint variable that uniquely solves the adjoint system
	\begin{equation}\label{eq:adjoint_system_p}
        \left\{\arraycolsep=3pt\def\arraystretch{1}
        \begin{array}{rcll}
               \div{ \sigma\nabla p} + \vect{b} \cdot \nabla p + (\divbb) p  & = & \imagu & \text{ in $\varOmega$},\\
                p & = & 0 & \text{on $\varGamma$},\\
                \sigma\dn{p} + p \vect{b} \cdot \nn -ip& = & 0 & \text{on $\varSigma$}.
        \end{array}
        \right.
        \end{equation}		 
\end{proposition}
The weak formulation of the adjoint system \eqref{eq:adjoint_system_p} is stated as follows: Find $p\in {\spaceV}(\varOmega)$ such that 
\begin{equation}\label{eq:adjoint_weak_form}
\intO{\left\{ \sigma\nabla p \cdot \nabla{\overline{\varphi}} + (\vect{b}\cdot \nabla \overline{\varphi})p \right\}}
+i\intS{p \overline{\varphi}}
=\intO{\imagu\overline{\varphi}},\ \forall\varphi \in {\spaceV}(\varOmega).
\end{equation} 
%
%
%
%
%
%
\begin{proof}[Proof of Proposition~\ref{prop:shape_gradients}]
	Clearly, as a consequence of the assumptions, $J$ is shape-differentiable and we can apply Hadamard's differentiation formula (see \cite{DelfourZolesio2011,HenrotPierre2018,SokolowskiZolesio1992}) to obtain
$dJ(\varOmega)[\VV] = \frac{1}{2}\intO{ \divVV\abs{\imagu}^{2} } + \intO{ \imagu\dot{u}_{2} }$.
For $\VV \in \sfTheta$ and sufficiently smooth $\psi$, the divergence theorem tells us that
\begin{equation}\label{eq:divergence_identity}
	-\intO{\divVV\psi}
	=\intO{\VV \cdot \nabla \psi}-\intG{\psi(\VV\cdot \nn)},
\end{equation}
Putting $\psi=\abs{\Im\{u\}}^{2}=\abs{\imagu}^{2} \in W^{1,1}(\varOmega)$ we deduce that
\begin{equation}\label{eq:dJ}
	dJ(\varOmega)[\VV] =-\intO{\imagu\VV \cdot \nabla \imagu} +\intO{ \imagu\dot{u}_{2} }.
\end{equation}
We rewrite the above expression with the goal of replacing $\dot{u}$ with the adjoint variable $p$.
To start, let us take $\varphi = \dot{u} \in \spaceV(\varOmega)$ in \eqref{eq:adjoint_weak_form}, to get
\begin{equation}\label{eq:weak_form_p:dotu}
\begin{aligned}
-\intO{\left\{ \sigma\nabla p \cdot \nabla{\overline{\dot{u}}} + (\vect{b}\cdot \nabla \overline{\dot{u}})p \right\}}
+i\intS{p \overline{\dot{u}}}
=\intO{\imagu\overline{\dot{u}}}.
\end{aligned}
\end{equation}
Next, we choose $\psi =p\in \spaceV(\varOmega)$ in \eqref{eq:material_derivative_of_u} and note that $B={\operatorname{div}}_{\tau} \VV=0$ on $\varSigma$, for any $\VV \in \sfTheta$, to obtain
$a(\dotu, p) = \mathcal{M}(u,p)$.
Comparing this equation with the sesquilinear form $a$ given in \eqref{eq:sesquilinear_form_a}, with $\varphi=\dot{u}$ and $\psi=p$, we arrive at the equation $a(\dot{u},p) = \mathcal{M}(u,p)$.

Now, conjugating Equation~\eqref{eq:weak_form_p:dotu}, {it can be verified that}
\begin{equation}\label{eq:identity_M}
-\intO{\imagu\dot{u}} 
	=-\intO{\imagu\VV\cdot \nabla{u}}+\intG{ \sigma \dn{u} \dn{\overline{p}} \Vn }.
\end{equation}

Finally, by comparing the imaginary parts on both sides of this equation and returning to \eqref{eq:dJ}, we obtain the desired shape derivative \eqref{eq:shape_gradient}.
\end{proof}
\begin{remark}\label{rem:necessary_condition}
Assume $\varOmega^{\star} \in \Oad$ solves Problem~\ref{prob:inverse}, i.e., $\imagu = 0$ in $\varOmega^{\star}$; then $\varOmega^{\star}$ is stationary for Problem~\ref{prob:shape_optimization}, hence $dJ(\varOmega^{\star})[\VV] = 0$ for all $\VV \in \sfTheta$.
\end{remark}
From now on, $\varOmega^{\star} \in \Oad$ is called a \textit{critical shape} of $J$ if it satisfies Remark~\ref{rem:necessary_condition}.
%
%
%
%
\subsection{Second-order shape derivative of the cost function at the critical shape}
In this section, we derive the second-order shape derivative (\cite[Def.~2.3]{DelfourZolesio1991a}) of $J$ using the material derivatives of $u$ and $p$. 
Direct computation via the shape derivatives of $u$ and $p$ requires $u, p \in \HH^3(\varOmega) \cap \spaceV(\varOmega)$, which holds if $\varOmega \in \Oad \cap C^{2,1}$ and $(f,g) \in H^{5/2}(\varOmega) \times H^{3/2}(\varOmega)$, {and the coefficients satisfy $\sigma \in W^{2,\infty}(D)$ and $\vect{b} \in W^{1,\infty}(D)^d$} (cf.~\cite{BacaniPeichl2013}). 
By contrast, expressing the second-order derivative in terms of material derivatives relaxes this requirement, allowing $\varOmega \in C^{1,1}$.

%
%
\begin{proposition}[Shape Hessian of $J$]\label{prop:shape_Hessian}
Assume $\varOmega$, $\VV$, $\WW$, and the Cauchy pair $(f,g)$ are sufficiently regular and admissible so that $J$ is twice shape differentiable. Then, the shape Hessian at a critical shape $\varOmega^{\star}$ is
\[
	d^{2}J(\varOmega^{\star})[\VV,\WW] 
	= \intGstar{h[\WW]\nn \cdot \VV}
	= \intGstar{\sigma\dn{w_2}[\WW]\dn{\realu}\nn \cdot \VV},
\]
where $\imagu = \Im\{u\}$, $u$ solves \eqref{eq:complex_PDE} with $\varOmega = \varOmega^{\star}$, and $w=w_{1}+iw_{2}=\Re\{w\}+i\Im\{w\} \in \HH^{2}(\varOmega^{\star}) \cap \spaceV(\varOmega^{\star})$.
The adjoint variable $w \in $ satisfies the following:
\begin{equation}\label{eq:adjoint_system_q}
    \left\{\arraycolsep=3pt\def\arraystretch{1}
    \begin{array}{rcll}
            \div{\sigma\nabla w} + \vect{b} \cdot \nabla w + (\divbb) w & = & \dot{u}_{2,\WW} & \text{ in $\varOmega^{\star}$},\\
            w & = & 0 & \text{on $\varGamma^{\star}$},\\
            \sigma\dn{w} + w \vect{b} \cdot \nn -iw& = & 0 & \text{on $\varSigma$}.\\
    \end{array}
    \right.
\end{equation}
\end{proposition}
\begin{remark}
	Notice that $w$ is simply the derivative of $p$ with respect to $\WW \in \sfTheta$ while $\dot{u}_{2,\WW}$ is the derivative of $p$ with respect to $\WW$ (cf. \cite{RabagoAzegami2018,Rabago2023b}).
\end{remark}
\begin{proof} 
Our assumptions ensure the existence of $\ddot{u}_{2,\VV,\WW} \in H^1(\varOmega)$.
Then, differentiating $J$ twice with respect to $\varOmega$, first along $\VV$ and subsequently along $\WW$, we obtain
\begin{align*}
d^{2}J(\varOmega)[\VV,\WW] 
	&= \frac{1}{2}\intO{ \left( \divVV \divWW \abs{\imagu}^{2}+ 2\divVV\imagu\dot{u}_{2,\WW} \right) }\\
 	&\qquad + \intO{ \left( \dot{u}_{2,\VV}\dot{u}_{2,\WW}+\imagu\ddot{u}_{2,\VV,\WW} \right) }.
\end{align*}
At $\varOmega = \varOmega^{\star}$, we have $\imagu = 0$, so we get $d^{2}J(\varOmega^{\star})[\VV,\WW] = \intOstar{ \dot{u}_{2,\VV}\dot{u}_{2,\WW} }$.
Our objective is to rewrite this expression in terms of an appropriate adjoint variable, eliminating $\dot{u}_{2,\VV}$ and $\dot{u}_{2,\WW}$. 
To this end, we multiply the first equation in \eqref{eq:adjoint_system_q} by $\overline{\dot{u}}_{\VV}$, integrate over $\varOmega^\star$, and take the complex conjugate, yielding:
\begin{equation}\label{eq:first_equation_for_hessian_computation}
\begin{aligned}
\hat{a}(\dot{u}_{\VV},  w)
	&:=-\intOstar{\left[ \sigma\nabla \overline{w} \cdot \nabla{\dot{u}_{\VV}}
		+ (\vect{b}\cdot \nabla \dot{u}_{\VV})\overline{ w} \right]}
		-i\intS{\overline{w} \dot{u}_{\VV}}\\
	&=\intOstar{\dot{u}_{2,\WW}\dot{u}_{\VV}}.
\end{aligned}
\end{equation}
Next, in \eqref{eq:material_derivative_of_u}, we set $\psi = w$ and replace $\dot{u}$ with $\dot{u}_{\VV}$, yielding $a(\dot{u}_{\VV}, w) = \mathcal{M}^\star({u}_{\VV}, w)$,
where $a$ is the sesquilinear form defined in \eqref{eq:sesquilinear_form_a} with $\varphi = \dot{u}_{\VV}$, $\psi = w$, and $\varOmega = \varOmega^\star$. That is,
\begin{equation}\label{eq:a_equal_hat_a}
a(\dot{u}_{\VV}, w) = -\hat{a}(\dot{u}_{\VV}, w),
\end{equation}
with $\hat{a}$ given in \eqref{eq:first_equation_for_hessian_computation}.
Equation \eqref{eq:a_equal_hat_a} then implies
\[
-\intOstar{\dot{u}_{2,\WW} \dot{u}_{\VV}} = \mathcal{M}^\star({u}_{\VV}, w) 
= -\intOstar{\dot{u}_{2,\WW} \VV \cdot \nabla \realu} + \intGstar{\sigma \dn{\overline{w}} \dn{\realu} \Vn}.
\]
Taking imaginary parts on both sides gives the desired expression.
\end{proof}
%
%
%

\subsection{Compactness of the Hessian at a critical shape}
Having computed the shape Hessian at $\varOmega^\star$, we analyze the stability of the shape optimization problem (see \cite{Dambrine2002,DambrinePierre2000,Eppler2000c,EpplerHarbrecht2012b}), showing its ill-posedness via the lack of coercivity of the Hessian's bilinear form in $H^{1/2}(\varGamma^\star)$.

The main result of this subsection is:
\begin{proposition}[Compactness at a critical shape]\label{prop:compactness_result}
The Riesz operator corresponding to the shape Hessian $d^2 J(\varOmega^\star)$ defined from $H^{1/2}(\varGamma^\star)$ to $H^{-1/2}(\varGamma^\star)$ is compact.
\end{proposition}

\begin{proof}
For $\varOmega^{\star} \in \Oad$ and $\VV, \WW \in  \sfTheta \subset C^{1,1}(\overline{D}; \mathbb{R}^d)$, we have $\Vn = \VV \cdot \nn$, $\Wn = \WW \cdot \nn \in H^{1/2}(\varGamma^{\star})$, $\dn{\realu} \in H^{1/2}(\varGamma^{\star})$, and $\dn{w_2}[\WW] \in H^{1/2}(\varGamma^{\star})$.

We introduce the following function mappings associated with $\varGamma^{\star}$:
\begin{align*}
	\mathcal{S}: H^{1/2}(\varGamma^{\star})^{d} &\longrightarrow H^{-1/2}(\varGamma^{\star}), & \mathcal{T}: H^{1/2}(\varGamma^{\star})^{d}	&\longrightarrow H^{1/2}(\varGamma^{\star}), \\
\WW &\longmapsto \sigma \, \dn{w_2}[\WW], & \VV &\longmapsto \dn{\realu} \, \Vn.
\end{align*}
Thus, we can write
\[
d^{2}J(\varOmega^{\star})[\VV,\WW]
=  \langle \mathcal{S}(\Vn), \mathcal{T}(\Wn) \rangle_{H^{-1/2}(\varGamma^\star),\, H^{1/2}(\varGamma^\star)}
= \intGstar{ \sigma \, \dn{w_2}[\WW] \, \dn{\realu} \, \Vn }.
\]

On the one hand, the operator $\mathcal{T}$ is continuous, since multiplication of
$\partial_{\mathbf{n}}\realu \in H^{1/2}(\varGamma^\star)$ by the $C^{0,1}$ function $\Vn$
defines a continuous mapping on $H^{1/2}(\varGamma^\star)$.
On the other hand, we show that $\mathcal{S}$ is compact.
To this end, we decompose $\mathcal{S}$ as $\mathcal{S}=S_{3}\circ S_{2}\circ S_{1}\circ S_{0}$, where
\begin{align*}
    {S}_0 : H^{1/2}(\varGamma^\star)^{d} &\longrightarrow H^{1}(\varOmega^\star), & 
    {S}_1 : H^{1}(\varOmega^\star) &\longrightarrow \HH^{2}(\varOmega^\star), \\
    \WW &\longmapsto \dot{u}_{2,\WW}, &
    \psi &\longmapsto w,\\[0.5em]
    {S}_2 : \HH^{2}(\varOmega^\star) &\longrightarrow H^{1/2}(\varGamma^\star), &
    {S}_3 : H^{1/2}(\varGamma^\star) &\longrightarrow H^{-1/2}(\varGamma^\star), \\
    w &\longmapsto \dn{w_2}, &
    \psi &\longmapsto \sigma \psi.
\end{align*}
The operators ${S}_0$, ${S}_1$, ${S}_2$, and ${S}_3$ are continuous, with ${S}_3$ also compact due to the Rellich--Kondrachov embedding $H^{1/2}(\varGamma^\star) \hookrightarrow H^{-1/2}(\varGamma^\star)$ \cite{AdamsFournier2003}.
Therefore, the composition satisfies $\mathcal{S}(\WW) = {S}_3\big({S}_2\big({S}_1\big({S}_0(\WW)\big)\big)\big) 
= \sigma \, \dn{w_2}[\WW]$.
Since ${S}_0$, ${S}_1$, and ${S}_2$ are continuous and ${S}_3$ involves a compact embedding, it follows that $\mathcal{S}$ is compact, as required.
\end{proof}
Proposition~\ref{prop:compactness_result} highlights the inherent instability of problem~\ref{prob:shape_optimization}. Near the critical shape $\varOmega^{\star} = D\setminus\overline{\omega}^\star$ and for small $t > 0$, $J$ follows its second-order expansion. Consequently, no uniform estimate like $C t \sqrt{J(\varOmega_t)}$ exists independent of the deformation field $\VV$. This shows that the shape gradient’s sensitivity strongly depends on the deformation direction and fails for highly oscillatory perturbations where $J$ degenerates.

Perimeter or surface-area regularization can stabilize shape optimization under noisy data
\cite{AfraitesDambrineEpplerKateb2007,CherratAfraitesRabago2025b,RabagoAzegami2018}, although its necessity depends on the boundary conditions, noise level, and shape regularity.
In some cases, second-order generalized impedance conditions yield accurate reconstructions without explicit regularization
\cite{AfraitesDambrineEpplerKateb2007,CaubetDambrineKatebTimimoun2013,CaubetDambrineKateb2013}, and ADMM-based methods also provide stable alternatives.
In Subsection~\ref{subsec:ADMM_algorithm}, we propose a tailored ADMM scheme for shape identification with complex PDE constraints, which improves reconstructions under noisy data and for geometries with pronounced concavities.

\section{Numerical Algorithms and Examples} 
\label{sec:numerical_algorithms_and_examples}
We implement the proposed method using a FEM-based descent algorithm driven by the shape gradient, following \cite{CherratAfraitesRabago2025b,CherratAfraitesRabago2025}. The first subsection summarizes the conventional numerical approach for clarity and self-containment.
\subsection{Conventional numerical scheme for shape optimization}
\label{subsec:conventional_scheme}
A natural descent direction for $J$ is $\VV = -G\nn$ with $G \in L^2(\varGamma)$ and $G \not\equiv 0$. 
However, this may lead to unstable reconstructions and a deterioration of mesh quality due to the formation of irregular or poorly shaped elements.
To address this issue, we adopt an $H^1$-Riesz representation of the shape gradient \cite{Doganetal2007,Azegami2020}; that is, we seek $\VV \in H_{\varSigma,0}^{1}(\varOmega)^d$ such that
\begin{equation}\label{eq:Sobolev_gradient_computation}
	c_{b} \intO{ \nabla \VV : \nabla \vect{\varphi}} + (1 - c_{b}) \intG{ \nabla_{\varGamma} \VV : \nabla_{\varGamma} \vect{\varphi} }
	= - \intG{ G \nn \cdot \vect{\varphi} },
\end{equation}
for all $\vect{\varphi} \in H_{\varSigma,0}^{1}(\varOmega)^{d}$, with $c_{b} \in (0,1]$.
In this work, $c_b$ is chosen as $0.7$  

The resulting Sobolev gradient \cite{Neuberger1997} smoothly extends $-G\nn$ into $\varOmega$, with the tangential term improving regularity \cite{Doganetal2007}.

{
\begin{remark}
While the $H^{1}$ Riesz representation promotes numerical stability of descent algorithms, steepest descent methods in Banach spaces such as $W^{1,p}$ or $W^{1,\infty}$ may better capture non-smooth geometric features \cite{DeckelnickHerbertHinze2022,DeckelnickHerbertHinze2024,DeckelnickHerbertHinze2025,DelfourZolesio2011}. 
In the present work, we adopt an $H^{1}$ Sobolev-gradient framework primarily for its ease of implementation, while still providing a satisfactory balance between stability and convergence for the inverse advection--diffusion problem in the presence of noisy measurements.
\end{remark}
}

To compute the $k$th domain approximation $\varOmega^{k}$, we follow these steps:
\begin{description}
    \setlength{\itemsep}{0.1em}
    \item[1. \textit{Initialization}] Select an initial shape $\varOmega^{0}$.
    \item[2. \textit{Iteration}] For $k = 0, 1, 2, \ldots$:
    \begin{enumerate}
    \item[2.1] Solve the state and adjoint state systems on the current domain $\varOmega^{k}$.
    \item[2.2] Select a step size $t^{k} > 0$ and compute the update vector $\VV^{k}$ in $\varOmega^{k}$.
    \item[2.3] Update the domain via $\varOmega^{k+1} = (\mathrm{id} + t^{k}\VV^{k})\varOmega^{k}$.
    \end{enumerate}
    \item[3. \textit{Stopping Test}] Repeat the iteration until a convergence criterion is satisfied.
\end{description}

{
\begin{remark}
We note that solving \eqref{eq:Sobolev_gradient_computation} smooths the descent velocity and improves mesh quality, but does not guarantee a diffeomorphic update. To prevent folding or inversion, one can either limit node displacements using an adaptive step size or monitor the Jacobian determinant to regenerate the mesh if necessary.
\end{remark}
}

In Step 2.2 we compute $t^k = \mu J(\varOmega^k)/|\VV^k|^2_{H^1(\varOmega^k)^d}$ \cite[p.~281]{RabagoAzegami2020}, reducing $\mu>0$ if necessary to avoid mesh inversion. In our experience, this strategy is sufficient. We set $\mu = 0.1$ in 2D and $0.01$ in 3D. The algorithm stops after a fixed number of iterations or if $t^k < 10^{-8}$.

We test the method in 2D and 3D. In 2D, we examine constant and spatially varying $\sigma$ with spatially varying $\vect{b}$ and introduce an ADMM modification to address limitations. The resulting CCBM--ADMM scheme is then applied to 3D, showing robustness to non-convex obstacles and noisy data.
All computations were performed on a MacBook Pro with an Apple M1 chip {with 16GB RAM, via \textsc{FreeFem++}} \cite{Hecht2012}.
\subsection{Numerical examples in 2D}
\label{subsec:Numerical_Examples_2D} 
For the 2D numerical examples, the specimen is taken as a unit circle centered at the origin, with $\sigma(x) = 1.1 + \sin(\pi x_{1})\sin(\pi x_{2})$ and $\vect{b} = (1.1 - \sin t, 1.1 + \cos t)$ for $t \in [0,2\pi]$.  
Synthetic data are generated from the Neumann boundary input $g(t) = 2 + \cos t$, with the measurement $f = u$ taken on $\partial \varOmega$.
To avoid ``inverse crimes" \cite[p.~154]{ColtonKress2013}, we use different schemes for data generation and inversion: $P_{2}$ elements on a finer mesh with 300 boundary nodes for data, and $P_{1}$ elements on a coarser mesh with 150 nodes ($h = 0.05$) for inversion.

We consider four obstacle geometries---ellipse, dumbbell, peanut, and L-block---shown as black solid interior shapes in Figure~\ref{fig:figure1}.

For all cases, the algorithm starts from $\varGamma^0 = C(\vect{0},0.6)$ and runs for 600 iterations, which provides satisfactory results without further tuning.  

To test noise robustness, we set $u^\delta = (1 + \delta \text{g.n.}) u^\star$, where $u^\star$ is the exact solution for input $g$ and ``g.n." denotes Gaussian noise with zero mean and standard deviation $\norm{u^\star}_{L^\infty(\varOmega)}$. The noisy measurement is $f|_{\varSigma} := u^\delta$.

\subsection{Numerical results and discussion in 2D}
\label{subsec:discussion_in_2D} 

Here, we present numerical results for the 2D case. Figure~\ref{fig:figure1} shows reconstructed shapes for several test cases: left and right panels correspond to constant and spatially varying $\sigma$, respectively. For each case, the true inclusion, initial guess (dashed), and reconstructions from exact and noisy data are shown. The outer and inner black curves indicate the domain boundary and true inclusion, while colored curves correspond to noise levels $\delta = 0\%, 5\%$.

Our method accurately reconstructs shapes in the noise-free case. Moderate noise reduces reconstruction quality, especially for irregular or non-convex inclusions, due to both noise and the variability of $\sigma$ (and $\vect{b}$); see Figure~\ref{fig:figure1}. Even though not shown here, similar tests indicate that spatially varying $\vect{b}$ tends to yield slightly less accurate results than the constant case.  

Figure~\ref{fig:costs_2d} illustrates the L-block case, which violates the $C^{1,1}$ regularity assumption on the obstacle. Despite this, the cost functional stabilizes after several iterations, lower final values are achieved with exact data, and gradient norms decrease with minor oscillations, confirming convergence. 
Thus, the reconstruction remains reasonable even for this Lipschitz-smooth obstacle.  
\begin{figure}[h!]
\centering
\resizebox{0.235\linewidth}{!}{\includegraphics{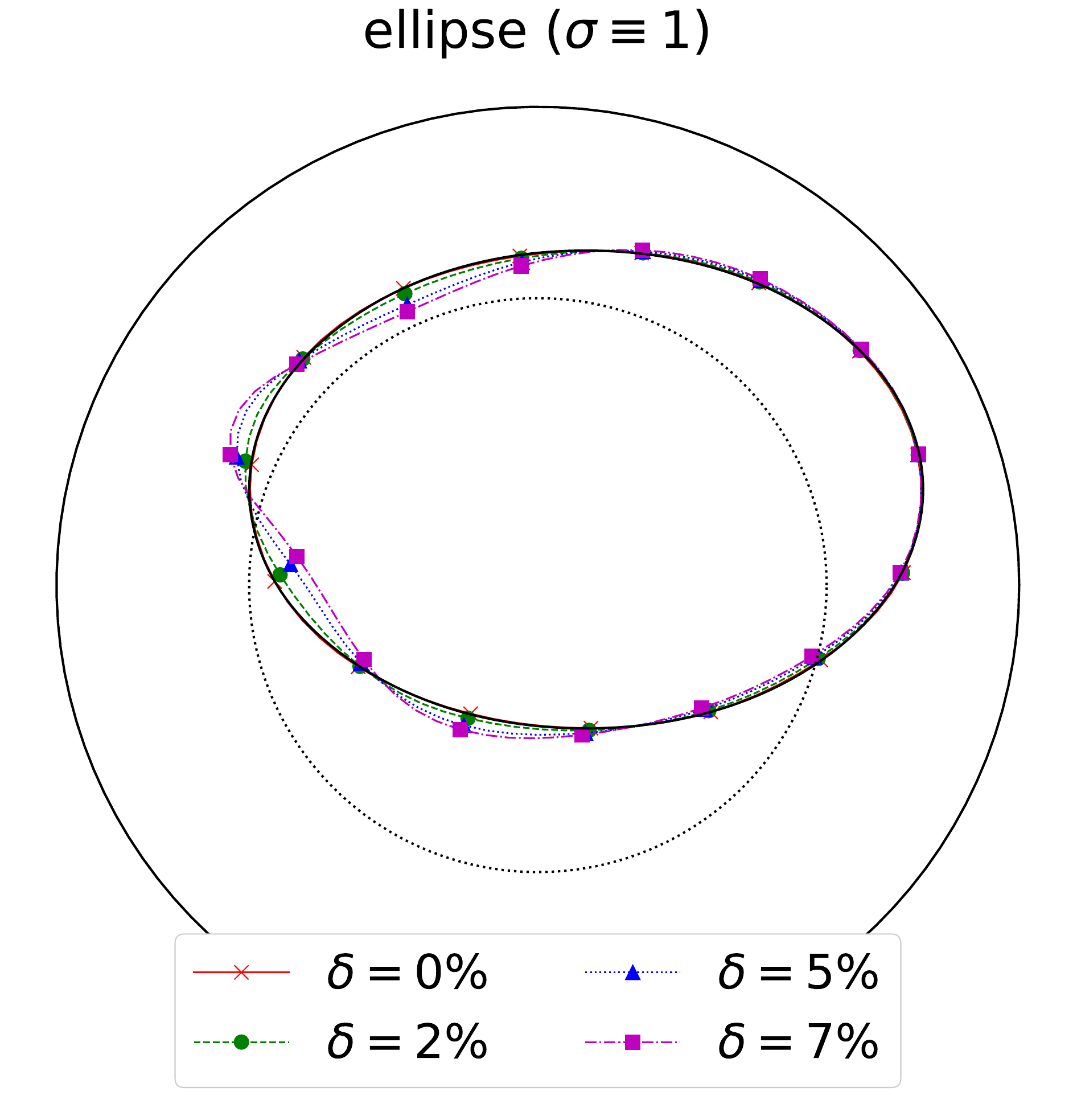}}\ 
\resizebox{0.235\linewidth}{!}{\includegraphics{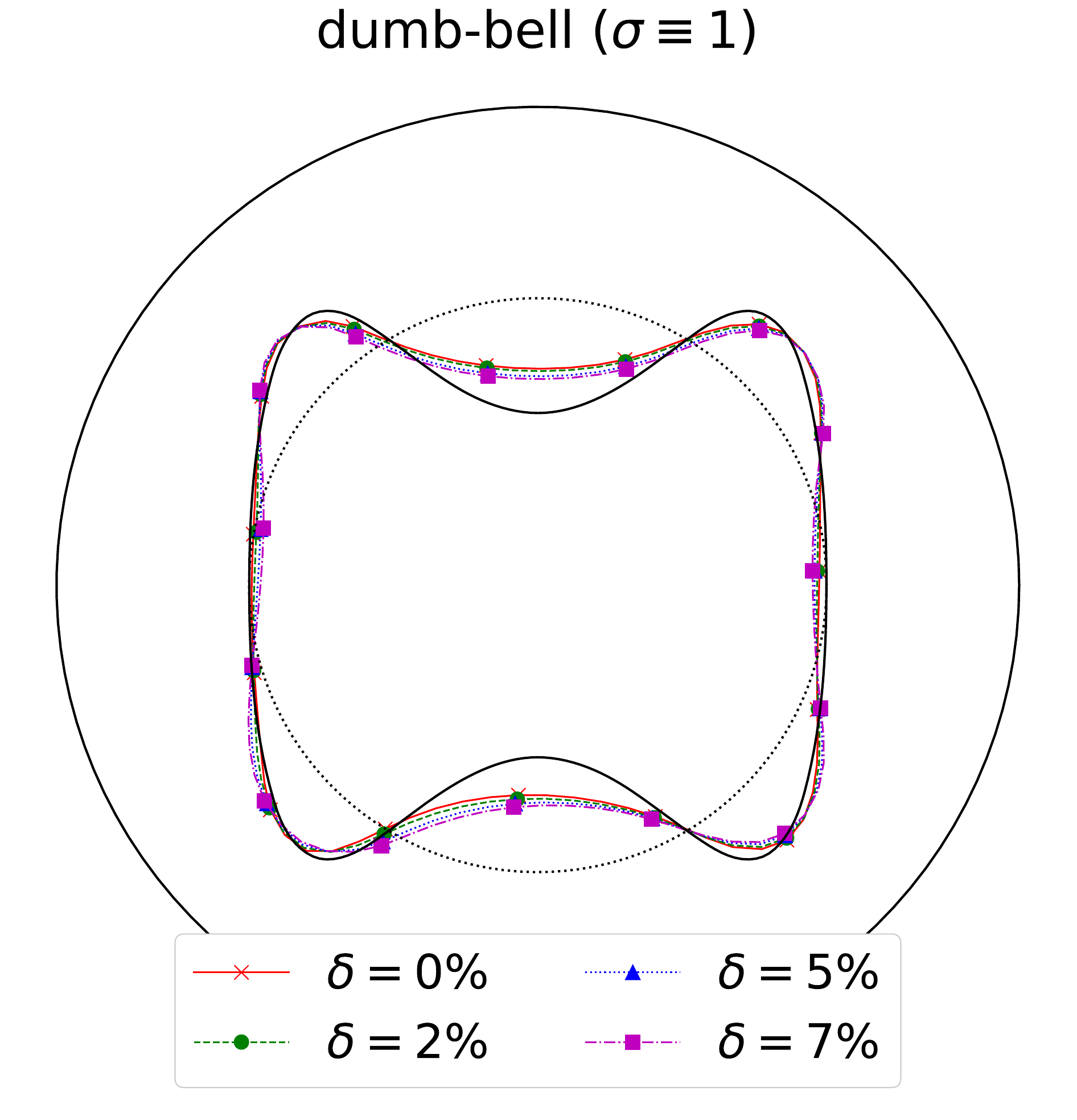}}\hfill 
\resizebox{0.235\linewidth}{!}{\includegraphics{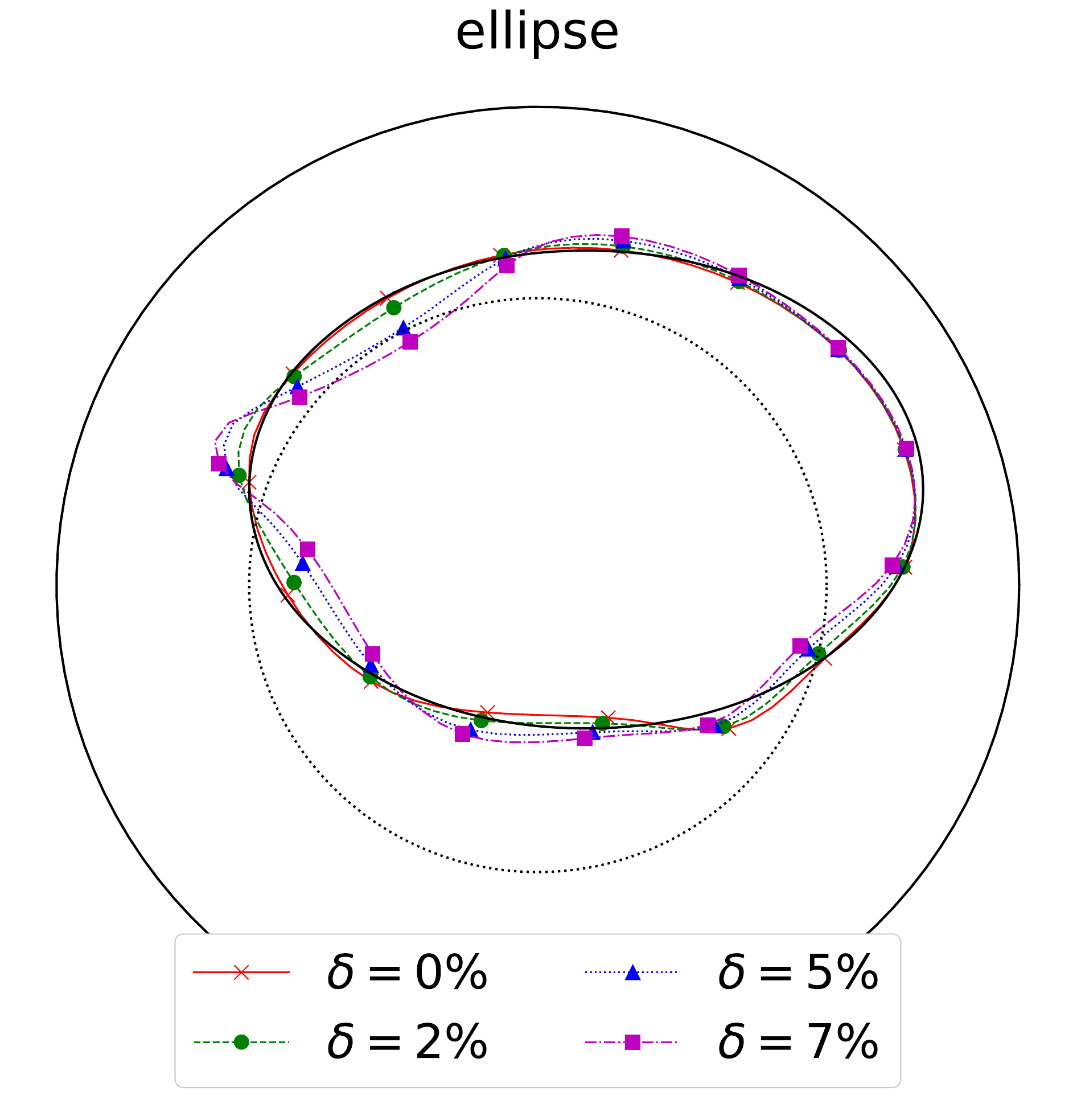}}\ 
\resizebox{0.235\linewidth}{!}{\includegraphics{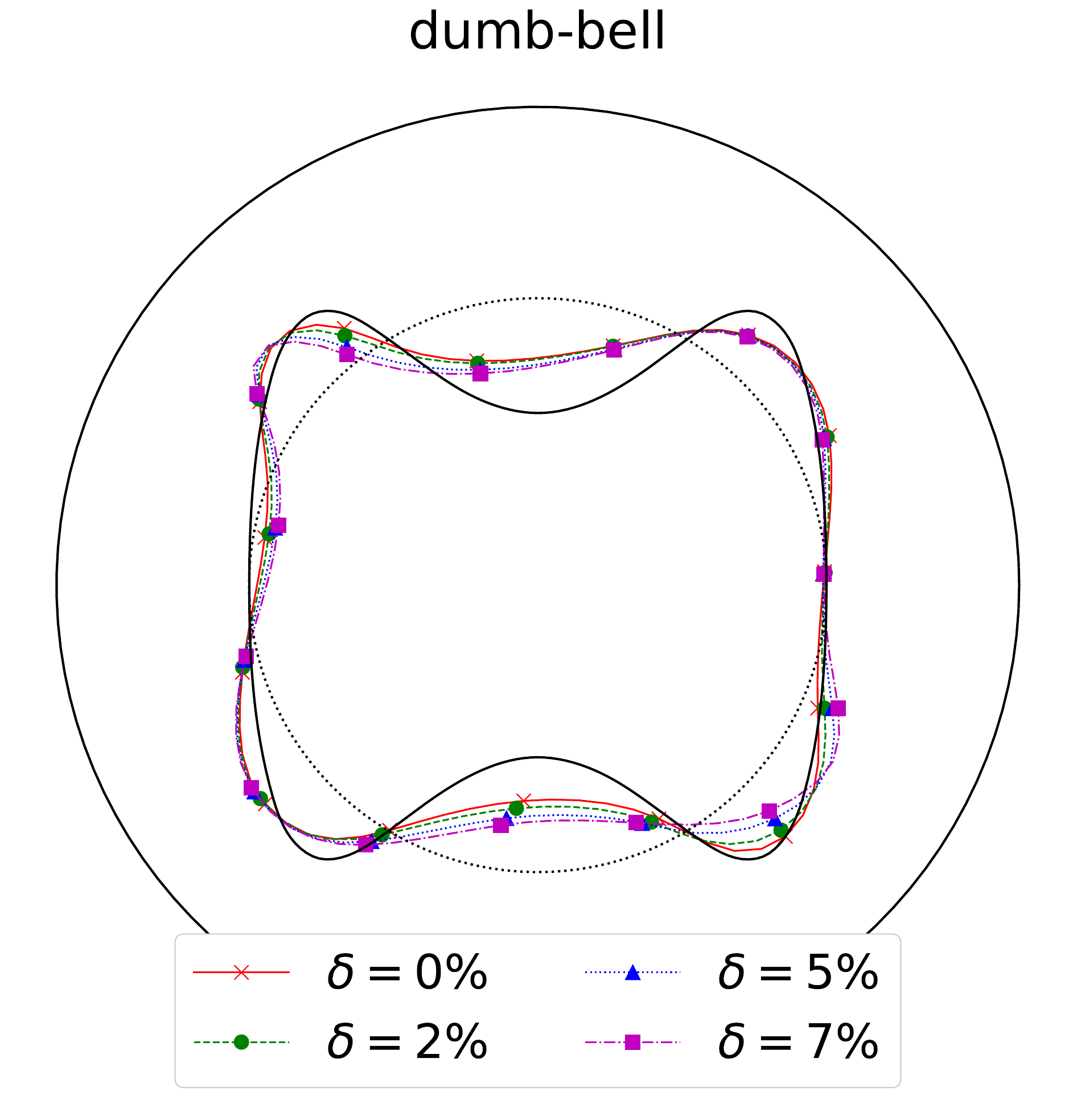}}\\[0.5em]
\resizebox{0.235\linewidth}{!}{\includegraphics{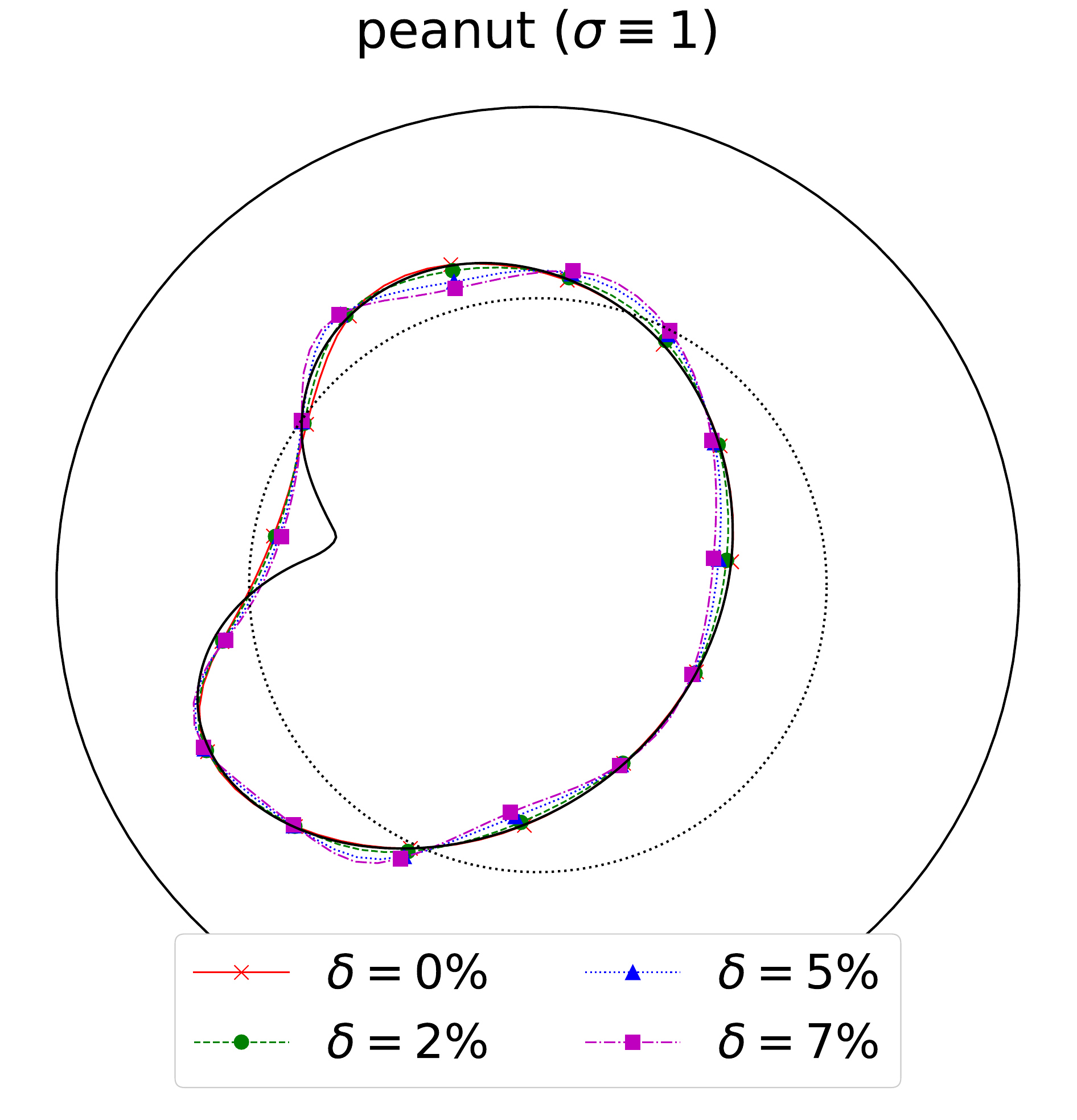}}\ 
\resizebox{0.235\linewidth}{!}{\includegraphics{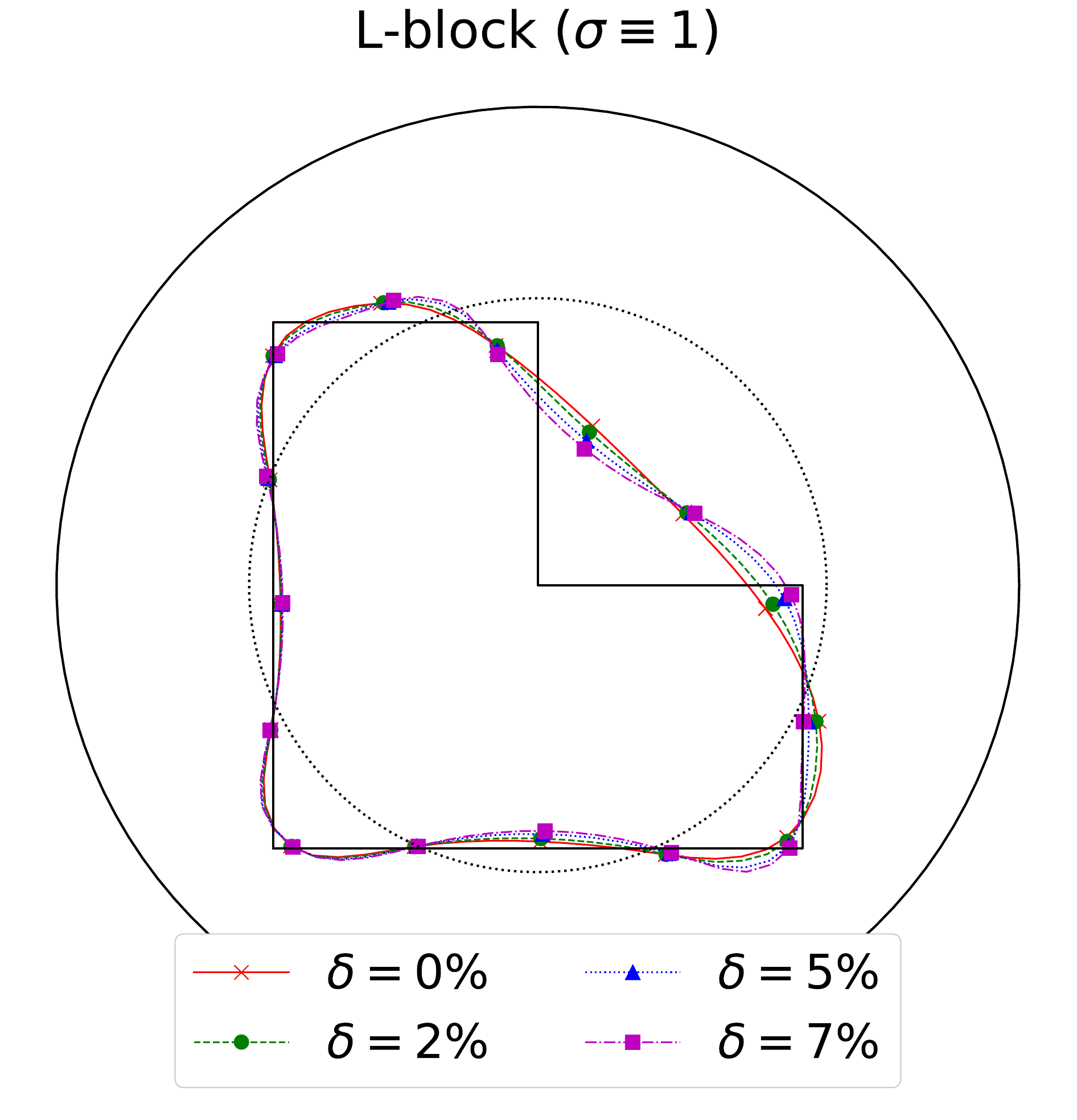}}\hfill
\resizebox{0.235\linewidth}{!}{\includegraphics{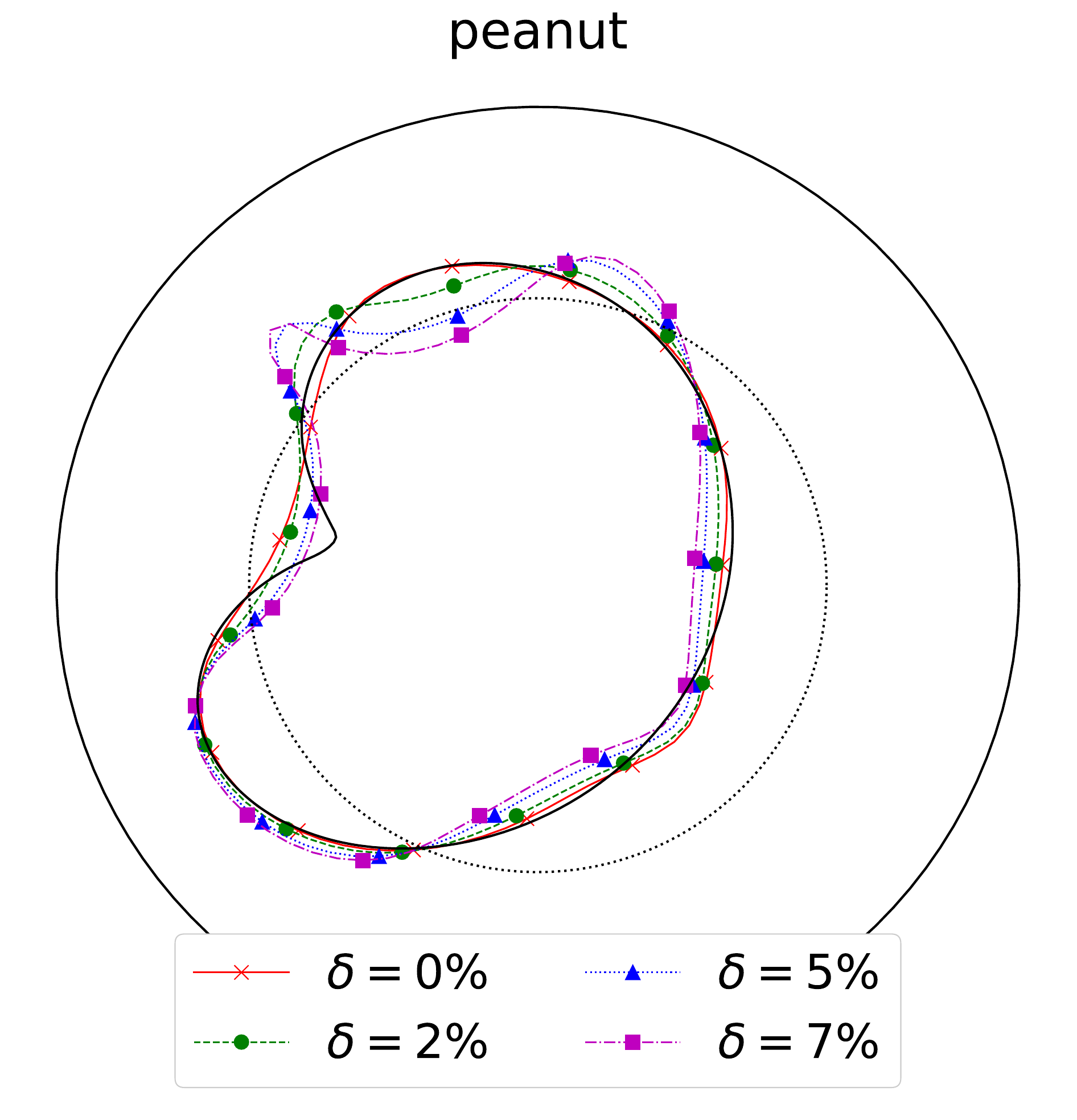}}\ 
\resizebox{0.235\linewidth}{!}{\includegraphics{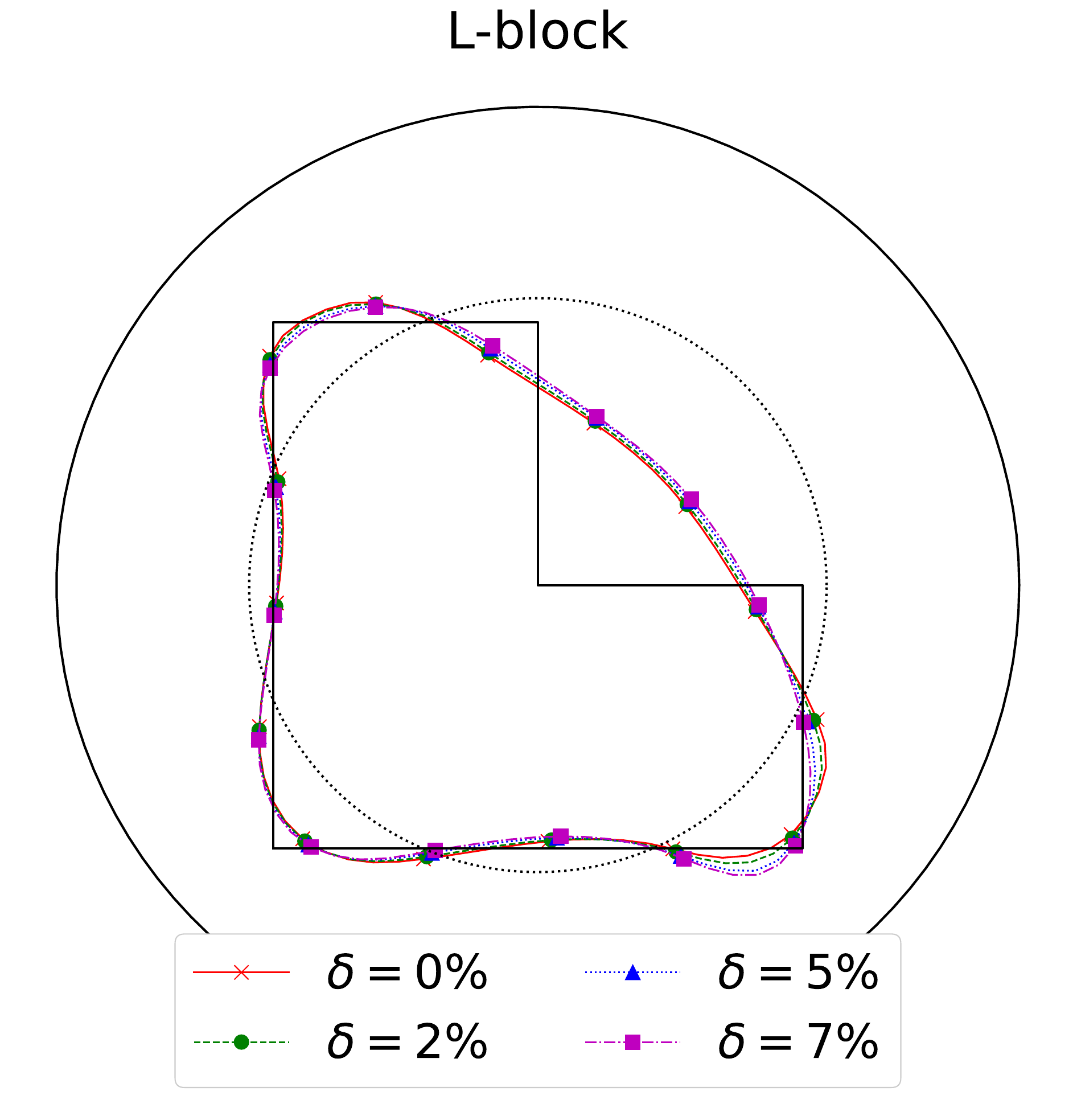}}
\caption{Results of the numerical experiments under exact and noisy measurements with noise levels $\delta = 2\%, 5\%$: left, constant $\sigma \equiv 1$; right, spatially varying $\sigma(x) = 1.1 + \sin(\pi x_{1})\sin(\pi x_{2})$.}
\label{fig:figure1}
\end{figure}
\begin{figure}[h!]
\centering
\resizebox{0.24\linewidth}{!}{\includegraphics{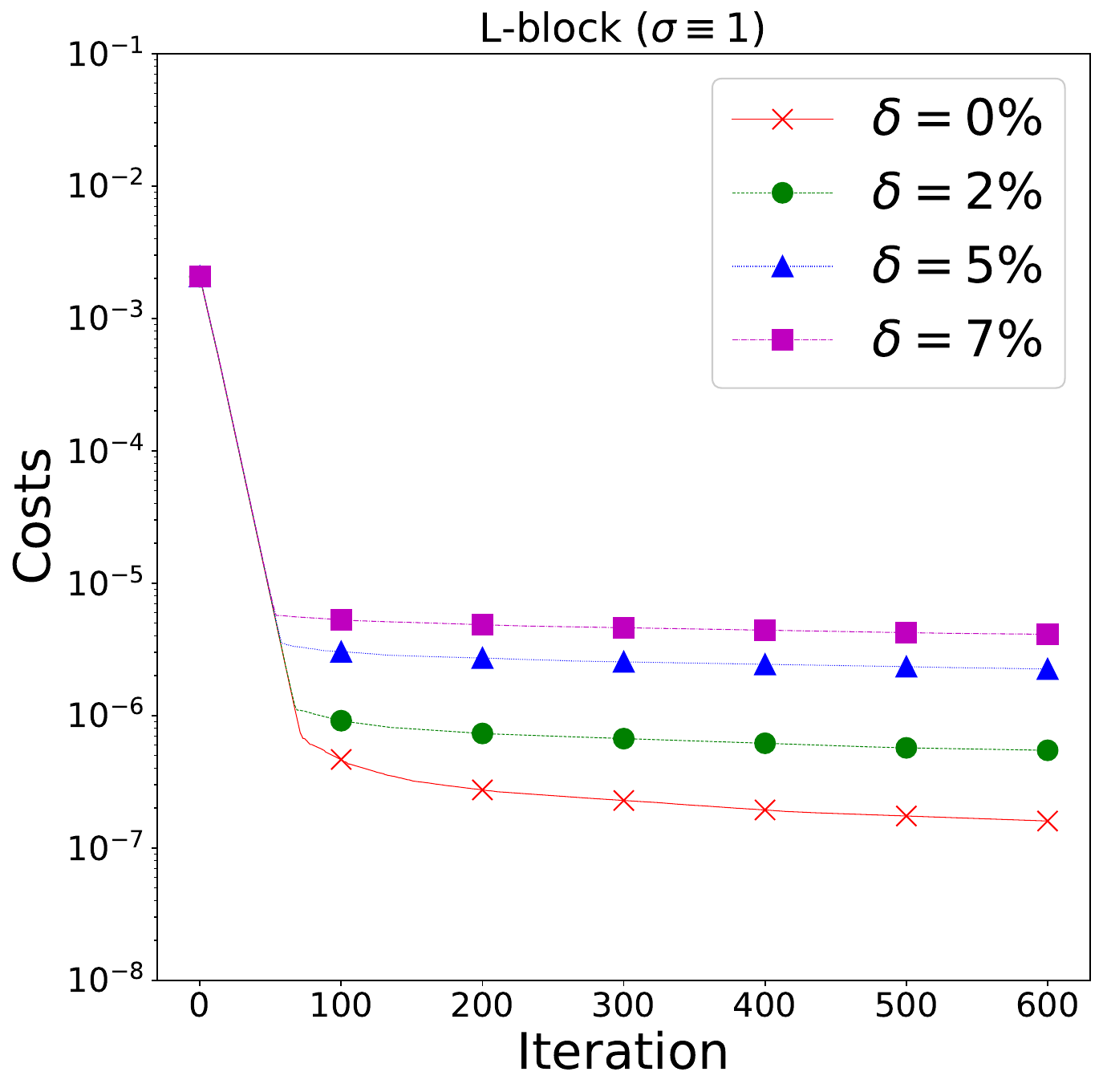}}\ 
\resizebox{0.24\linewidth}{!}{\includegraphics{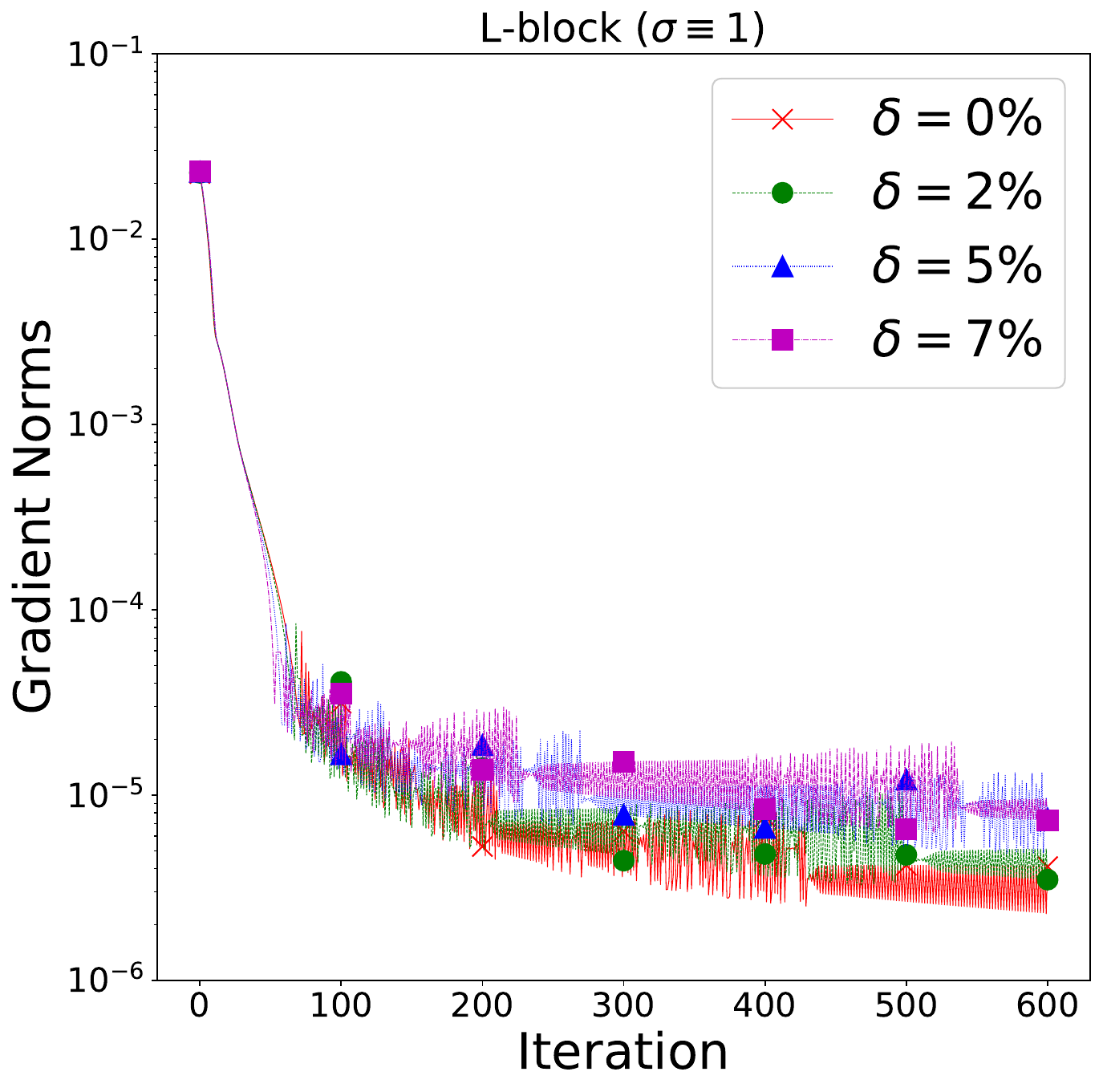}}\hfill
\resizebox{0.24\linewidth}{!}{\includegraphics{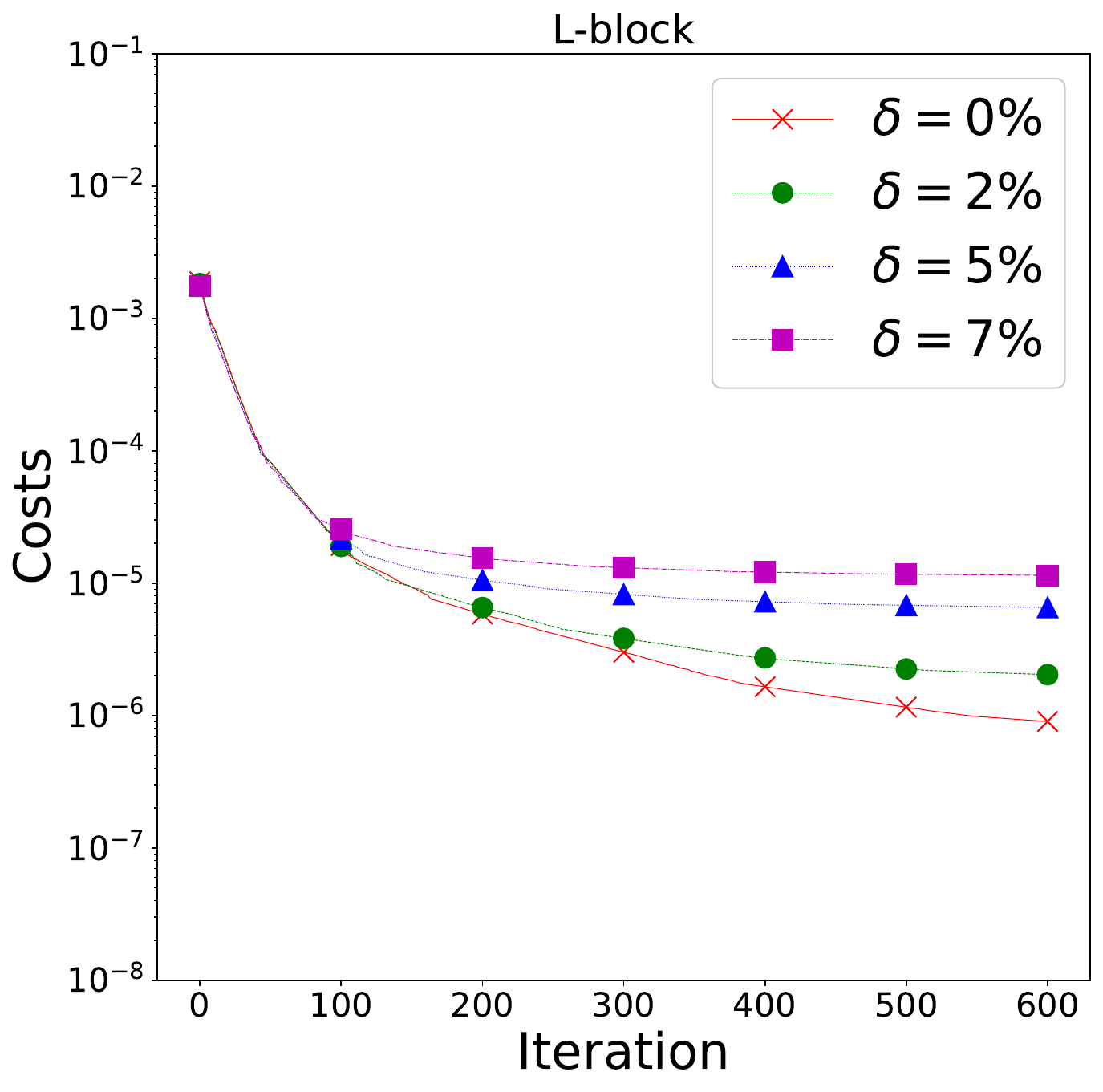}}\
\resizebox{0.24\linewidth}{!}{\includegraphics{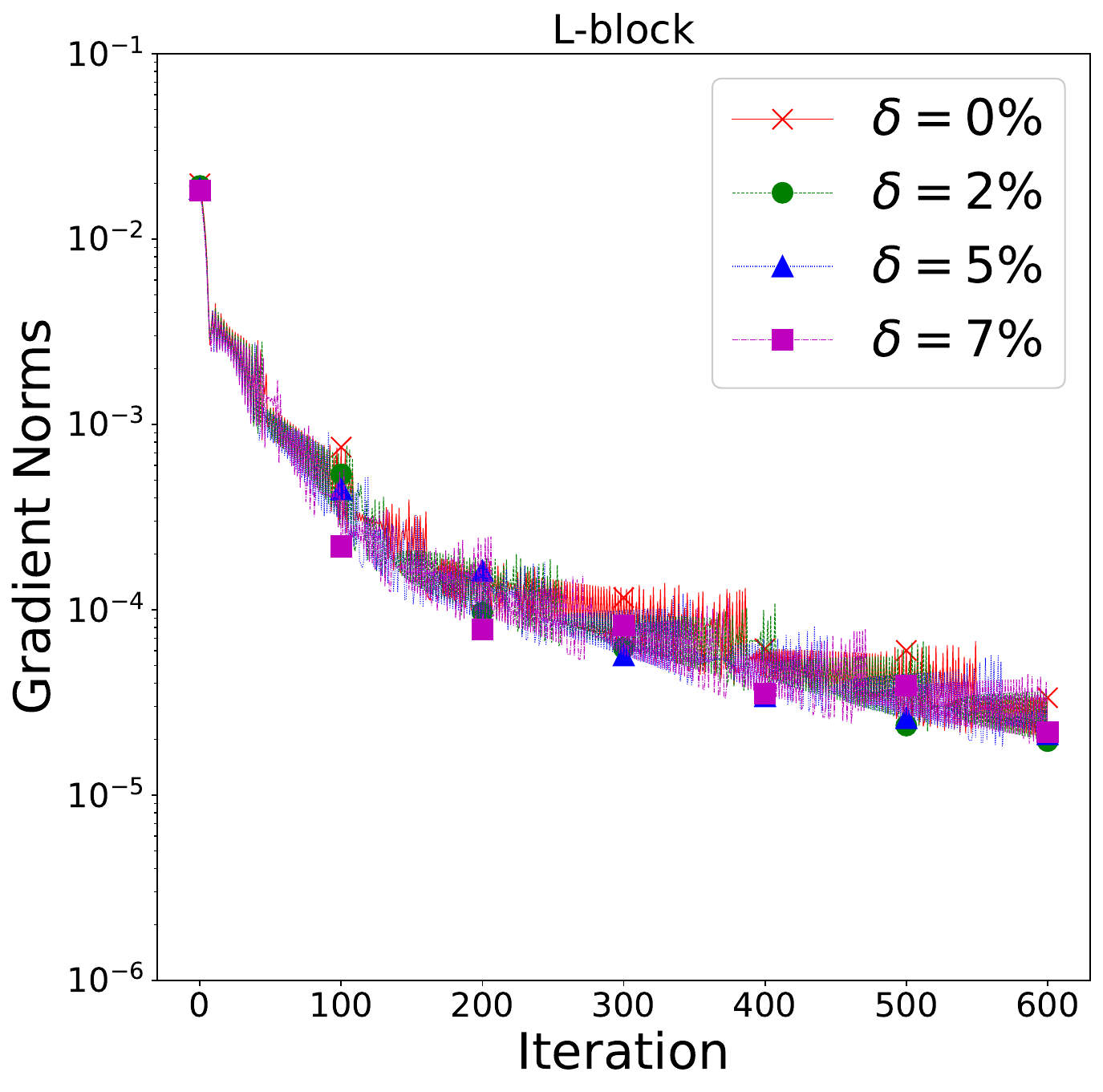}}\ 
\caption{Cost and gradient norm histories for the L-block case: left, constant $\sigma \equiv 1$; right, spatially varying $\sigma(x) = 1.1 + \sin(\pi x_{1})\sin(\pi x_{2})$}
\label{fig:costs_2d}
\end{figure}
%
%
%
%
\subsection{Alternating Direction Method of Multipliers in shape optimization setting}\label{subsec:ADMM_algorithm}
Previously, we highlighted the difficulty of recovering concavities in non-convex obstacles due to noise and spatially varying diffusion and drift. To address this, we introduce an ADMM-based modification, applied to the 3D setting in Subsection~\ref{subsec:Numerical_Examples_3D}. Accordingly, we reformulate Problem~\ref{prob:shape_optimization} as follows:

\begin{problem}\label{prob:optimal_shape_problem}
Given constants $a$ and $b$ with $b \geqslant a$, find the shape $\omega^{\ast}$ in the admissible set
\[
	\ADMMset = \Bigl\{ \omega \in \smallOad \mid a \leqslant \realu \leqslant b \ \text{a.e. in } \varOmega \Bigr\}
\]
such that
\begin{equation}\label{eq:control} 
	\omega^{\ast} = \argmin_{\omega \in {\ADMMset}} J({\varOmega})
	:= \argmin_{\omega \in {\ADMMset}} \frac{1}{2} \intO{ \abs{\imagu}^{2} } .
\end{equation}	
\end{problem}

Note that identifying $\omega$ is equivalent to identifying $\varOmega$, since $\varOmega = D \setminus \overline{\omega}$, so defining one automatically determines the other.

To directly incorporate the inequality constraint in the cost function, we introduce the auxiliary variable $v$ which satisfies $v = \realu$ a.e. in ${\varOmega}$ and consider the set $\mathcal{G}$ defined as follows:
\[
	\mathcal{G} = \left\{ (\omega,v) \in \ADMMset \times L^{2}({\varOmega}) \mid \text{$\realu = v$ a.e. in ${\varOmega}$} \right\}.
\]

Then, we can rewrite \eqref{eq:control} as
\begin{equation}\label{eq:control_Uad}
	(\omega^{\ast}, v^{\ast}) = \argmin_{(\omega,v) \in \mathcal{G}} \left\{ J({\varOmega}) + U_{\mathcal{K}}(v) \right\},
\end{equation} 
where the set $\mathcal{K}$ is the closed, convex, non-empty set of $L^{2}({\varOmega})$ defined by
\[
	\mathcal{K} = \left\{ v \in L^{2}({\varOmega}) \mid a \leqslant v \leqslant b \ \ \text{a.e. in ${\varOmega}$} \right\},
\]
and $U_{\mathcal{K}}$ is the indicator functional of the set $\mathcal{K}$; that is, $U_{\mathcal{K}}(v) = 0$ if $v \in {\mathcal{K}}$, and $U_{\mathcal{K}}(v) = \infty$ if $v \in L^{2}({\varOmega}) \setminus \mathcal{K}$.

To apply ADMM to the control model \eqref{eq:control_Uad}, we first define the augmented Lagrangian functional.
This is possible since the minimum of problem \eqref{eq:control_Uad} corresponds to the saddle point of the following augmented Lagrangian functional
\begin{equation}\label{eq:augmented_lagrangian}
	{L}_{\beta}(\omega, v; \lambda) = J({\varOmega}) + U_{\mathcal{K}}(v) + \frac{\beta}{2} \intO{\vert \realu - v \vert^2} + \intO{\lambda (\realu - v)},
\end{equation}
where $\lambda$ is the Lagrange multiplier and $\beta > 0$ is a penalty parameter.
Following~\cite{CherratAfraitesRabago2025b,RabagoHadriAfraitesHendyZaky2024}, we fix the penalty parameter $\beta$. 
Although $\beta$ could be optimized via a bilevel approach~\cite{Dempe2020}, keeping it fixed simplifies the method and yields consistently good numerical results.

We find a saddle point of ${L}_{\beta}$ using an iterative ADMM procedure (Algorithm \ref{algo:ADMM_algorithm}).
Starting from initial values $\omega^{0} \in \ADMMset$ and $v^{0}, \lambda^{0} \in L^{2}({\varOmega})$, the algorithm generates a sequence of iterates $(\omega^{k}, v^{k}, \lambda^{k})$, for $k \in \mathbb{N}$, by solving the following minimization problems sequentially and alternately:
\begin{align}
    \omega^{k+1}    &= \argmin_{\omega \in \ADMMset} {L}_{\beta}(\omega, v^{k}; \lambda^{k}); \label{eq:controle}\tag{SP1} \\
    v^{k+1}         &= \argmin_{v \in L^{2}({\varOmega})} {L}_{\beta}(\omega^{k+1}, v; \lambda^{k}); \label{eq:etat}\tag{SP2} \\
    \lambda^{k+1}   &= \lambda^{k} + \beta (\realu^{k+1} - v^{k+1}), \qquad \text{where } \realu^{k+1} := \realu(\varOmega^{k+1}). \label{eq:parametre1}\tag{SP3} 
\end{align} 
 
\begin{algorithm}
\begin{enumerate}\itemsep0.1em 
	\item \textit{Input} Fix $\beta$, $a$, and $b$, and define the Cauchy pair $(f, g)$.
	\item \textit{Initialization} Choose an initial shape $\omega^{0}$ and set the initial values $v^{0}$ and $\lambda^{0}$.
	\item \textit{Iteration} For $k = 1, 2, \ldots$, compute $(\omega^{k}, v^{k}, \lambda^{k})$ using equations \eqref{eq:controle}--\eqref{eq:parametre1} through sequential computations:
	$
		\{v^{k}, \lambda^{k}\} \ \stackrel{\eqref{eq:controle}}{\longrightarrow} \ \omega^{k+1} 
		\ \stackrel{\eqref{eq:etat}}{\longrightarrow} \ v^{k+1}
		\ \stackrel{\eqref{eq:parametre1}}{\longrightarrow} \ \lambda^{k+1}.
	$
	\item \textit{Stop Test} Repeat \textit{Iteration} until convergence.
\end{enumerate}
\caption{ADMM algorithm for the solution of problem \eqref{eq:control}.}
\label{algo:ADMM_algorithm}
\end{algorithm}

The resolution of \eqref{eq:controle} and \eqref{eq:etat} is outlined in the next two subsections.
\subsubsection{Solution of $\omega$-subproblem \eqref{eq:controle}}
We first consider equation~\eqref{eq:controle}, which minimizes ${L}_{\beta}$ with respect to $\omega$:
\[
 \omega^{k+1} = \argmin_{\omega \in \ADMMset} \left\{ J({\varOmega}) + U_{\mathcal{K}}(v^{k}) + \frac{\beta}{2} \intO{\vert \realu - v^{k} \vert^2} + \intO{\lambda^{k} (\realu - v^{k})} \right\}.
\]

Consider the following shape functional:
\[
\resizebox{\textwidth}{!}{$
Y^{k}({\varOmega}) := {L}_{\beta}(\omega, v^{k}; \lambda^{k}) = \frac{1}{2} \intO{\vert \imagu \vert^2} + \frac{\beta}{2} \intO{\vert \realu - v^{k} \vert^2} + \intO{\lambda^{k} (\realu - v^{k})}.
$}
\]

{The $\omega$-subproblem \eqref{eq:controle} requires the shape derivative of $Y^k(\varOmega)$, stated below.}
\begin{proposition}
	\label{prop:shape_gradient_of_Y}
	Let $\varOmega \in \Oad$ and $\VV \in \sfTheta$.
	Then, $Y^{k}(\varOmega)$ is shape differentiable, and its shape derivative at $\varOmega$ in the direction $\VV$ is
	\begin{equation} 
		d Y^{k}(\varOmega)[\VV] 
		= \intG{\left( \sigma ( \dn{\realq} \dn{\imagu} 
		- \dn{\imagq} \dn{\realu} ) + \frac{\beta}{2} (v^{k})^2 - \lambda^{k} v^{k} \right) \Vn },
		\label{eq:shape_gradient_of_Y} 	
	\end{equation} 
	where $u = \realu + i \imagu$ satisfies \eqref{eq:complex_PDE} and the adjoint variable $q = \realq + i \imagq \in \HH^{2}(\varOmega) \cap \spaceV(\varOmega)$ uniquely solves the adjoint system
	\begin{equation} \label{eq:adjoint_problem:q}
	\left\{
	\begin{array}{rcll}
		\div{\sigma \nabla q} + \vect{b} \cdot \nabla q + (\divbb) q & = & i(\beta (\realu - v^{k}) + \lambda^{k}) - \imagu & \text{in } \varOmega, \\
		q & = & 0 & \text{on } \varGamma, \\
		\sigma \dn{q} + q \vect{b} \cdot \nn - i q & = & 0 & \text{on } \varSigma.
	\end{array}
	\right.
	\end{equation}	 
\end{proposition}
{Expression \eqref{eq:shape_gradient_of_Y} can be obtained similarly to the proof of Proposition~\ref{prop:shape_gradients}. Alternatively, a standard chain rule argument can be applied, using the material derivative approach, the identity \eqref{eq:divergence_identity}, and the adjoint method via the adjoint system \eqref{eq:adjoint_problem:q} to simplify the resulting terms. For brevity, the full derivation is omitted.}
\begin{remark}
By considering an adjoint that satisfies
\begin{equation}
\left\{
\begin{array}{rcll}
	\div{\sigma \nabla q} + \vect{b} \cdot \nabla q + (\divbb)q & = & (\beta (\realu - v^k) + \lambda^k) +i \imagu & \text{in } \varOmega, \\
	q & = & 0 & \text{on } \varGamma, \\
\sigma \dn{q} + q\,\vect{b} \cdot \nn - i q & = & 0 & \text{on } \varSigma,
\end{array}
\right.
\end{equation}
we obtain the equivalent form of the shape derivative $dY^k(\varOmega)[\VV]$:
\[
dY^k(\varOmega)[\VV] 
	= \intG{\left( -\sigma ( \dn{\realq }\dn{\realu } +\dn{\imagq }\dn{\imagu })   
			+ \frac{\beta}{2} (v^k)^2 - \lambda^k v^k \right) \Vn }.
\]
\end{remark}
The preceding formulation of the shape gradients and their adjoint systems yields an expression that avoids material derivatives and provides a natural choice. 
Here, we propose an alternative formulation using a different set of adjoint variables: one for the derivative of $\imagu$ and one for $\realu$. 
The resulting shape gradient of $Y^{k}$ is given in the following proposition:
\begin{proposition}[Another shape gradient structure of $Y^{k}(\varOmega)$]
	\label{prop:another_shape_gradient_of_Y}
	Let $\varOmega \in \Oad$ and $\VV \in \sfTheta$.
	Then, $Y^{k}(\varOmega)$ is shape differentiable, and its shape derivative at $\varOmega$ in the direction $\VV$ is
	\begin{equation}\label{eq:another_shape_gradient_of_Y} 	
	\begin{aligned}
		d Y^{k}(\varOmega)[\VV] 
			&= \intG{ G(u,p) \Vn }
			- \intG{ \sigma ( \dn{\realLambda} \dn{\realu}  + \dn{\imagLambda} \dn{\imagu} ) \Vn }\\
			&\qquad + \intG{ \left( \frac{\beta}{2} (v^{k})^2 - \lambda^{k} v^{k} \right) \Vn },
	\end{aligned}
	\end{equation} 
	where $G(u,p)$ is given in \eqref{eq:shape_gradient} while the adjoint variable ${\Lambda} = \realLambda + i \imagLambda \in \HH^{2}(\varOmega) \cap \spaceV(\varOmega)$ uniquely solves the adjoint system
	\begin{equation} \label{eq:adjoint_problem:Lambda}
	\left\{
	\begin{array}{rcll}
		\div{\sigma \nabla {\Lambda}} + \vect{b} \cdot \nabla {\Lambda} + (\divbb) {\Lambda} & = & \beta (\realu - v^{k}) + \lambda^{k} & \text{in } \varOmega, \\
		{\Lambda} & = & 0 & \text{on } \varGamma, \\
		\sigma \dn{{\Lambda}} + {\Lambda} \vect{b} \cdot \nn - i {\Lambda} & = & 0 & \text{on } \varSigma.
	\end{array}
	\right.
	\end{equation}	 
\end{proposition}
{The proof is omitted as it follows standard arguments.}

From \eqref{eq:another_shape_gradient_of_Y}, the first integral is independent of the adjoint $\Lambda$, while the second is independent of $p$. This motivates the use of “partial” shape gradients \cite{HIKKP2009}, originally applied to the exterior Bernoulli free boundary problem. Accordingly, we define four shape gradients (Table~\ref{tab:radii_and_beta}) for computing deformations in our numerical scheme (Algorithm~\ref{algo:CCBM--ADMM--SGBD}). Their benefit is demonstrated numerically in Subsection~\ref{subsec:effect_of_the_choice_of_shape_gradient}, highlighting the advantage of selectively applying the adjoint method and using only part of the exact gradient.

%
\subsubsection{Extension and regularization of the deformation field}
The shape gradient of $Y$, as for $J$, is supported only on the boundary $\varGamma$ and may be insufficiently smooth for finite element implementation. 
To regularize the descent vector over $\varOmega$, we employ the same extension--regularization technique introduced in Section~\ref{subsec:conventional_scheme} and apply the Sobolev gradient-based descent (SGBD) algorithm (Algorithm~\ref{algo:SGBD_algorithm}) to solve \eqref{eq:controle}.

\begin{algorithm} 
\caption{SGBD algorithm for the $\omega$-subproblem \eqref{eq:controle}}\label{algo:SGBD_algorithm} 
\begin{enumerate}
 \item \textit{Input:} Fix $c_{b}$, $\mu$, $\beta$, $a$, $b$, and $\varepsilon$. Set $\lambda^{k}$, and initialize $\varOmega^{k}_{0} = \varOmega^{k}$, $u^{k}_{0} = u^{k}$, $v^{k}_{0} = v^{k}$. 	
 \item \textit{Iteration:} For $m = 0, 1, 2, \ldots$:
 \begin{enumerate}\itemsep0.1em  
  \item[2.1] Solve \eqref{eq:weak_form_of_state} and the adjoint problem \eqref{eq:adjoint_problem:q} on the current domain $\varOmega^{k}_{m}$.  
  \item[2.2] Compute $\VV_{m}^{k} = \VV$, where $\VV \in H^{1}_{\varSigma,0}(\varOmega^{k}_{m})^d$ solves \eqref{eq:Sobolev_gradient_computation} with $G$ replaced by the chosen shape gradient ($G_q^k$, $G_\sharp^k$, $G_p^k$, or $G_\Lambda^k$).  
  \item[2.3] Set $t^{k} = \mu J^{k}(\varOmega^{k}_{m}) / \|\VV_{m}^{k}\|_{H^{1}(\varOmega^{k}_{m})^d}$ and update $\varOmega^{k}_{m+1} = \{ x + t^{k} \VV_{m}^{k}(x) \mid x \in \varOmega^{k}_{m} \}$.
 \end{enumerate}
 \item \textit{Stopping criterion:} Repeat the iteration while $\| dY^{k}(\varOmega^{k}_{m})[\VV_{m}^{k}] \| \geqslant \varepsilon$.  
 \item \textit{Output:} $\varOmega^{k+1} = \varOmega^{k}_{m+1}$.
\end{enumerate}
\end{algorithm}

In Step 2.3 of Algorithm~\ref{algo:SGBD_algorithm}, the step size $t^{k}$ is computed using the same strategy in Section~\ref{subsec:conventional_scheme}.
%
%
%
%
%
\subsubsection{Solution of the $v$-subproblem \eqref{eq:etat}}
We now consider equation~\eqref{eq:etat}, which minimizes ${L}_{\beta}$ with respect to $v$:
\[
\resizebox{\textwidth}{!}{$
\begin{aligned}
v^{k+1} &= \argmin_{v \in L^{2}({\varOmega})}
\Big\{ J(\varOmega^{k+1})+U_{\mathcal{K}}(v)
+ \frac{\beta}{2}  \intO{\vert {\realu^{k+1}}-v \vert^2} 
+ \intO{\lambda^{k}  ( {\realu^{k+1}}-v )} 
\Big\}.
\end{aligned}
$}
\]
Applying the projection method, we obtain $v^{k+1} = P_{\mathcal{K}}(\realu^{k+1} + \lambda^{k}/\beta)$, where $P_{\mathcal{K}}(\varphi) = \max(a, \min(b, \varphi))$ for all $\varphi \in L^{2}(\varOmega)$ denotes the projection onto the admissible set $\mathcal{K}$.
\subsubsection{CCBM--ADMM--SGBD algorithm}
Finally, we modify Algorithm~\ref{algo:ADMM_algorithm} to solve Problem~\ref{prob:optimal_shape_problem} with an inequality constraint on the real part of the state. Specifically, it is refined into the nested, iterative CCBM--ADMM--SGBD scheme for the optimal control problem \eqref{eq:control_Uad}, as detailed in Algorithm~\ref{algo:CCBM--ADMM--SGBD}:
\begin{algorithm}[htp!] 
\caption{CCBM--ADMM--SGBD}\label{algo:CCBM--ADMM--SGBD}
\begin{enumerate} \itemsep0.1em 
 \item \textit{Initialization:} Specify $(f,g)$, fix $N \in \mathbb{N}$, $\beta$, $a$, $b$, $\mu$, $c_{b}$, and $\varepsilon$, and choose $\omega^{0}$, $v^{0}$, and $\lambda^{0}$.
\item \textit{Iteration:} For $k = 0, \ldots, N$:
     \begin{enumerate}\itemsep0.1em
      \item[2.1] Compute the state solution $u^{k}$ of \eqref{eq:weak_form_of_state} on the current domain and set $\realu^{k} = \Re\{u^{k}\}$, the solution of \eqref{eq:real_part} associated with $\omega^{k}$.
      \item[2.2] Solve the adjoint state equation(s) corresponding to the chosen shape gradient.
      \item[2.3] Update $\omega^{k+1}$ using the gradient-descent method in Algorithm~\ref{algo:SGBD_algorithm}.
      \item[2.4] Update $v^{k+1}$ as $\displaystyle v^{k+1} = \max \big( a, \min (\realu^{k+1} + \lambda^{k}/\beta, b) \big)$.
      \item[2.5] Update $\lambda^{k+1} = \lambda^{k} + \beta (\realu^{k+1} - v^{k+1})$.
     \end{enumerate}
 \item \textit{Stopping criterion:} Repeat the iteration until convergence.
\end{enumerate}
\end{algorithm}

Algorithm~\ref{algo:CCBM--ADMM--SGBD} naturally handles noisy data. In the following experiments, we use the CCBM–ADMM–SGBD scheme with the extension-regularization method from Subsection~\ref{subsec:conventional_scheme}, which suffices for accurate reconstructions under moderate noise. Additional regularization terms, such as perimeter, surface area, or volume, are unnecessary here.
%
%
%
\subsection{Effect of adjoint formulation and partial gradient selection}
\label{subsec:effect_of_the_choice_of_shape_gradient} 
We evaluate the CCBM--ADMM--SGBD (or CCBM--ADMM) scheme against the conventional CCBM approach using three exact obstacle geometries. 
Figure~\ref{fig:exact3d} presents the geometries, the initial guess, and the corresponding mesh profiles. 
Again, inverse crimes are avoided in the reconstruction, following the same strategy as before.
\begin{figure}[h!]
\centering
\resizebox{0.2\linewidth}{!}{\includegraphics{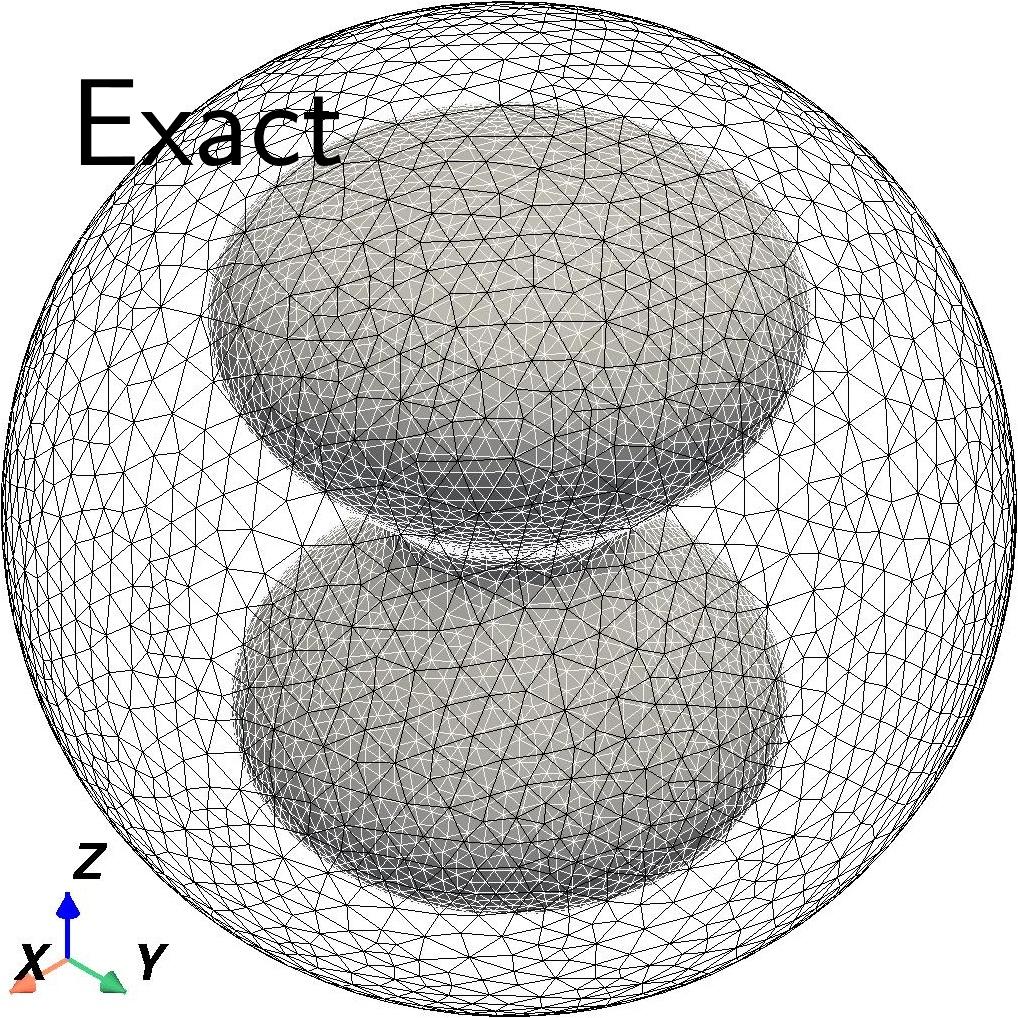}}\ \ 
\resizebox{0.2\linewidth}{!}{\includegraphics{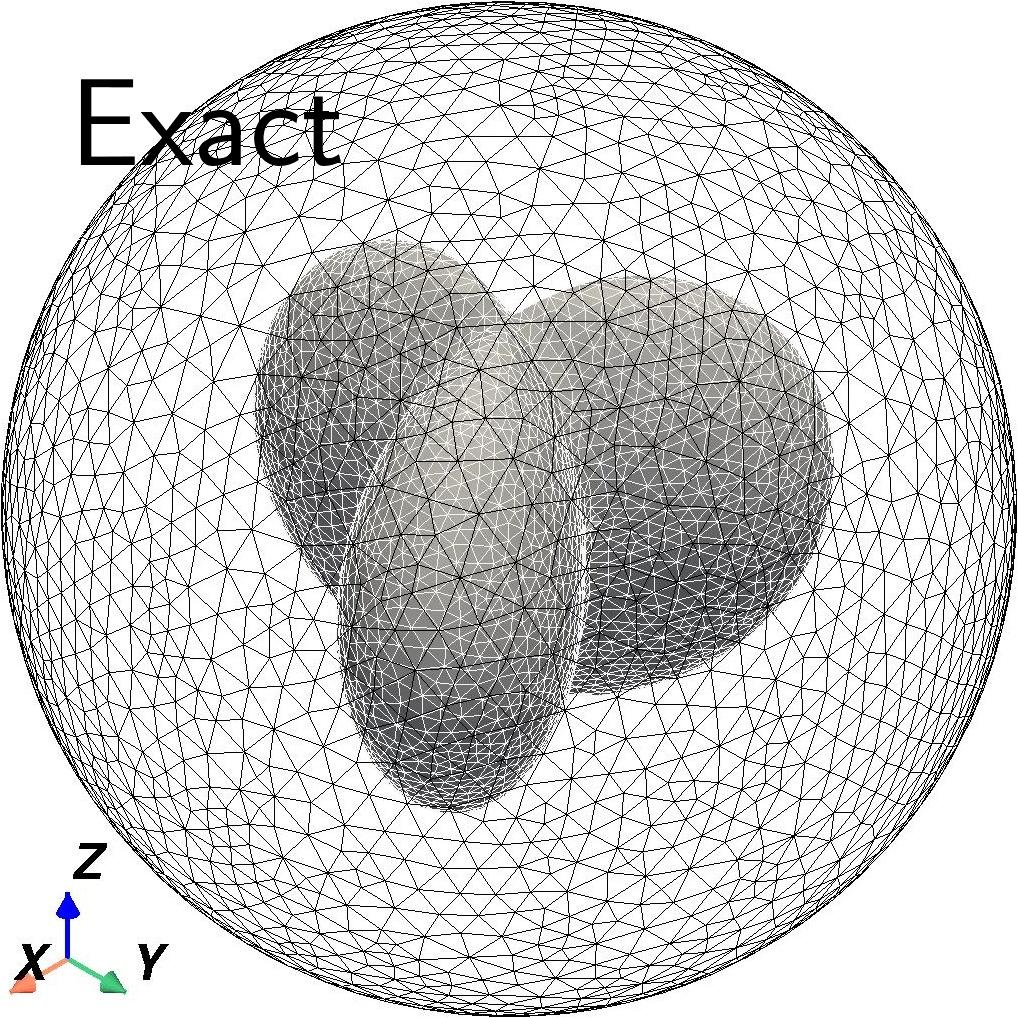}}\ \
\resizebox{0.2\linewidth}{!}{\includegraphics{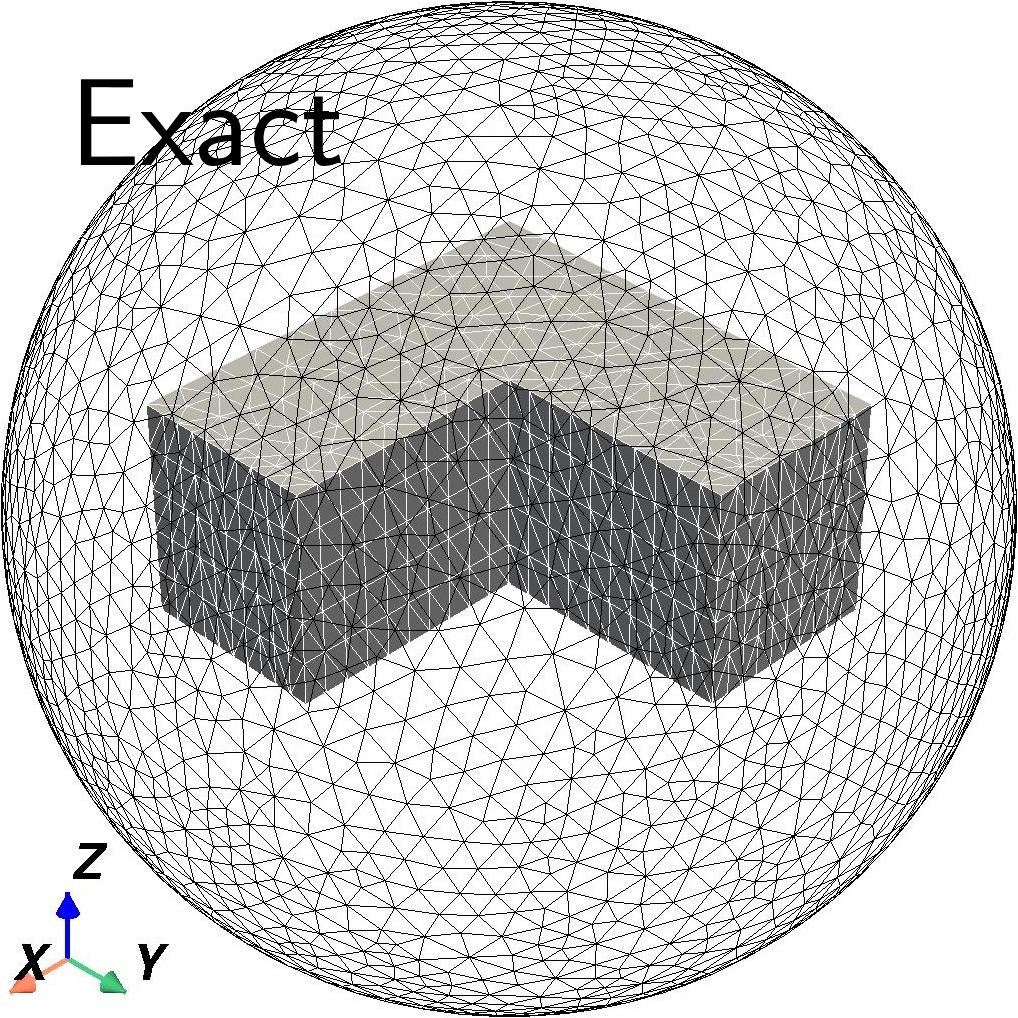}}\ \
\resizebox{0.2\linewidth}{!}{\includegraphics{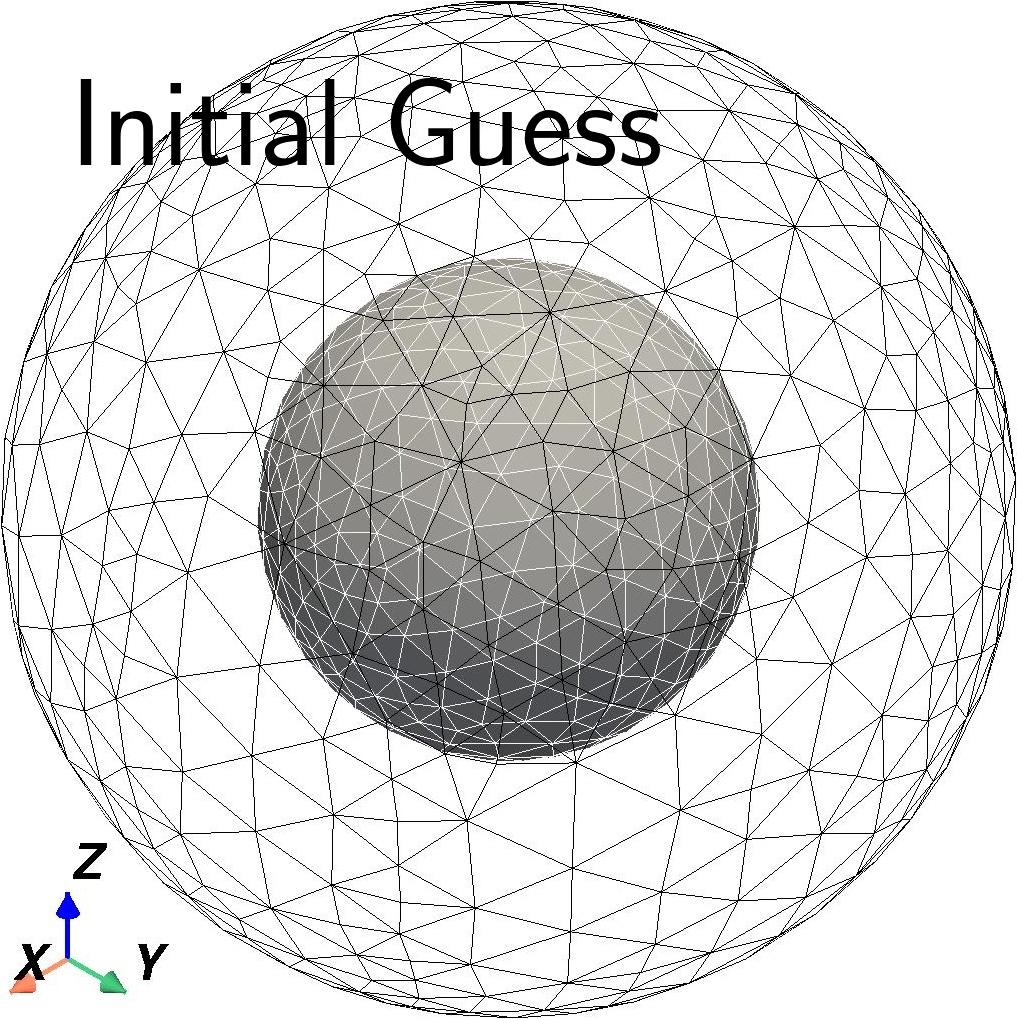}}
\caption{Geometry and mesh of the exact cavities (first three columns) and an initial guess with radius $r = 0.5$.}
\label{fig:exact3d}
\end{figure}

Before highlighting the advantages of the proposed CCBM--ADMM scheme over the conventional CCBM, we first assess the impact of using partial shape gradients, motivated by the adjoint method’s application.

We consider the unit ball $D = B_1(0) \subset \mathbb{R}^3$ and write $x=(x_1,x_2,x_3)$.  
The coefficients are defined as $\sigma(x) = 1.1 + \prod_{i=1}^3 \sin(\pi x_i)$ and $\mathbf{b}(x) = (1,1,1) + 0.5\,x$, with boundary data $g(x) = e^{|x|^2}$ prescribed on $\partial D$.  
For the preliminary tests, we use an L-block-shaped obstacle and set the parameters as follows:  
$N = 2400$, $\lambda^0 = 0.001$, $a = 0.5 \min u(\varOmega \setminus \overline{\omega}^{\star})$, $b = 1.5 \max u(\varOmega \setminus \overline{\omega}^{\star})$, $v^0 = 1$, $\varepsilon = 10^{-6}$, and $\omega^0 = B_{r}(0)$ with $r \in \{0.55, 0.575, 0.6\}$.  
For each method, the selected radii, chosen from the tested range, yield the most accurate reconstructions.
We use exact measurements to clearly observe the impact of the method employed in the CCBM--ADMM scheme, whose other specifics are summarized in Table~\ref{tab:radii_and_beta}.
\begin{table}[h!]
\centering
\renewcommand{\arraystretch}{1.5} 
\resizebox{\textwidth}{!}{%
\begin{tabular}{|c|l|c|c|c|}
\hline
 		& Kernel 			& $r$ 		& $\beta$ \\ \hline
Method 1 & $G_{q}^{k} := \sigma ( \dn{\realq} \dn{\imagu}  - \dn{\imagq} \dn{\realu} ) + \frac{\beta}{2} (v^{k})^2 - \lambda^{k} v^{k}$ 		& $0.600$ 	& $0.010$ \\ \hline
Method 2 & $G_{\sharp}^{k} := G(u,p) - \sigma ( \dn{\realLambda} \dn{\realu}  + \dn{\imagLambda} \dn{\imagu} ) 
					+ \frac{\beta}{2} (v^{k})^2 - \lambda^{k} v^{k}$ 	& $0.575$ 	& $0.010$ \\ \hline
Method 3 & $G_{p}^{k} := G(u,p) + \frac{\beta}{2} (v^{k})^2 - \lambda^{k} v^{k}$ 		& $0.550$		& $0.005$ \\ \hline
Method 4 & $G_{\Lambda}^{k} := - \sigma ( \dn{\realLambda} \dn{\realu}  + \dn{\imagLambda} \dn{\imagu} ) 
					+ \frac{\beta}{2} (v^{k})^2 - \lambda^{k} v^{k}$ & $0.550$ 	& $0.001$ \\ \hline
\end{tabular}%
}
\renewcommand{\arraystretch}{1} 
\caption{Radii and choices of $\beta$ for each of the methods used.}
\label{tab:radii_and_beta}
\end{table}

%
%
%
%
%
\begin{figure}[h!]
\centering
\resizebox{0.2\linewidth}{!}{\includegraphics{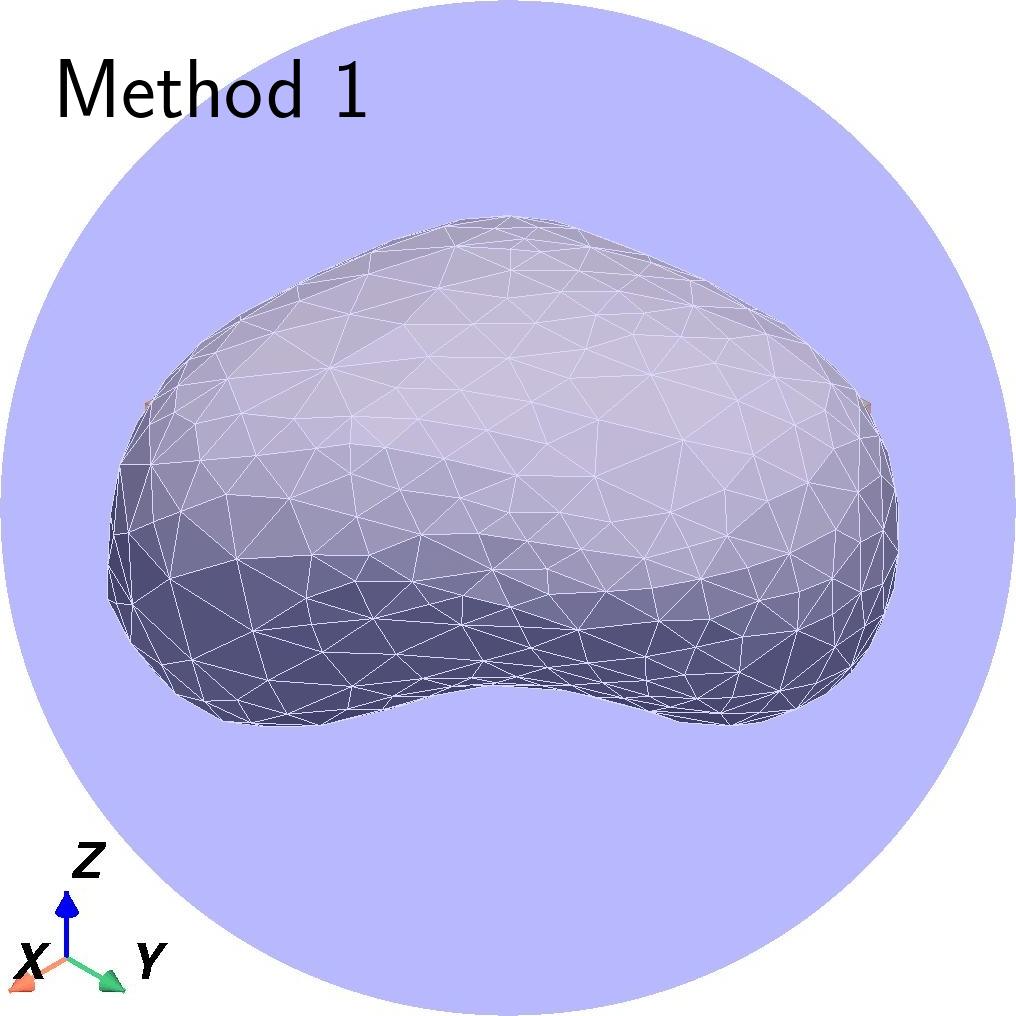}}\ 
\resizebox{0.2\linewidth}{!}{\includegraphics{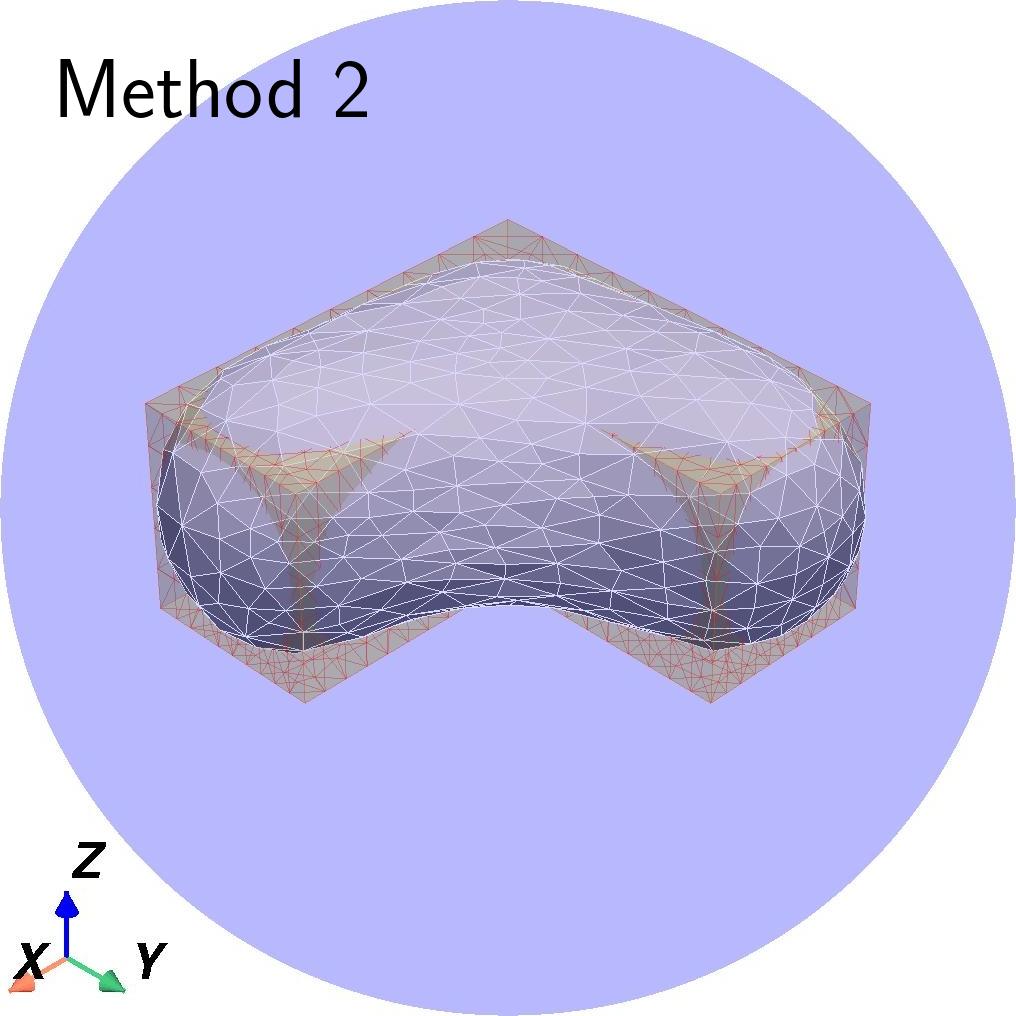}}\ 
\resizebox{0.2\linewidth}{!}{\includegraphics{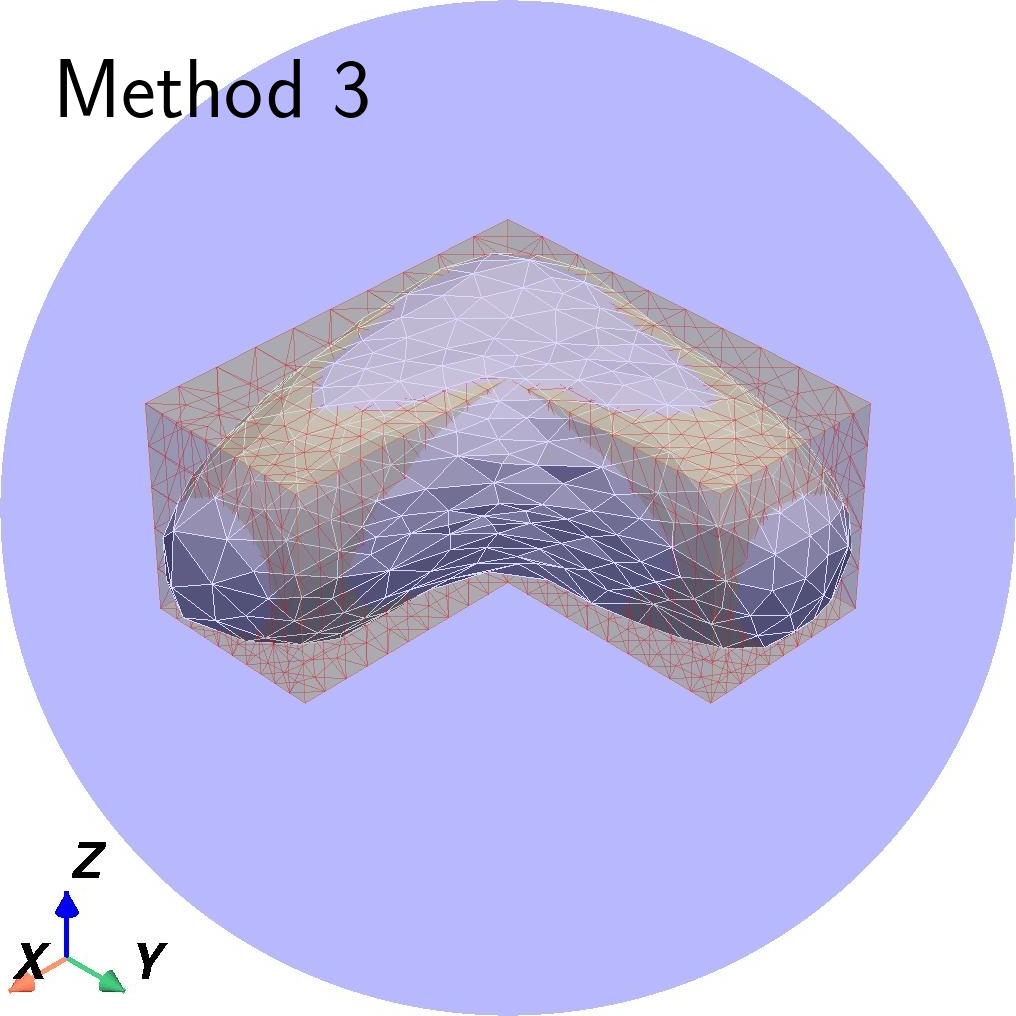}}\ 
\resizebox{0.2\linewidth}{!}{\includegraphics{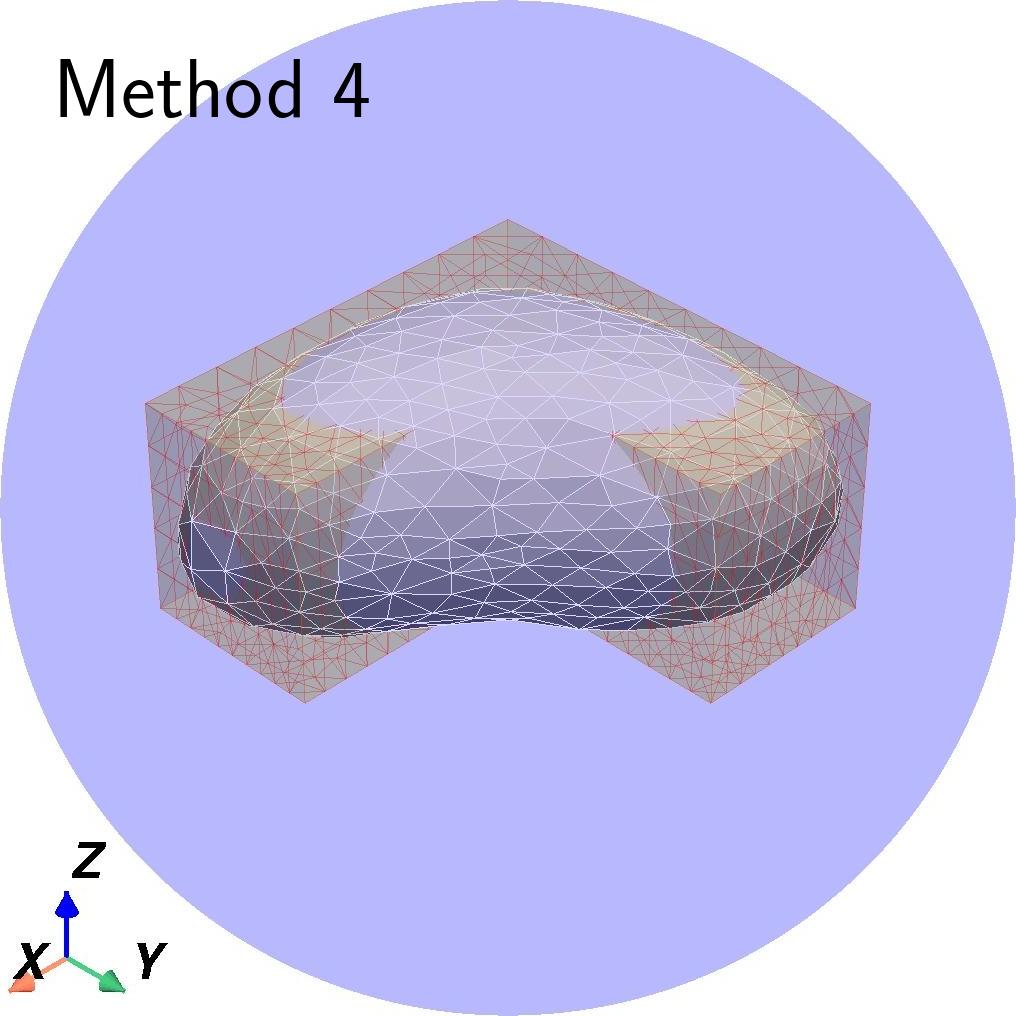}}\\[0.5em]
\resizebox{0.2\linewidth}{!}{\includegraphics{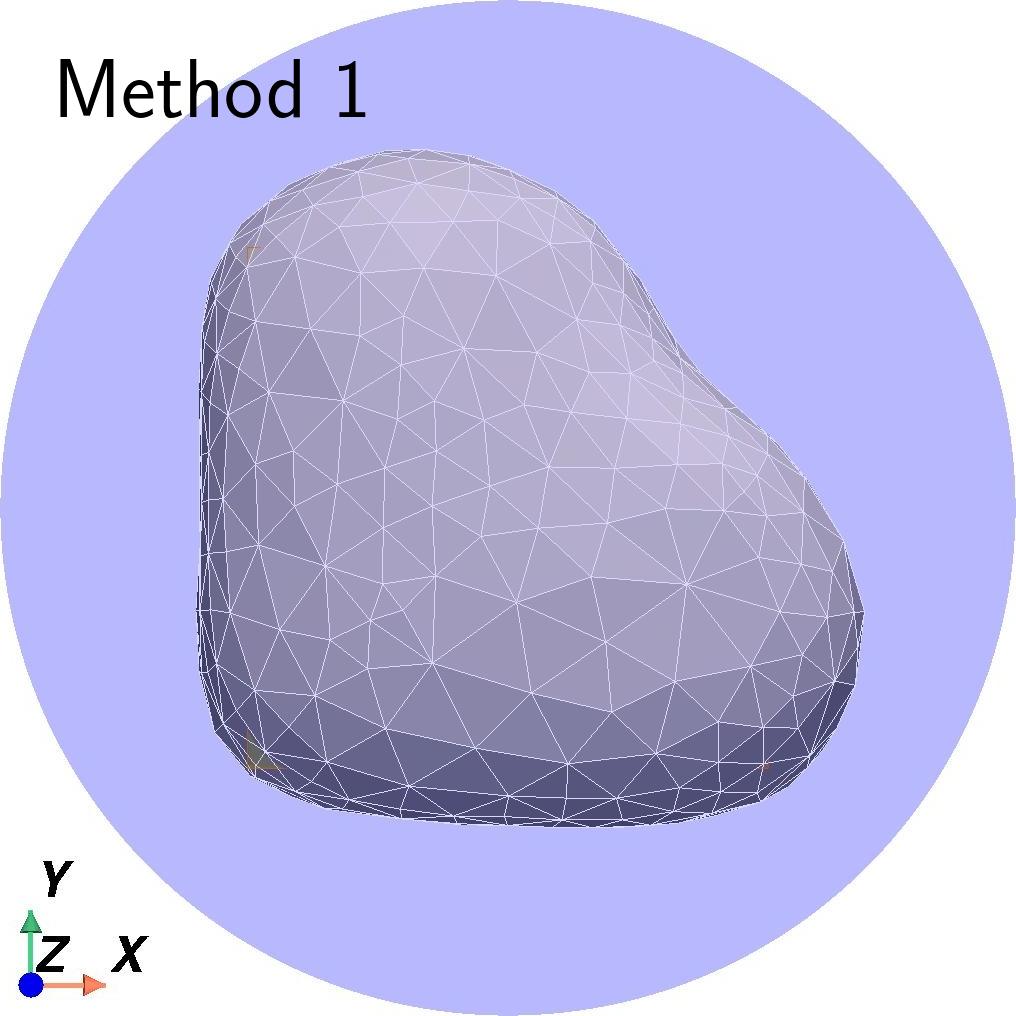}}\ 
\resizebox{0.2\linewidth}{!}{\includegraphics{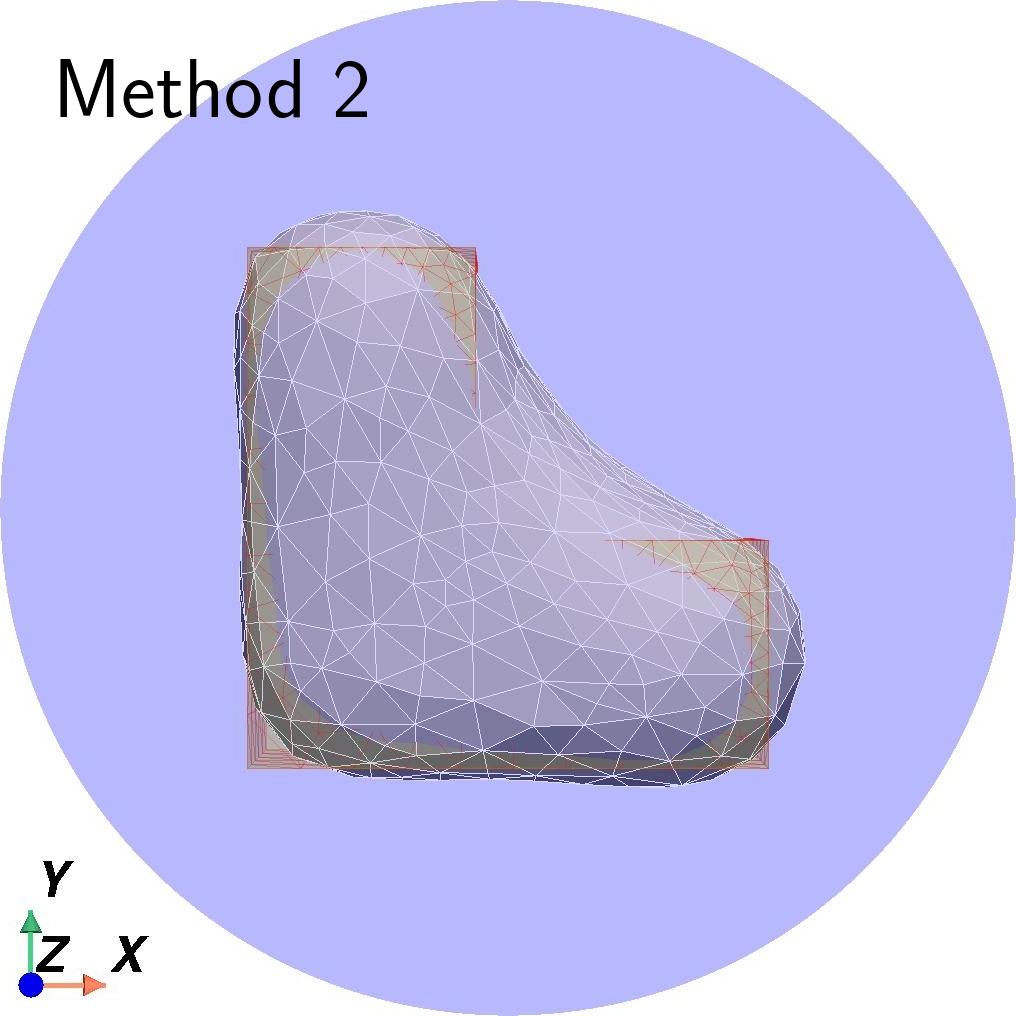}}\ 
\resizebox{0.2\linewidth}{!}{\includegraphics{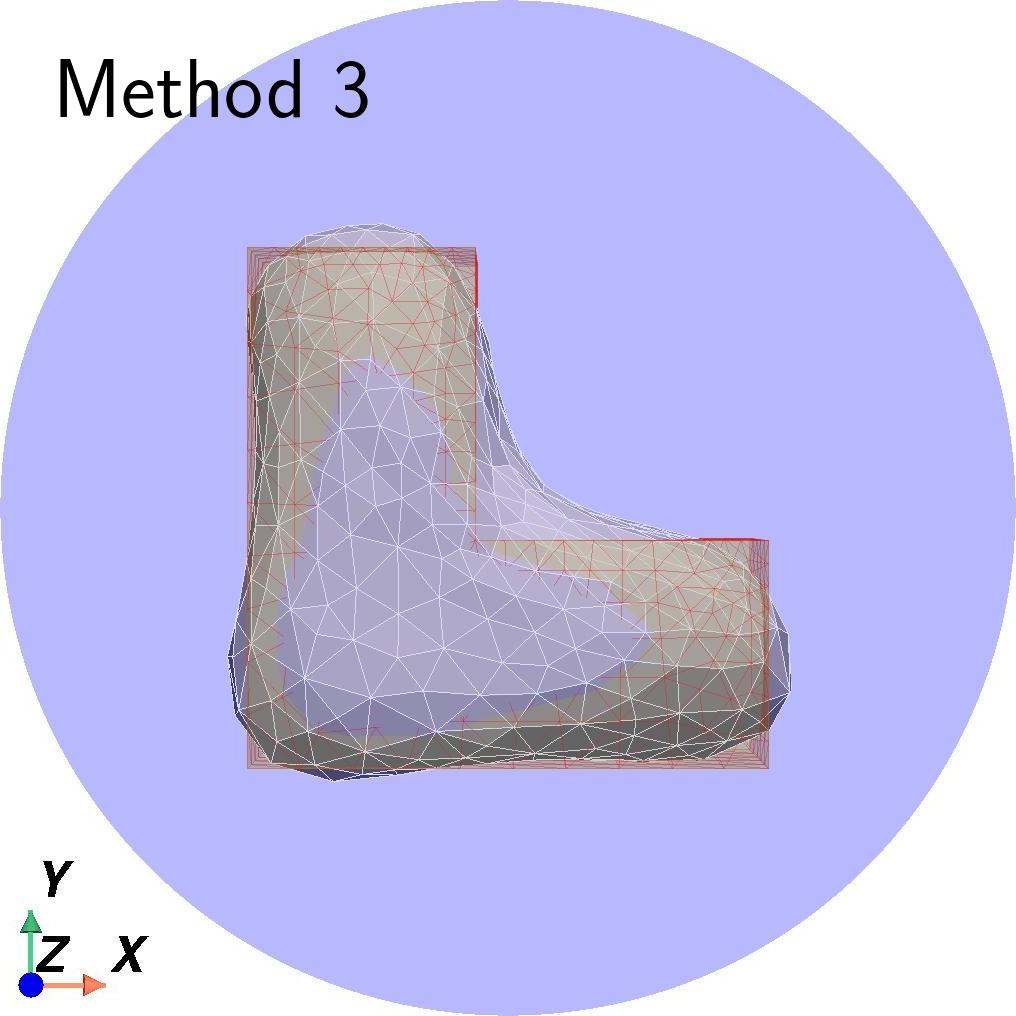}}\ 
\resizebox{0.2\linewidth}{!}{\includegraphics{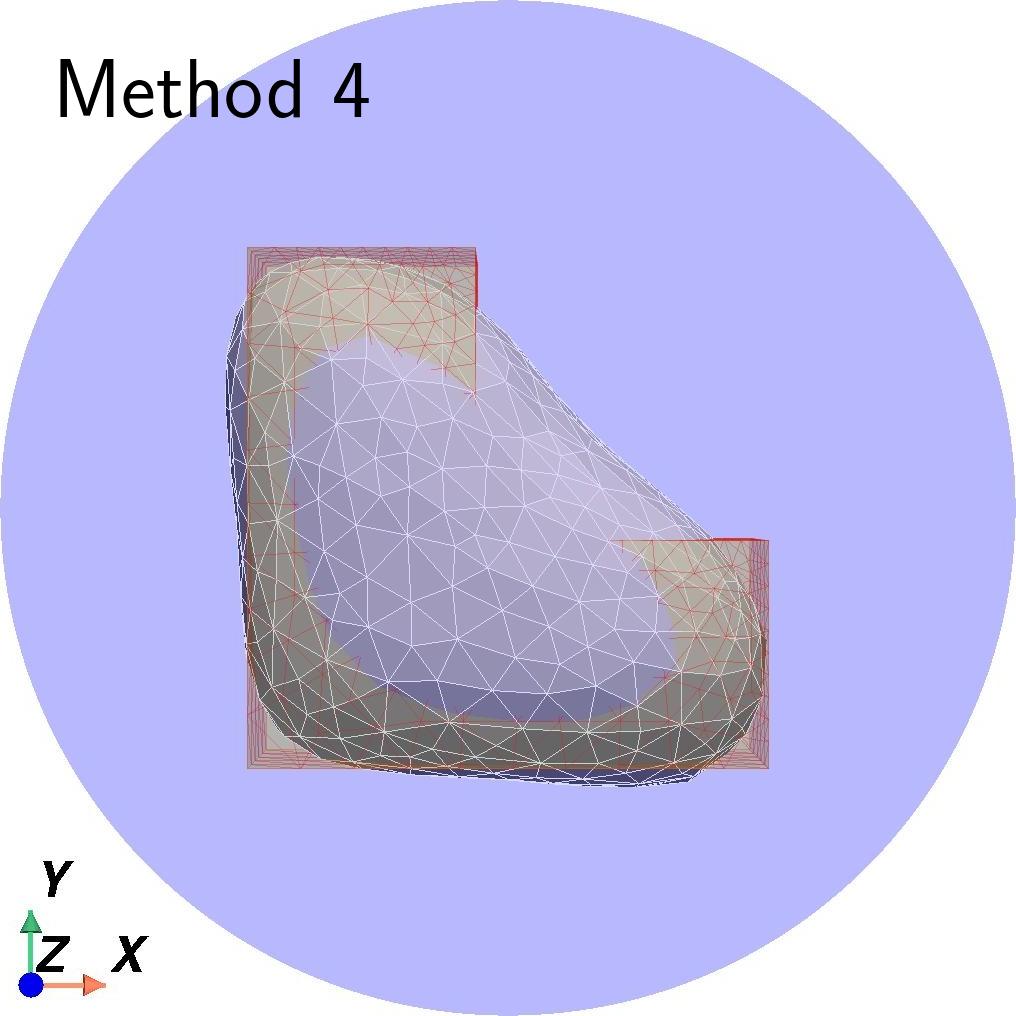}}
\caption{Impact of adjoint and gradient selection on reconstructed shape}
\label{fig:preliminary_tests}
\end{figure}
%
%
%
%
%
%
%
%
%
\begin{figure}
\centering
\resizebox{0.3\linewidth}{!}{\includegraphics{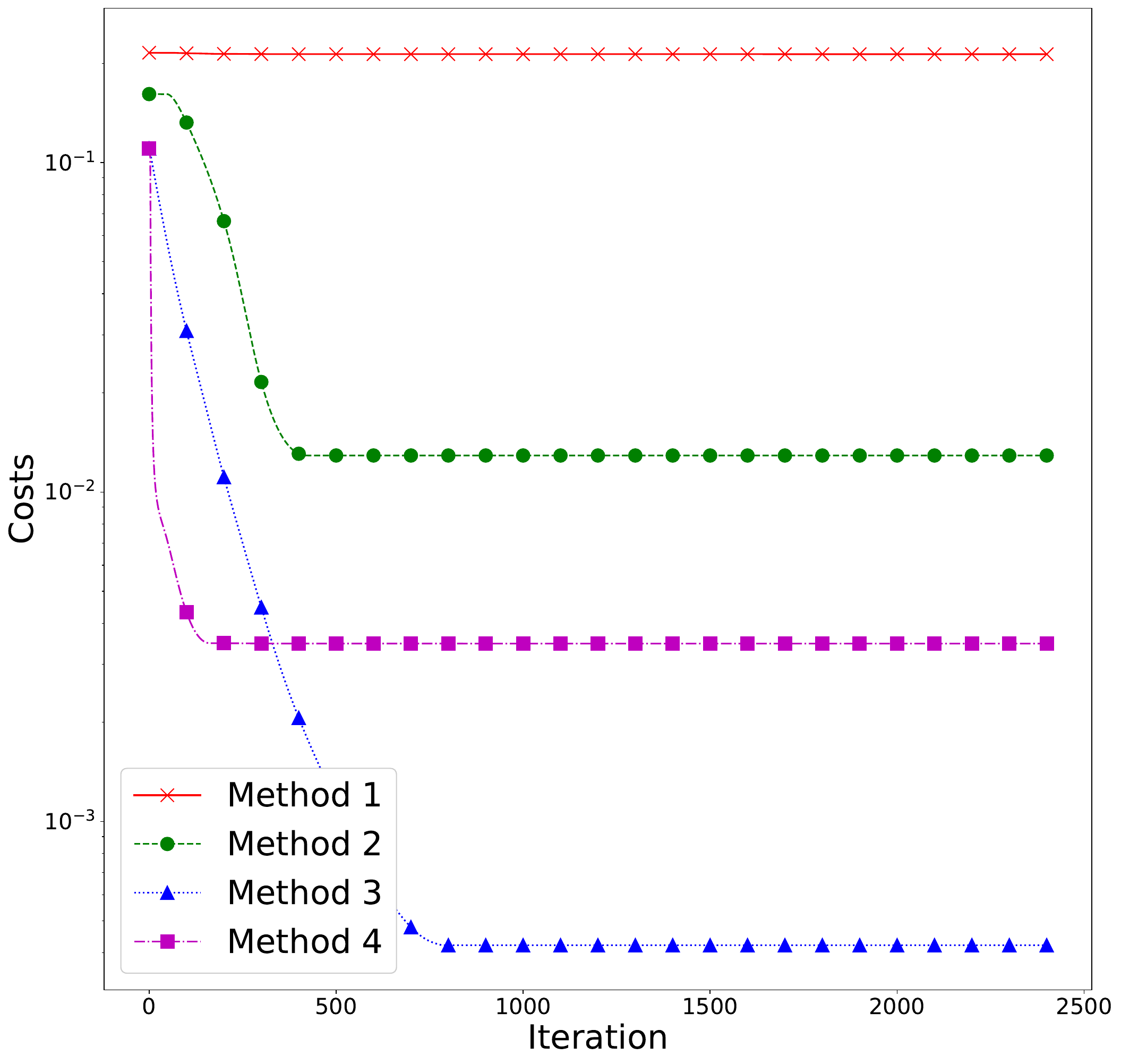}}\quad
\resizebox{0.3\linewidth}{!}{\includegraphics{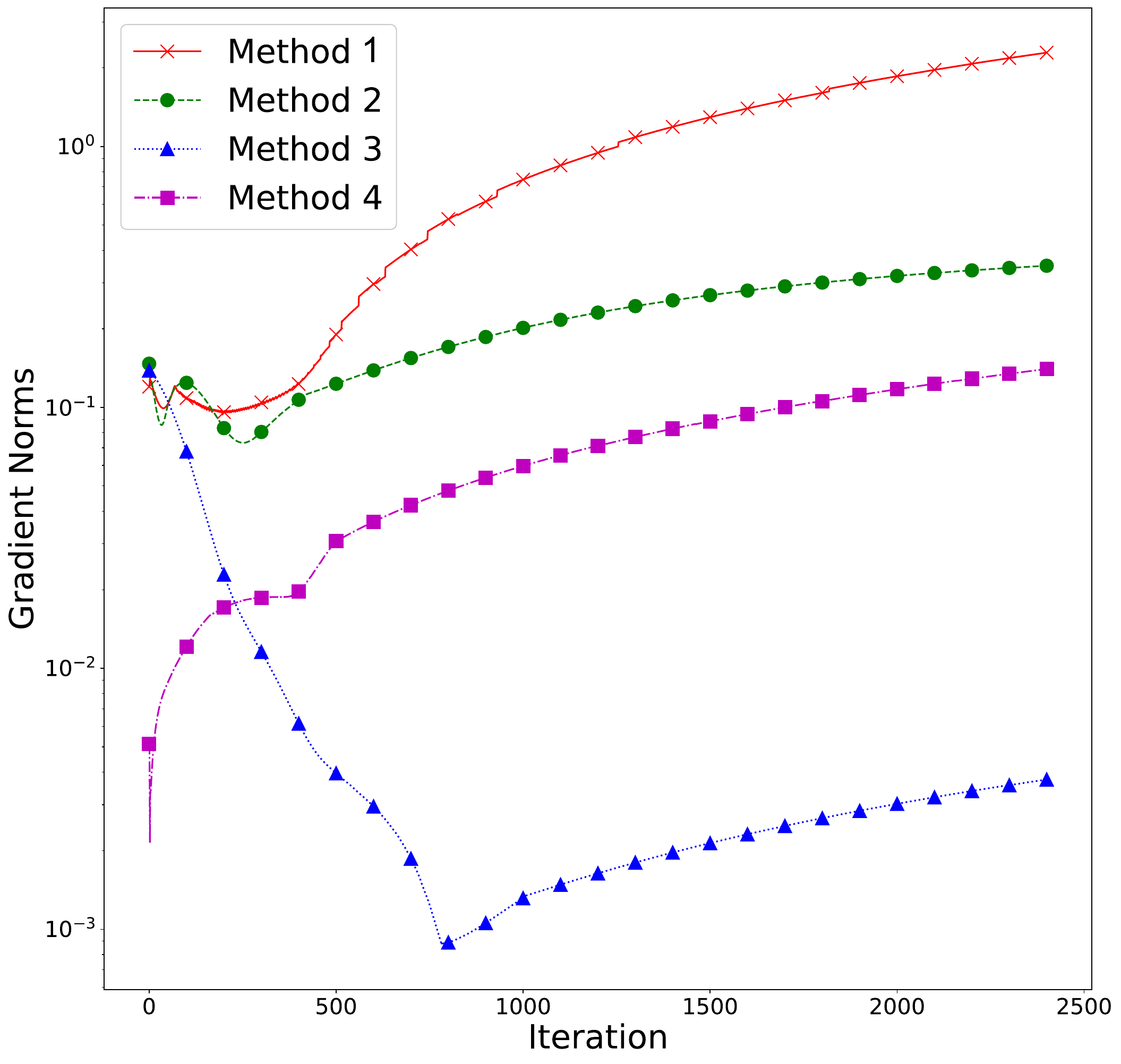}}
\caption{Histories of cost and gradient norms for shapes in Fig.~\ref{fig:preliminary_tests}}
\label{fig:preliminary_tests_costs_and_norms}
\end{figure}

The computational results for the methods listed in Table~\ref{tab:radii_and_beta} are shown in Figures~\ref{fig:preliminary_tests} and \ref{fig:preliminary_tests_costs_and_norms}. 
Method~3, which employs the kernel $G^{k}_{p}$, delivers the most accurate obstacle reconstruction and the lowest cost functional value. 
While Methods~2 and 3 both capture concave regions effectively, Method~3 outperforms all others overall. 
In contrast, Method~1 is less sensitive to geometric features, leading to less precise reconstructions.

Method~3 also combines superior accuracy with favorable computational efficiency. 
For $2400$ optimization iterations on a mesh with $1377$ vertices, $4524$ triangular elements, and $2358$ boundary elements (degrees of freedom equal to the number of vertices), the CPU times for Methods~1--4 are $4267$s, $4446$s, $3041$s, and $4187$s, respectively. 
Thus, Method~3 is faster than Methods~1 and~4, while Method~2 is the most time-consuming.  

Based on these results, we select Method~3 for further numerical experiments to benchmark the CCBM--ADMM scheme against the conventional approach with the exact kernel $G$ in the next subsection.

\subsection{Numerical examples in 3D}
\label{subsec:Numerical_Examples_3D} 
We now consider three-dimensional examples with slightly modified space-dependent coefficients to illustrate the advantages of the ADMM framework and the influence of boundary data on the reconstruction.

The specimen is again the unit ball $D = B_1(0) \subset \mathbb{R}^3$; however, we now prescribe  
$\sigma(x) = 1.1 + \sin(\pi x)\sin(\pi y)$ and define  
\[
\vect{b}(x) = \left(1.0 + 0.5 \sin\left(\arctan\left(\frac{x_2}{x_1}\right)\right),\; 1.0 + 0.5 \cos\left(\arctan\left(\frac{x_2}{x_1}\right)\right),\; 1.5\right)^{\top},
\]
where $x = (x_1, x_2, x_3) \in D \subset \mathbb{R}^3$.  
The boundary data are synthetically generated as $g(x) = \exp(x_1^2 + x_2^2)$ for $x \in \partial D$.

We compare the proposed CCBM--ADMM scheme with the conventional CCBM-based shape optimization using the three exact obstacle geometries shown in Figure~\ref{fig:exact3d}. 
The computational parameters used in the inversion procedure are identical to those in the previous section, except that $N=1200$ and the initial domain is $\omega^0 = B_r(0)$ with $r \in \{0.5, 0.575\}$, selected via preliminary tests to achieve the best reconstructions.
Figures~\ref{fig:exact_measurement} and~\ref{fig:noisy_measurement} compare reconstructions obtained with the conventional CCBM and CCBM--ADMM. 
The conventional method performs poorly, even with exact data and more so under moderate noise ($\delta = 5\%$), whereas CCBM–ADMM yields stable and satisfactory reconstructions, particularly for non-convex shapes and deep concavities.
Moreover, the proposed scheme remains effective to a Lipschitz-smooth obstacle, despite a mild violation of the underlying smoothness assumptions.

Figure~\ref{fig:cost_and_gradient} shows the evolution of the cost functional and gradient norm for the L-block case. 
The cost converges under moderate noise and reaches lower final values for smaller noise levels. 
The gradient norm initially decreases and then oscillates due to the last two terms in the shape gradient of $Y^k$ (see~\eqref{eq:shape_gradient_of_Y}), but converges after additional iterations.

Overall, these results confirm the robustness and effectiveness of the ADMM-based approach for shape optimization, in line with previous studies~\cite{CherratAfraitesRabago2025b,RabagoHadriAfraitesHendyZaky2024}.

%
%
%
%
\begin{figure}[h!]
\centering
\resizebox{0.275\linewidth}{!}{\includegraphics{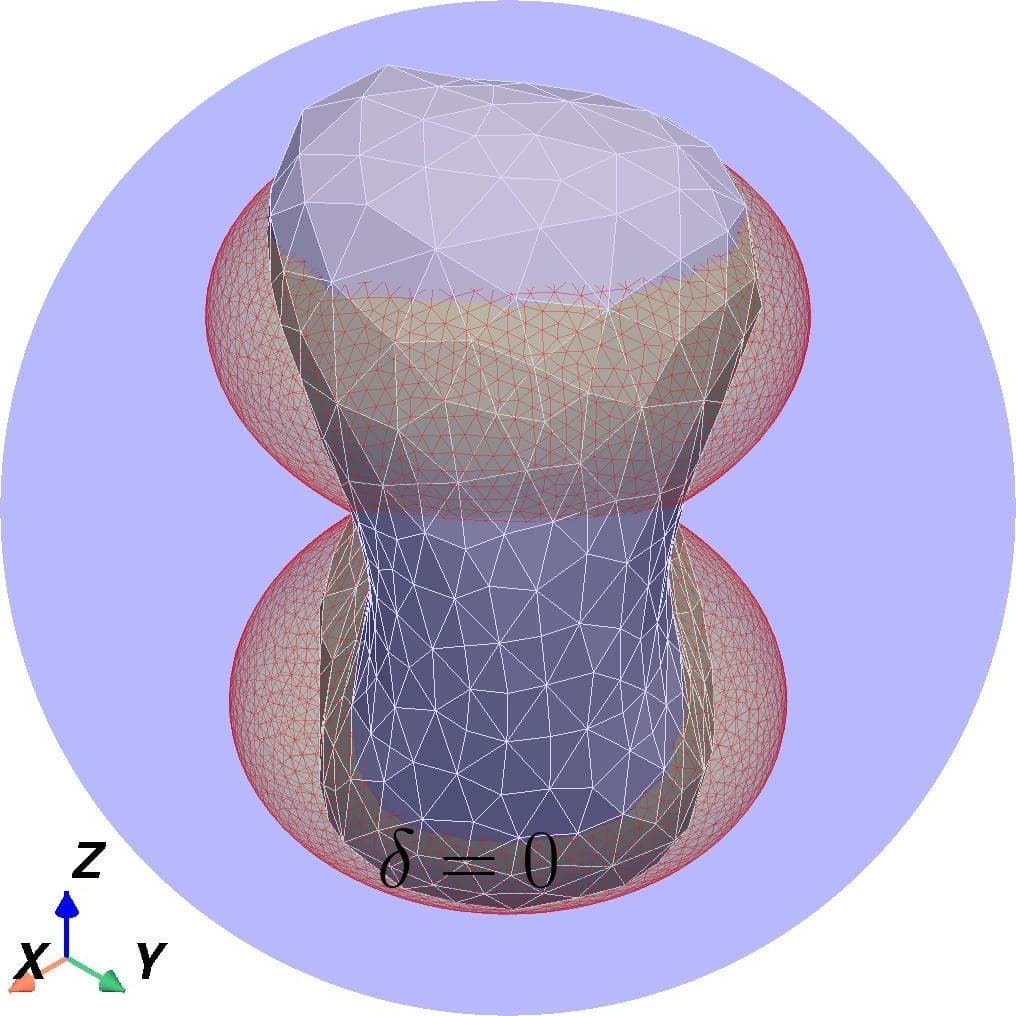}}\quad
\resizebox{0.275\linewidth}{!}{\includegraphics{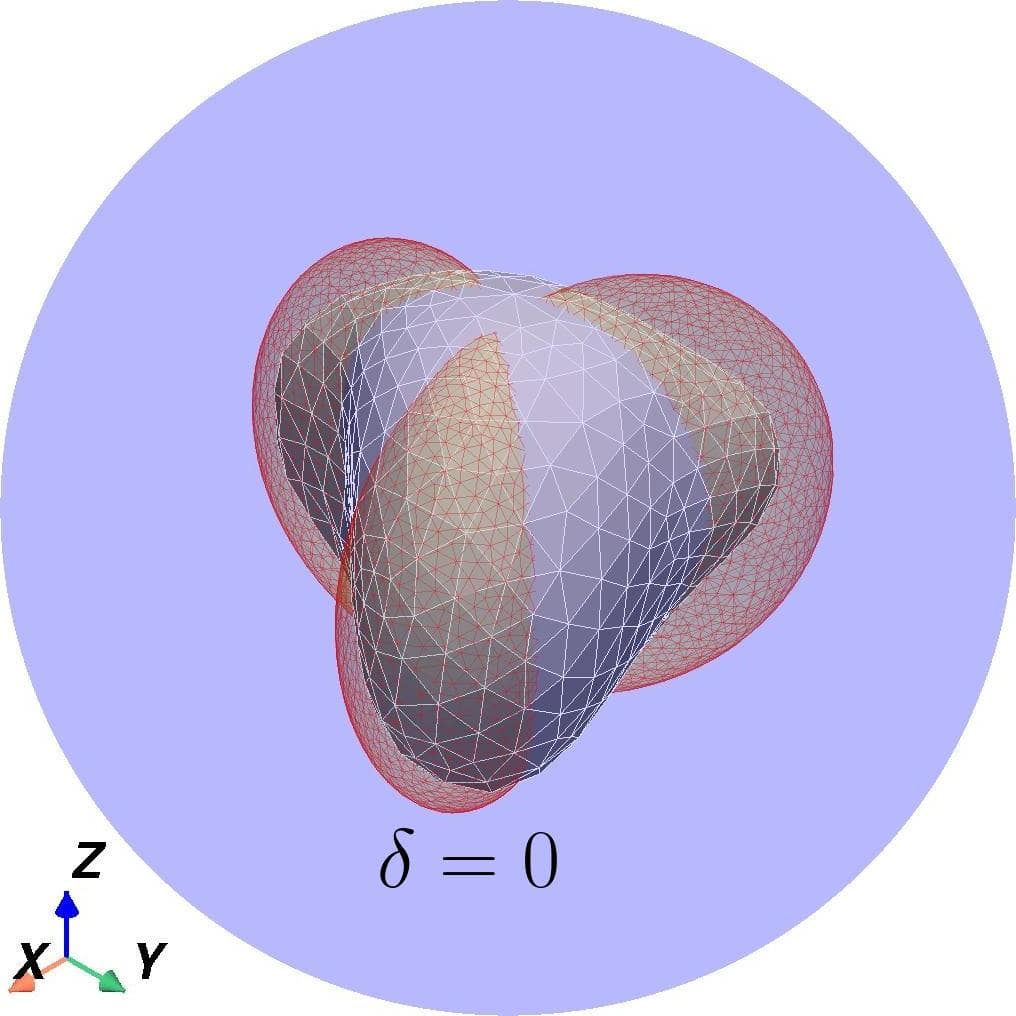}}\quad
\resizebox{0.275\linewidth}{!}{\includegraphics{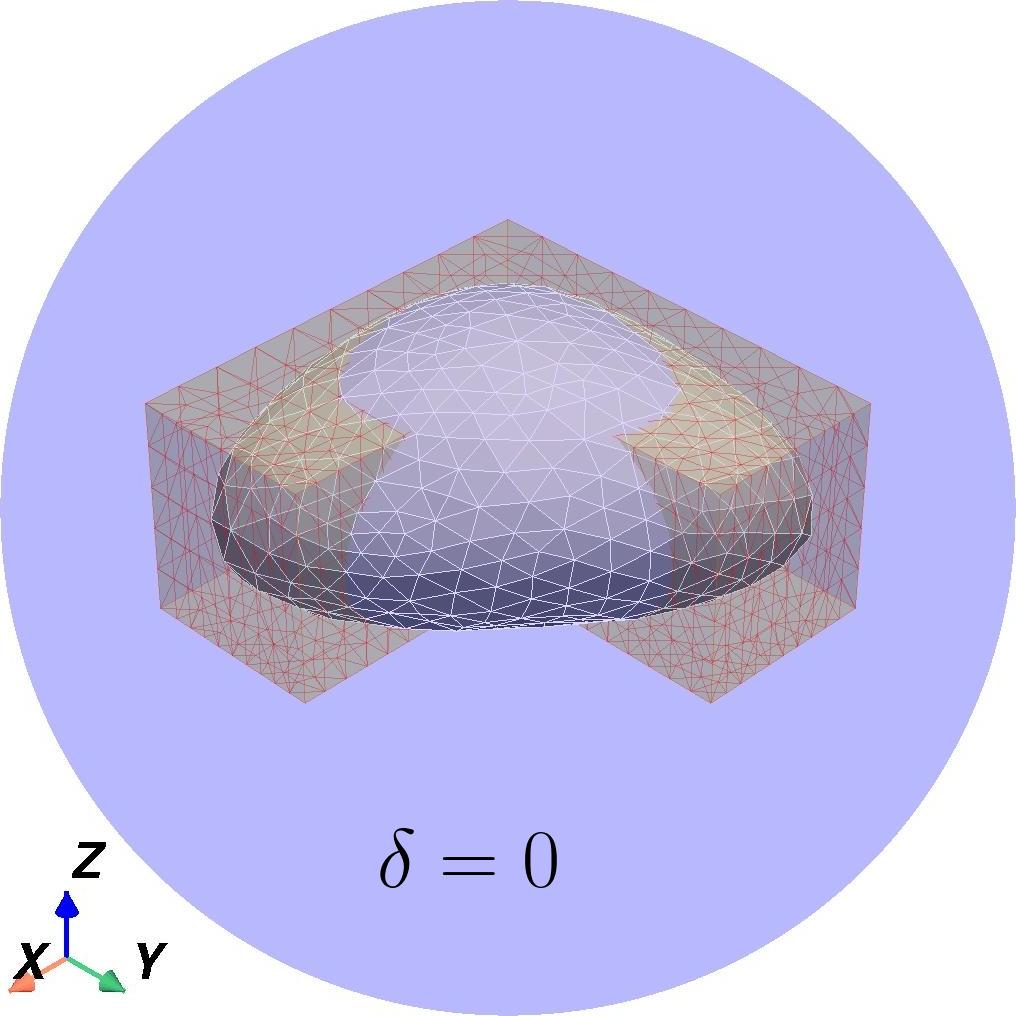}}\\[0.5em]  
\resizebox{0.275\linewidth}{!}{\includegraphics{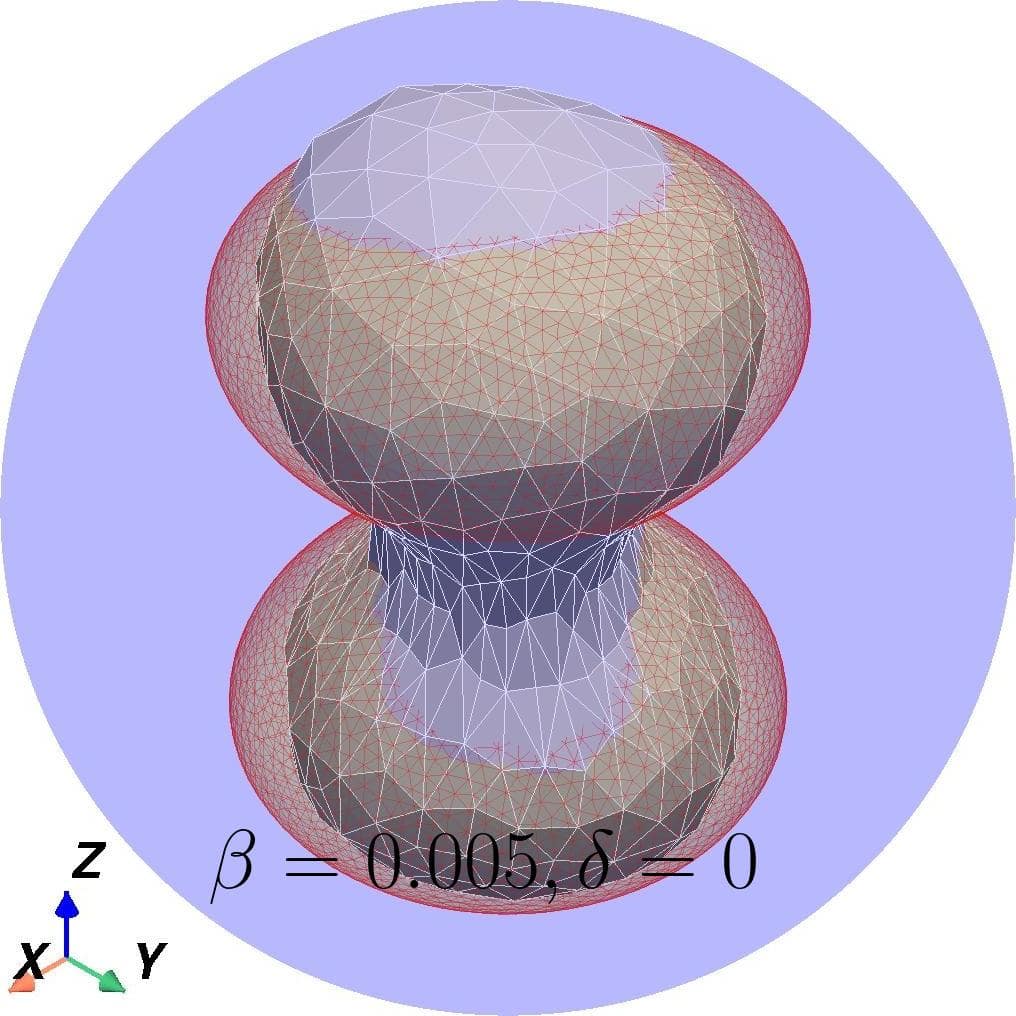}}\quad
\resizebox{0.275\linewidth}{!}{\includegraphics{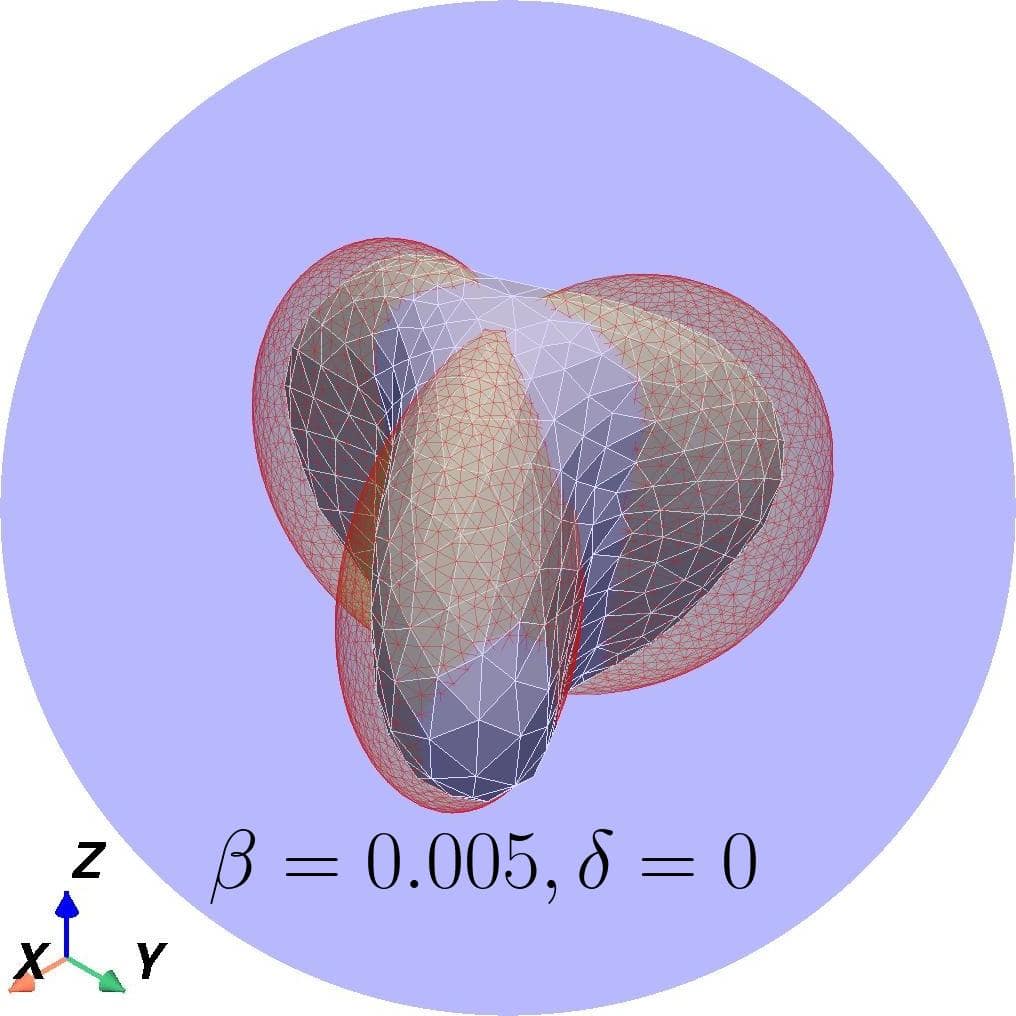}}\quad
\resizebox{0.275\linewidth}{!}{\includegraphics{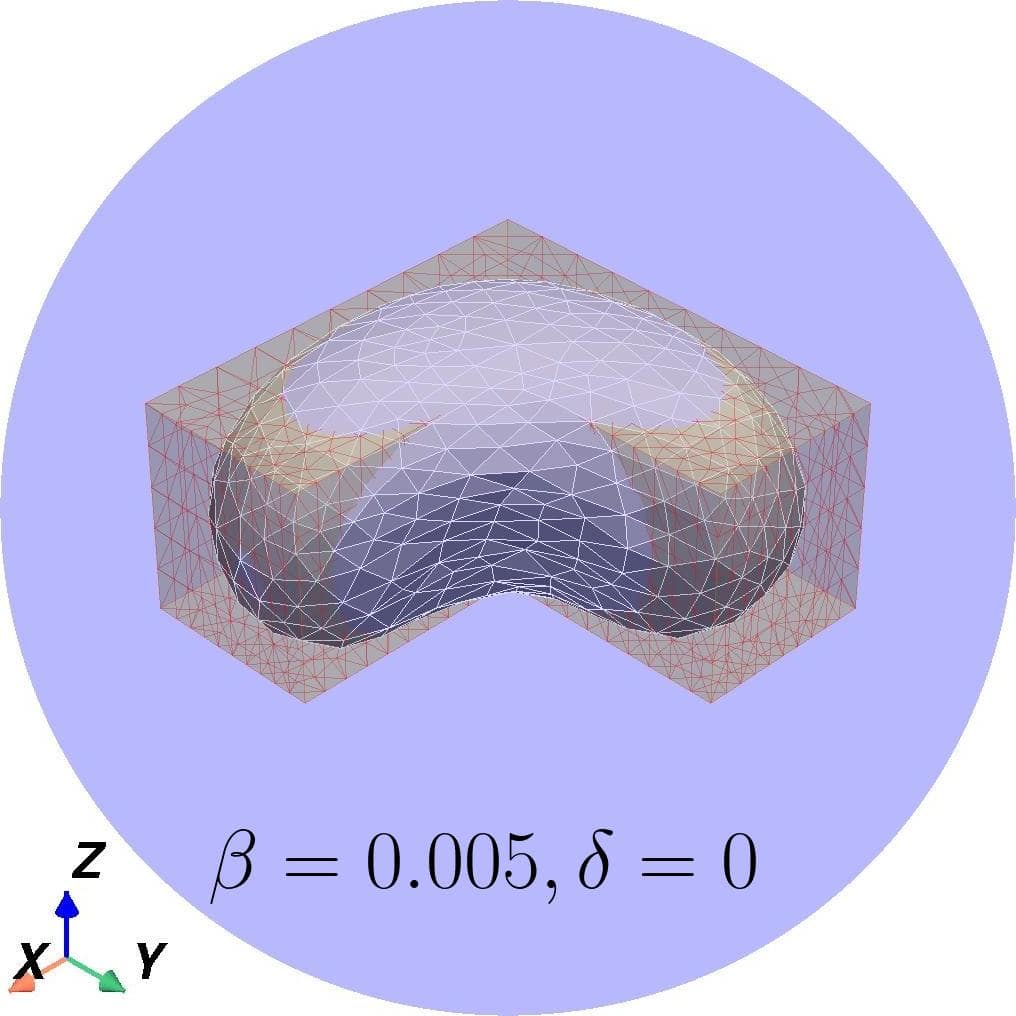}}
\caption{Reconstructed shapes using CCBM (top) and Algorithm~\ref{algo:CCBM--ADMM--SGBD} (bottom) under exact measurement $\delta = 0\%$}
\label{fig:exact_measurement}
\end{figure}
%
%
%
%
\begin{figure}[h!]
\centering
\resizebox{0.275\linewidth}{!}{\includegraphics{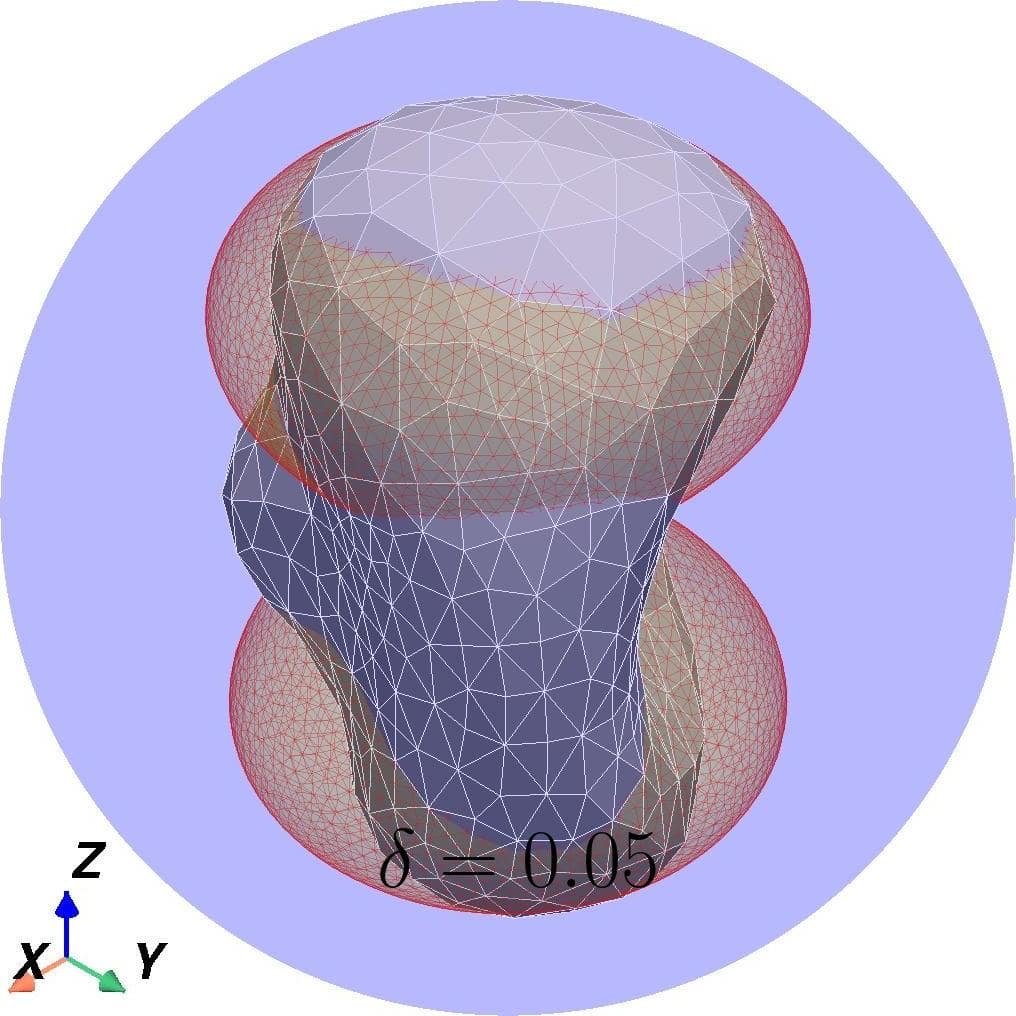}}\quad
\resizebox{0.275\linewidth}{!}{\includegraphics{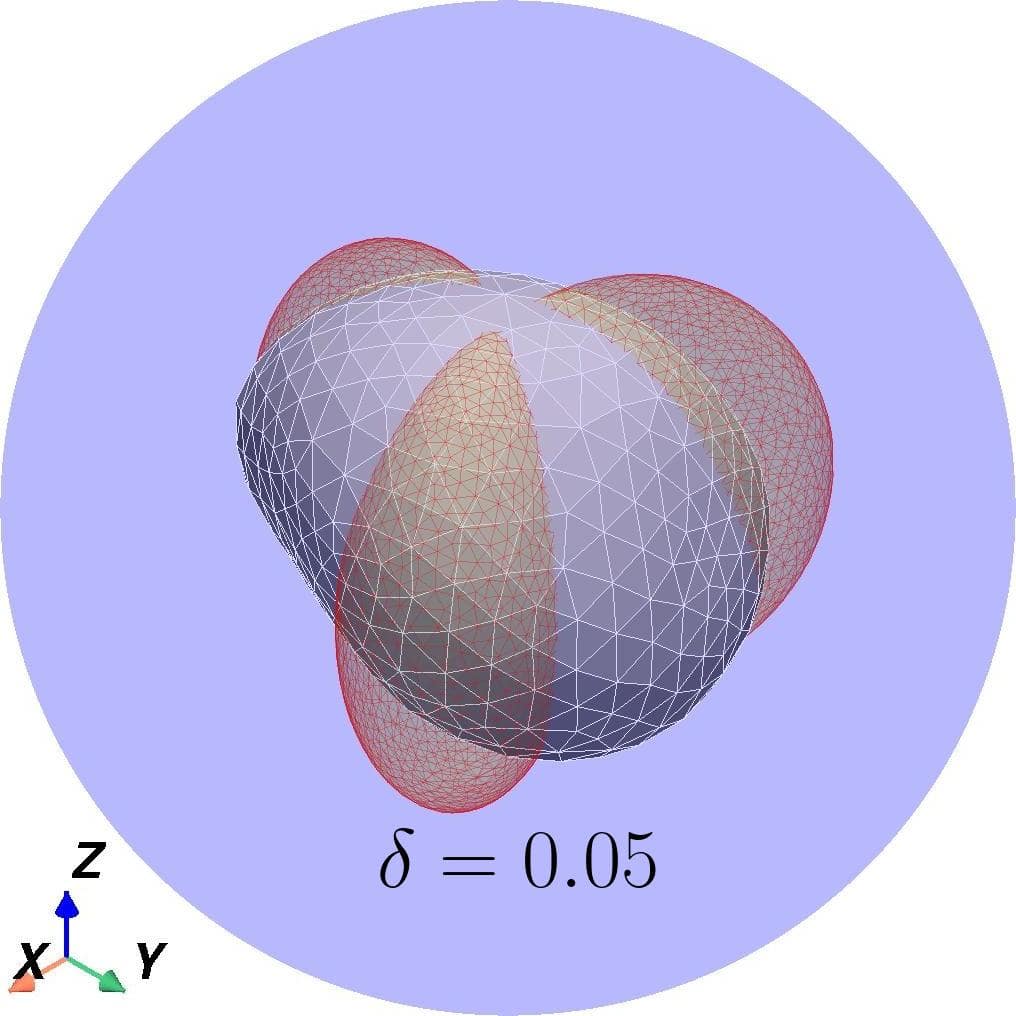}}\quad
\resizebox{0.275\linewidth}{!}{\includegraphics{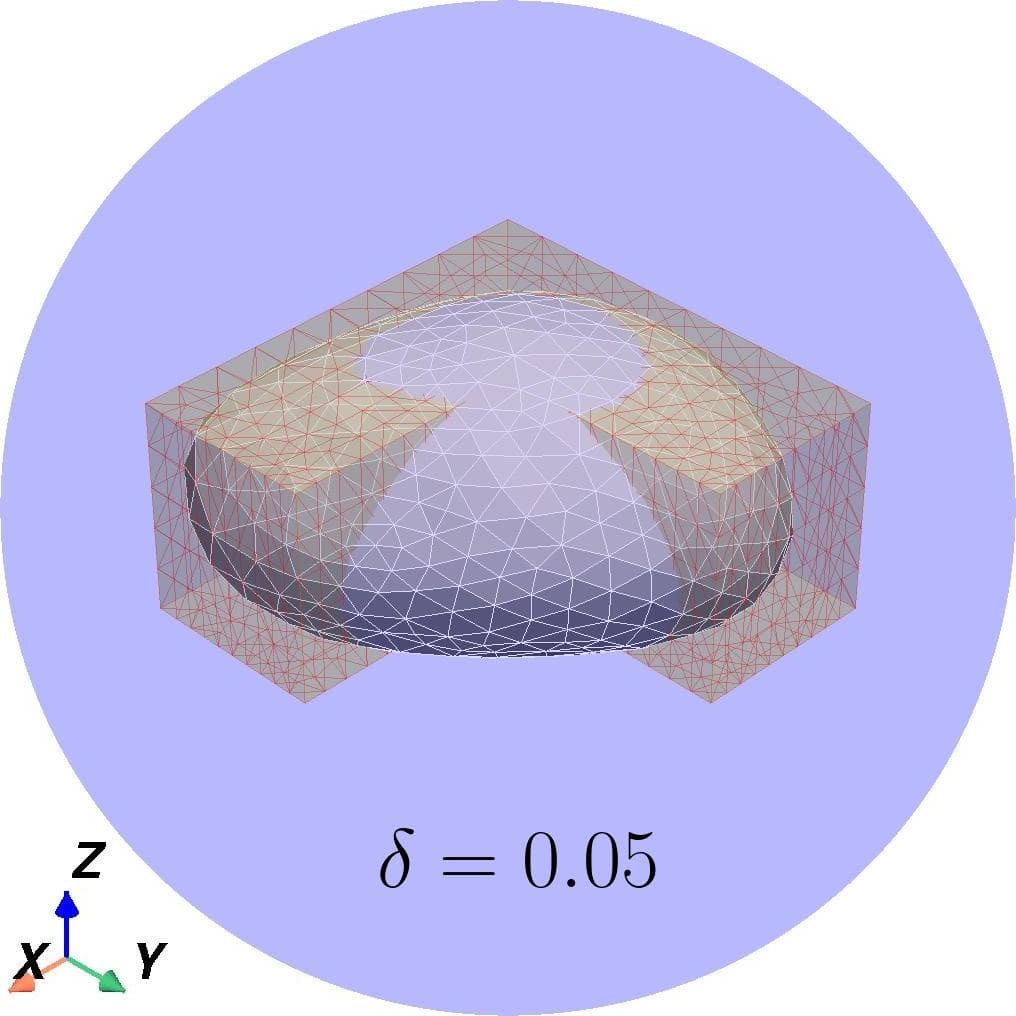}}\\[0.5em]  
\resizebox{0.275\linewidth}{!}{\includegraphics{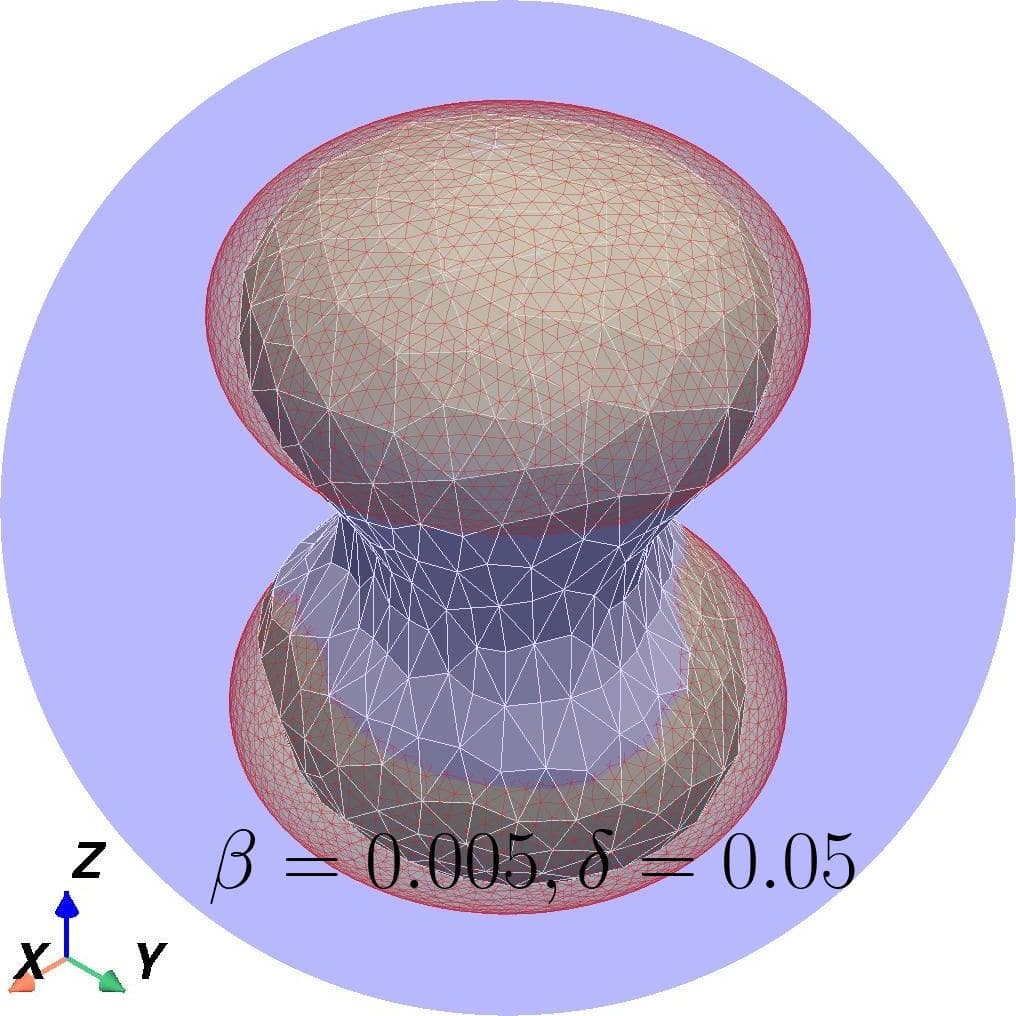}}\quad
\resizebox{0.275\linewidth}{!}{\includegraphics{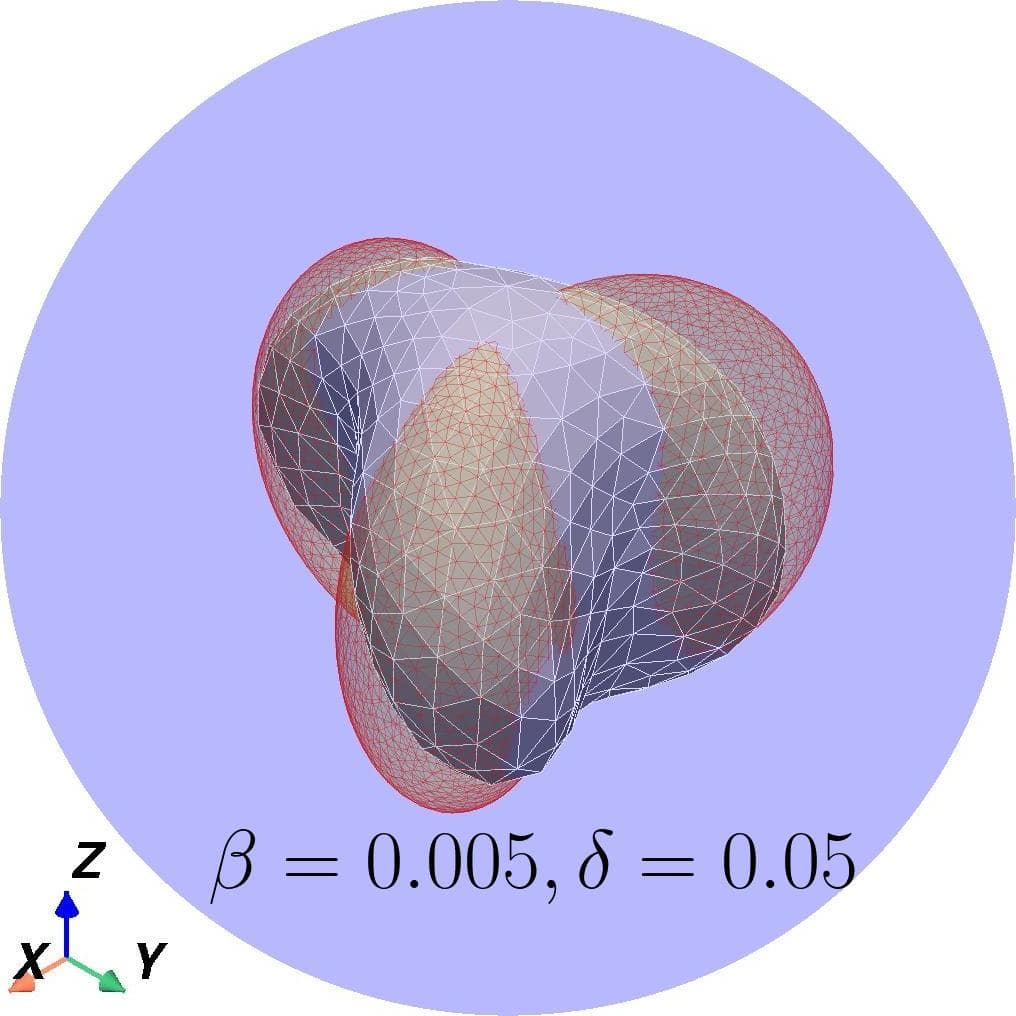}}\quad
\resizebox{0.275\linewidth}{!}{\includegraphics{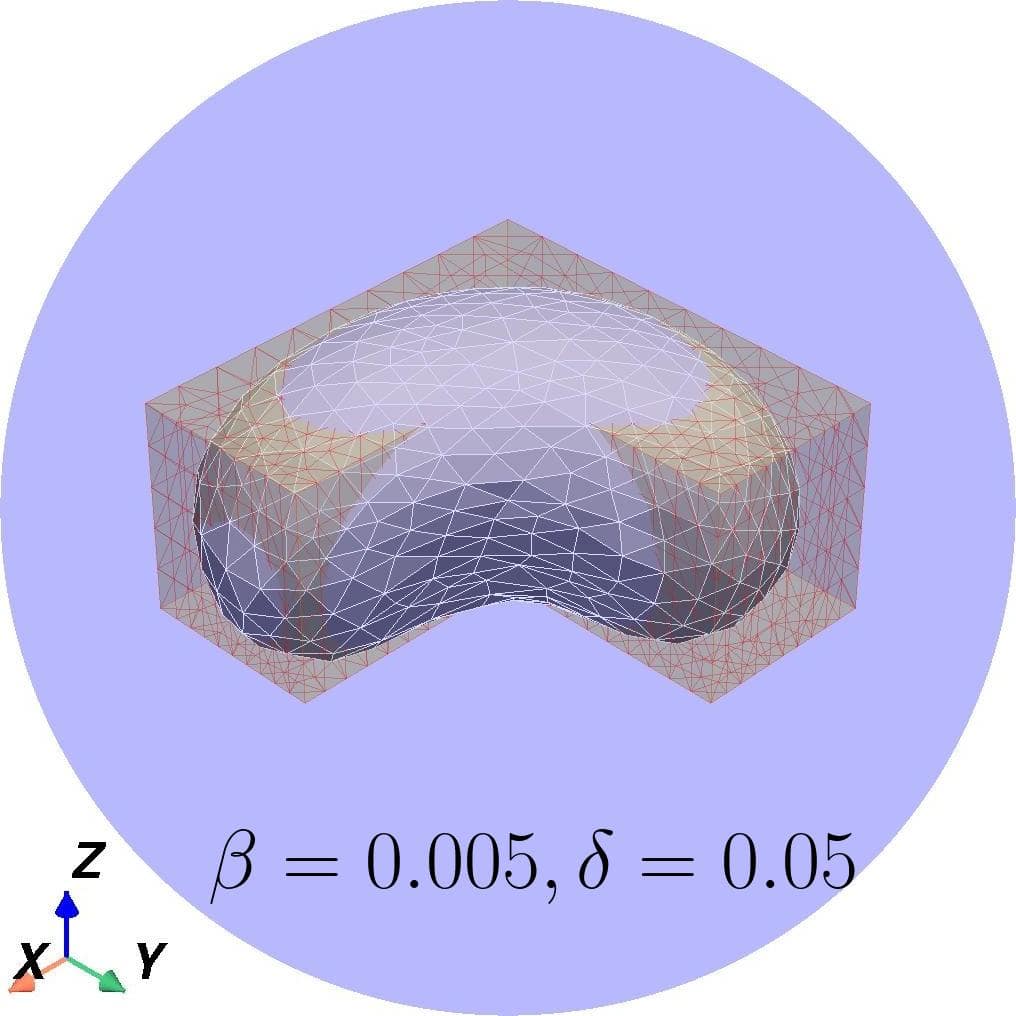}}
\caption{Reconstructed shapes using CCBM (top) and Algorithm~\ref{algo:CCBM--ADMM--SGBD} (bottom) under noisy measurement $\delta = 5\%$}
\label{fig:noisy_measurement}
\end{figure}
\begin{figure}[h!]
\centering 
\resizebox{0.24\linewidth}{!}{\includegraphics{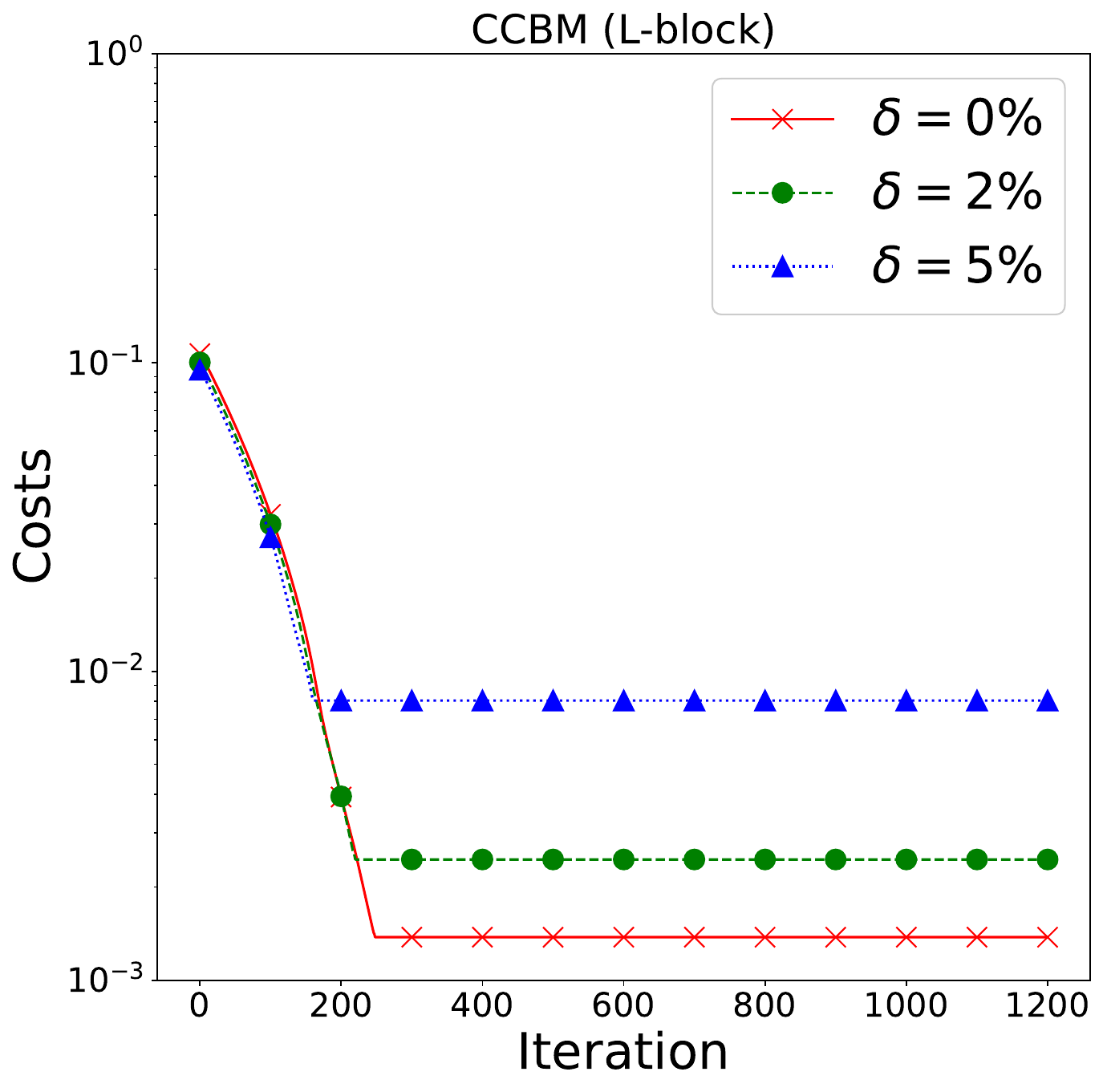}}\
\resizebox{0.24\linewidth}{!}{\includegraphics{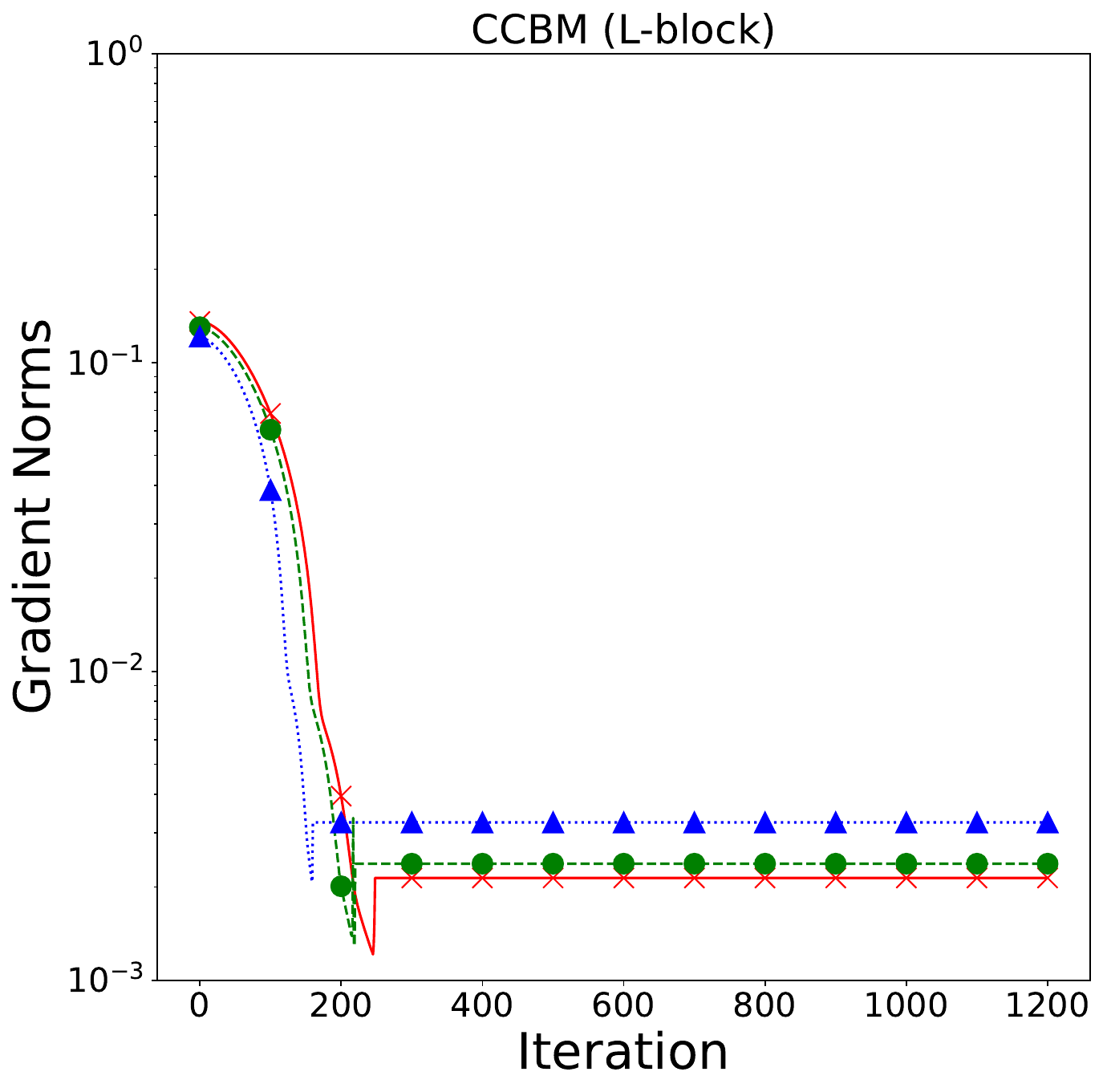}}\quad
\resizebox{0.24\linewidth}{!}{\includegraphics{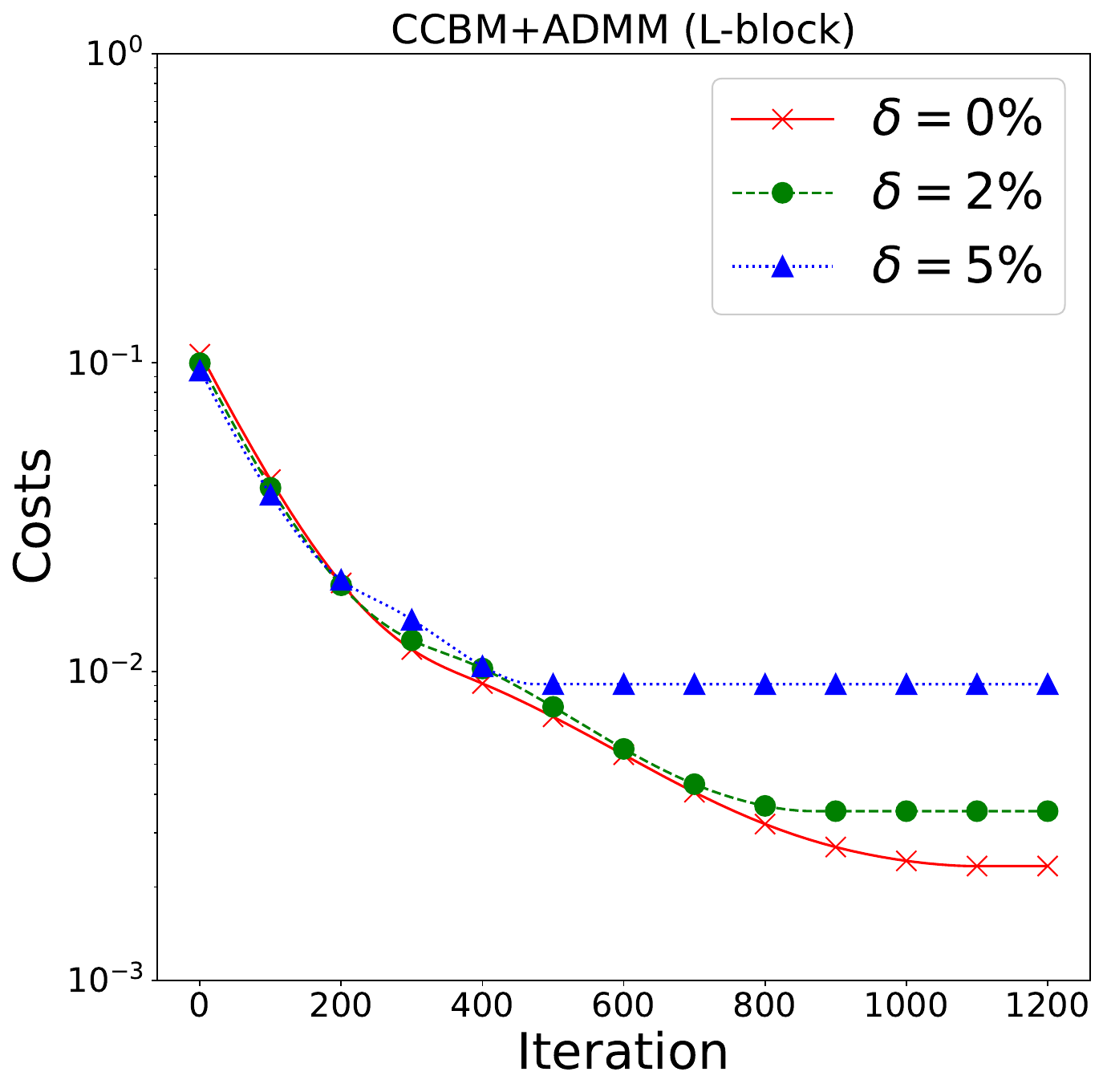}}\
\resizebox{0.24\linewidth}{!}{\includegraphics{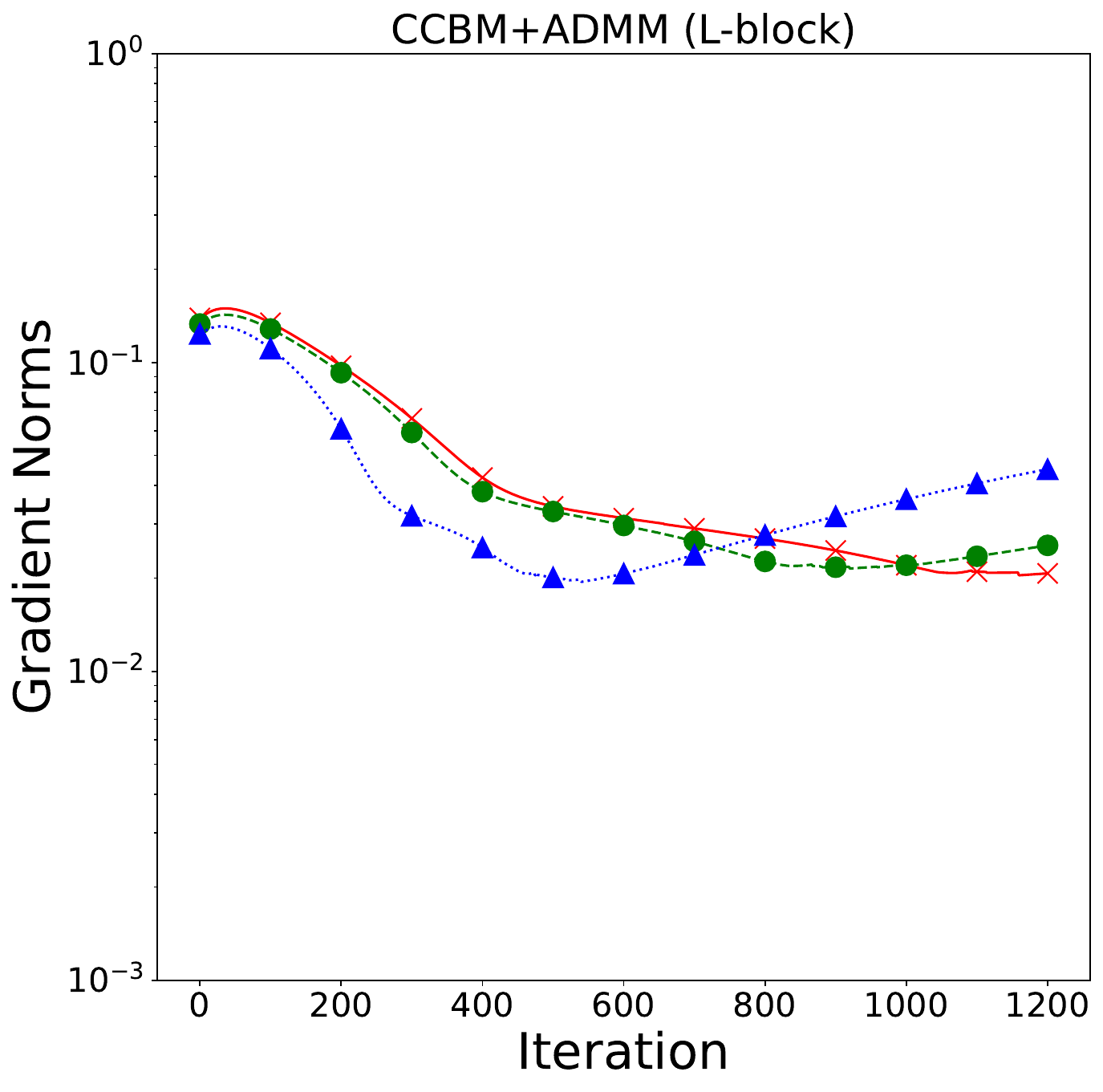}}

\caption{Histories of costs and gradient norms for the L-block case; left, conventional CCBM; right, proposed CCBM--ADMM scheme}
\label{fig:cost_and_gradient}
\end{figure}

\section{Conclusion}\label{sec:conclusion}
We have proposed a shape optimization method for inverse advection--diffusion problems that embeds the coupled complex boundary method (CCBM) within an ADMM-inspired augmented Lagrangian framework. 
The use of a Sobolev-gradient descent introduces a natural smoothing effect, eliminating the need for explicit perimeter or surface-area regularization commonly employed to stabilize reconstructions from noisy data.

Two- and three-dimensional numerical experiments demonstrate accurate recovery of complex inclusion geometries, with clear improvements for concave and geometrically intricate features where conventional approaches typically fail. 
The combination of a tailored adjoint formulation and partial gradient updates further enhances reconstruction quality and yields consistently lower cost values, even when the smoothness assumptions underlying the analysis are only partially satisfied.

Overall, the results confirm the effectiveness of the proposed CCBM--ADMM framework and support recent findings~\cite{CherratAfraitesRabago2025b,RabagoHadriAfraitesHendyZaky2024} on the suitability of ADMM-based methods for noisy and geometrically challenging inverse shape problems.


\bigskip

\textit{Acknowledgements}
We would like to express our sincere gratitude to the two reviewers for their valuable feedback and insightful suggestions, which greatly contributed to improving the quality of this work.
The work of JFTR is supported by the JSPS Postdoctoral Fellowships for Research in Japan (Grant Number JP24KF0221), and partially by the JSPS Grant-in-Aid for Early-Career Scientists (Grant Number JP23K13012) and the JST CREST (Grant Number JPMJCR2014).

\bibliographystyle{alpha}
\bibliography{main}

\end{document}